\documentclass[11pt,reqno]{amsproc}
\usepackage{color}
\usepackage{psfrag}
\usepackage{fullpage}
\usepackage{epstopdf}
\usepackage{graphicx}
\usepackage{hyperref}
\usepackage{subfigure}
\usepackage{algorithmic}
\usepackage{algorithm}
\usepackage{enumerate}
\newtheorem{theorem}{Theorem}[section]
\newtheorem{remark}[theorem]{Remark}

\title{A numerical methodology for enforcing maximum principles and 
the non-negative constraint for transient diffusion equations}
\author{$\mathrm{K.~B.~Nakshatrala}^{\#}$ 
\footnote{\noindent \#Corresponding author:~\emph{\textbf{e-mail:}} 
knakshatrala@uh.edu, \emph{\textbf{phone:}}+1-713-743-4418 \\
$^*$Graduate student} \\
{\tiny Department of Civil and Environmental Engineering, 
University of Houston, Houston, Texas 77204-4003.} \\ 
$\mathrm{H.~Nagarajan}^{*}$ \\
{\tiny Department of Mechanical Engineering, Texas A\&M 
University, College Station, Texas 77843-3193.} \\
$\mathrm{M.~Shabouei}^{*}$ \\
{\tiny Department of Civil and Environmental Engineering, 
University of Houston, Houston, Texas 77204-4003.} 
}

\date{\today}

\begin{document}

\begin{abstract}
  Transient diffusion equations arise in many branches of engineering 
  and applied sciences (e.g., heat transfer and mass transfer), and are 
  parabolic partial differential equations. It is well-known that, under 
  certain assumptions on the input data, these equations satisfy 
  important mathematical properties like maximum principles and 
  the non-negative constraint, which have implications in mathematical 
  modeling. However, existing numerical formulations for these types 
  of equations do not, in general, satisfy maximum principles and the 
  non-negative constraint. 
  In this paper, we present a methodology for enforcing maximum 
  principles and the non-negative constraint for \emph{transient} 
  anisotropic diffusion equation. The method of horizontal lines 
  (also known as the Rothe method) is applied in which the time 
  is discretized first. This results in solving steady anisotropic 
  diffusion equation with decay equation at every discrete time 
  level. 
  The proposed methodology can handle general computational grids 
  with no additional restrictions on the time step. We illustrate the 
  performance and accuracy of the proposed formulation using 
  representative numerical 
  examples. We also perform numerical convergence of the proposed 
  methodology. For comparison, we also present the results from 
  the standard single-field semi-discrete formulation and the results 
  from a popular software package, which all will violate maximum 
  principles and the non-negative constraint.  
\end{abstract}

\maketitle 


\section{INTRODUCTION AND MOTIVATION}
\label{Sec:TransientDMP_Introduction}
Certain quantities (e.g., concentration of a chemical 
species and absolute temperature) naturally attain 
non-negative values. A violation of the non-negative 
constraint for these quantities will imply violation 
of some basic tenets of Physics\footnote{There are 
some systems for which negative temperatures are 
possible (see Kittel and Kroemer \cite{Kittel_Kroemer}). 
Such cases are beyond the scope of this paper.}. It is, 
therefore, imperative that such physical constraints 
are met by mathematical models and by their associated 
numerical formulations. 
In this paper, we shall focus on two popular transient 
mathematical models, in which physical restrictions 
like the non-negative constraint play a central role. 
The first model is based on Fick's assumption (commonly 
referred to as Fick's law) and balance of mass. Fick's 
assumption is a simple constitutive model to describe 
the diffusion of a chemical species in which the flux 
is proportional to the negative gradient of the 
concentration. The second model is based on Fourier's 
assumption and balance of energy, which describes heat 
conduction in a rigid conductor. Both these constitutive 
models combined with their corresponding balance laws 
give rise to transient diffusion-type equations, 
which are parabolic partial differential equations. 

There has been tremendous progress in Applied Mathematics 
for these type of equations with respect to existence and 
uniqueness results, qualitative behavior of solutions, 
estimates, and other mathematical properties \cite{Pao,
  Evans_PDE}. In particular, it has been shown that 
transient diffusion-type equations satisfy the so-called 
maximum principles \cite{Pao}. It will be shown in a 
subsequent section that the non-negative constraint 
can be shown as a consequence of maximum principles 
under certain assumptions. 
Analytical solutions to several problems have been 
documented in various monographs (e.g., see references 
\cite{Carslaw_Jaeger,Ozisik}). However, it should be 
noted that most of these solutions are for isotropic 
and homogeneous media, and for simple geometries. For 
problems involving anisotropic and heterogeneous media, 
and complex geometries; finding analytical solutions 
is not possible, and one has to resort to numerical 
solutions. 
Obtaining physically meaningful numerical solutions for 
transient diffusion equation that satisfy maximum principles 
and the non-negative constraint is the main aim of this paper. 
It is well-known (and will be discussed in subsequent sections) 
that many popular numerical schemes (including the ones that 
are based on the finite element method) do not satisfy maximum 
principles and the non-negative constraint. Even for isotropic 
diffusion, stringent restrictions on the time step and the 
computational mesh are necessary to meet these important 
mathematical properties. 

The usual approach of solving linear second-order parabolic 
partial differential equations under the finite element method 
is to employ Galerkin formalism for spatial discretization. 
Several theoretical results (which include convergence 
proofs, a-priori estimates) for this approach can be 
found in the literature (e.g., see Reference 
\cite{Douglas_Dupont_SIAMJNA_1970_v07_p575}). But it 
has been adequately documented in the literature that 
this approach will not satisfy maximum principles and 
the non-negative constraint (for example, see Reference 
\cite{Harari_CMAME_2004_v193_p1491}, and also the 
discussion in Appendix). Thus, there is a need to 
develop new methodologies that will satisfy important 
mathematical properties like maximum principles and 
the non-negative constraint, and thereby improve the 
overall predictive capabilities of current numerical 
schemes. 

\subsection{Maximum principles for diffusion-type 
  equations in numerical setting}
The first study on maximum principles in the context of 
finite elements can be traced back to the seminal paper 
by Ciarlet and Raviart \cite{Ciarlet_Raviart_CMAME_1973_v2_p17}, 
which considered steady-state isotropic diffusion, low-order 
approximation, and simplicial elements. 
Ciarlet and Raviart points out that the single-field 
formulation (which is based on the Galerkin formalism) 
does not converge uniformly for isotropic diffusion 
equation unless some restrictions are placed on the 
mesh. In particular, they show that a sufficient 
condition for a three-node triangular element to 
converge uniformly and to meet maximum principles 
is that the triangle has to be acute (i.e., all the 
angles are less than 90-degrees). But this sufficient 
condition is valid is only for isotropic diffusion, 
and steady-state.
Since then, several other studies have addressed maximum 
principles for steady-state diffusion-type equations. A 
more detailed account of various works can be found in 
references \cite{Nakshatrala_Valocchi_JCP_2009_v228_p6726,
Nagarajan_Nakshatrala_IJNMF_2011_v67_p820,
Payette_Nakshatrala_Reddy_IJNME_2012}. Although these papers 
have considered steady-state diffusion equation, the discussion 
in these papers is applicable to transient diffusion equations. 
A brief summary of these three papers is as follows. 
In Reference \cite{Nakshatrala_Valocchi_JCP_2009_v228_p6726}, 
a non-negative methodology for mixed finite element formulation 
has been proposed for steady-state diffusion equation using 
techniques from convex quadratic programming. The paper also 
studied the effect of the non-negative methodology on the 
element local mass balance. 
In Reference \cite{Nagarajan_Nakshatrala_IJNMF_2011_v67_p820}, 
a methodology has been proposed for steady-state diffusion equation 
with decay that satisfies maximum principles and the non-negative 
constraint on general computational grids. (Note that the maximum 
principle for diffusion with decay is slightly different from the 
maximum principle with out decay.) This methodology will be 
utilized later in the present paper. 
In Reference \cite{Payette_Nakshatrala_Reddy_IJNME_2012}, 
a systematic study on the effect of high-order approximation 
on the violation of maximum principles and the non-negative 
constraint. In particular, it has been shown using numerical 
simulations that the violation of the non-negative constraint 
does not decrease with $p$-refinement. 
Some representative works in other areas of discretization 
to obtain monotone solutions include finite volume methods 
\cite{Lipnikov_Shashkov_Svyatskiy_Vassilevski_JCP_2007_v227_p492, 
Lipnikov_Svyatskiy_Vassilevski_JCP_2009_v228_p703,
Lipnikov_Svyatskiy_Vassilevski_JCP_2010_v229_p4017}, 
and mimetic finite difference methods 
\cite{Lipnikov_Manzini_Svyatskiy_JCP_2011_v230_p2620}.

\subsubsection{Maximum principles for transient systems}
Transient diffusion equations fall in the realm of 
parabolic partial differential equations (PDEs), 
whereas steady-state diffusion equations are 
elliptic PDEs. A noticeable difference in maximum 
principles for parabolic PDEs and the corresponding 
ones for elliptic PDEs is that, in the case of a 
parabolic PDE, the maximum can occur either on the 
boundary of the domain or in the initial conditions. 
On the other hand, for a second-order elliptic PDE, 
the classical maximum principle says that the maximum 
occurs on the boundary of the domain (under some 
appropriate conditions on the input data and 
domain). A more precise mathematical treatment in 
Section \ref{Sec:TransientDMP_Governing_equations}.

Several papers have also addressed maximum principles 
for transient systems (i.e., parabolic problems) in 
numerical setting. Herrera and Valocchi 
\cite{Herrera_Valocchi_GW_2006_v44_p803} have employed 
flow-oriented derivatives with backward Euler to obtain 
non-negative solutions in the context of finite difference 
and finite volume methods. One method that is commonly 
employed in the area of subsurface hydrology is by Chen 
and Thomee \cite{Chen_Thomee_JAMS_1985_v26_p329}. This 
method is based on the standard single-field formulation 
but employs lumped capacity matrix. 
(By the standard single-field formulation we refer to the 
formulation obtained by employing the semi-discrete approach 
using method of vertical lines at integral time steps, and 
Galerkin formalism for spatial discretization. See Appendix 
for more details of this formulation.) 
It is noteworthy that lumping capacity matrix approach 
is commonly considered as a variational crime \cite{Hughes}. 
More importantly, lumped capacity matrix is not sufficient 
to meet maximum principles and the non-negative constraint 
for anisotropic diffusion even if one employs the backward 
Euler time stepping scheme (e.g., see subsection 
\ref{Subsec:TransientDMP_plate_with_hole} and Appendix). 
Reference \cite{Berzins_CNME_2001_p659} also alters the 
capacity matrix to preserve positivity for parabolic 
problems but restricts to isotropic diffusion. Other 
notable works are 
\cite{Rank_Katz_Werner_IJNME_1983_v19_p1771, 
  Porru_Serra_JAMS_1994_v56_p41,
  Farago_Horvath_Korotov_ANM_2005_v53_p249,
  Elshebli_AMM_2008_v32_p1530}, which all focused on 
getting restrictions on the mesh (and in some cases 
on the time step) to meet maximum principles. More 
importantly, they did not consider anisotropy, and 
such restrictions are not possible for anisotropic 
and heterogeneous medium. 

There are several papers that considered consistent 
capacity matrices, but derived restrictions on the 
time step to satisfy maximum principles 
\cite{Mizukami_CMAME_1986_v59_p101,Thomas_Zhou_CNME_1998_p809,
  Ilinca_Hetu_CMAME_2002_v191_p3073,Harari_CMAME_2004_v193_p1491,
  Horvath_IJCMA_2008_v55_p2306}. 
A striking difference between the time step restrictions 
with respect to numerical stability and maximum principles 
is that numerical stability places an upper bound on the 
selection of the time step whereas maximum principles 
place a lower bound on the selection of the time step. 
The time step is selected based on the following 
inequality:
\begin{align}
  0 < \Delta t_{\mathrm{crit}}^{\mathrm{MP}} \leq \Delta t 
  \leq \Delta t_{\mathrm{crit}}^{\mathrm{stability}} 
\end{align}
where $\Delta t_{\mathrm{crit}}^{\mathrm{stability}}$ is the critical time 
step to obtain stable results, and $\Delta t_{\mathrm{crit}}^{\mathrm{MP}}$ 
is the critical time step to satisfy maximum principles. It should be 
however mentioned that these works on deriving time step restrictions 
have considered one-dimensional problems or isotropic media, and these 
conditions are not applicable otherwise. To the best of our knowledge, 
none of the prior works presented a methodology for transient anisotropic 
diffusion equations to satisfy maximum principles and the non-negative 
constraint on general computational grids with no further restrictions 
on the time step.

\subsection{Our approach and main contributions of this paper}
In this paper, we shall employ the Rothe method (or the method of 
horizontal lines) \cite{Rothe_MA_1930_v102_p650} to solve transient 
anisotropic diffusion equation. There are several papers in the literature 
that have employed Rothe method to solve parabolic equations 
\cite{Harari_CMAME_2004_v193_p1491, Bornemann_ICSE_1990_v02_p279,
Lang_Walter_ANM_1993_v13_p135,Chapko_Kress_JIEA_1997_v09_p47}. 
These papers, except for Reference \cite{Harari_CMAME_2004_v193_p1491}, 
did not apply the Rothe method in the context of maximum principles. 
Although Reference \cite{Harari_CMAME_2004_v193_p1491} 
addressed maximum principles by using the Rothe method, but 
the formulation is restricted to isotropic diffusion. In addition, 
Reference \cite{Harari_CMAME_2004_v193_p1491} employed techniques 
from stabilized methods, which is different from the approach taken 
in this paper. In the proposed formulation, the temporal discretization 
using the Rothe method will give rise to inhomogeneous elliptic partial 
differential equation, which is solved using the approach presented in 
our earlier paper \cite{Nagarajan_Nakshatrala_IJNMF_2011_v67_p820}. An 
attractive aspect of the proposed methodology is that there are no 
additional restrictions on the time step to meet maximum principles.

%







\subsection{An outline and notation used in this paper}
The remainder of this paper is organized as follows. In Section 
\ref{Sec:TransientDMP_Governing_equations}, we present governing 
equations for transient anisotropic diffusion, and discuss 
maximum principles and the non-negative constraint. In Section 
\ref{Sec:TransientDMP_Derivation}, we derive a methodology for 
enforcing maximum principles and the non-negative constraint 
for transient anisotropic diffusion equation using the method 
of horizontal lines. In Section \ref{Sec:TransientDMP_NR}, we 
illustrate the performance of the proposed formulation using 
representative numerical examples. Finally, conclusions are 
drawn in Section \ref{Sec:TransientDMP_Conclusions} with a 
discussion on plausible future works on enforcing maximum 
principles.

The symbolic notation adopted in this paper is as follows. 
Repeated indices do not imply summation. (That is, we do 
not employ Einstien's summation convention.) We shall 
employ the standard notation for open, closed and half-open 
intervals \cite{Bartle_Sherbert}: 
\begin{align}
    (a, b) := \{\mathrm{x} \in \mathbb{R} \; \big| \; a < \mathrm{x} < b \}, \; 
    [a, b] := \{\mathrm{x} \in \mathbb{R} \; \big| \; a \leq \mathrm{x} \leq b \}, \nonumber \\
    (a, b] := \{\mathrm{x} \in \mathbb{R} \; \big| \; a < \mathrm{x} \leq b \}, \; 
    [a, b) := \{\mathrm{x} \in \mathbb{R} \; \big| \; a \leq \mathrm{x} < b \}
\end{align}
Similar to our earlier paper \cite{Nagarajan_Nakshatrala_IJNMF_2011_v67_p820}, 
we shall make a distinction between vectors in the continuum and finite 
element settings. We also make a distinction between second-order 
tensors in the continuum setting versus matrices in the context of the 
finite element method. The continuum vectors are denoted by lower case 
boldface normal letters, and second-order tensors will be denoted by 
upper case boldface normal letters (for example, vector $\mathbf{x}$ 
and second-order tensor $\mathbf{D}$). In the finite element context, 
we shall denote the vectors using lower case boldface italic letters, 
and the matrices are denoted using upper case boldface italic letters. 
For example, vector $\boldsymbol{v}$ and matrix $\boldsymbol{K}$. Other 
notational conventions adopted in this paper are introduced as needed.

\section{GOVERNING EQUATIONS: TRANSIENT ANISOTROPIC DIFFUSION}
\label{Sec:TransientDMP_Governing_equations}
Let $\Omega \subset \mathbb{R}^{nd}$ be a bounded open set, where 
``$nd$'' denotes the number of spatial dimensions. The boundary 
is denoted by $\partial \Omega$, which is assumed to be piecewise 
smooth. A spatial point is denoted by $\mathbf{x} \in \overline{\Omega}$. 
The gradient and divergence with respect to $\mathbf{x}$ are denoted 
by $\mathrm{grad}[\cdot]$ and $\mathrm{div}[\cdot]$, respectively. 
Let $t \in [0, \mathcal{I}]$ denote the time, where $\mathcal{I} 
> 0$ denotes the length of the time interval. The concentration 
of an inert chemical species is denoted by $c(\mathbf{x},t)$. The 
(spatial) boundary is divided into two parts: $\Gamma^{\mathrm{D}}$ 
and $\Gamma^{\mathrm{N}}$ such that $\Gamma^{\mathrm{D}} \cup 
\Gamma^{\mathrm{N}} = \partial \Omega$ and $\Gamma^{\mathrm{D}} 
\cap \Gamma^{\mathrm{N}} = \emptyset$. 
$\Gamma^{\mathrm{D}}$ is that part of the boundary on which 
Dirichlet boundary condition (i.e., the concentration) 
is prescribed, and $\Gamma^{\mathrm{N}}$ is the part of the 
boundary on which Neumann boundary condition (i.e., the 
flux) is prescribed. The unit outward normal to the 
boundary is denoted by $\mathbf{\hat{n}}(\mathbf{x})$. 
The governing equations for transient anisotropic 
diffusion can be written as follows:
\begin{subequations}
  \label{Eqn:TransientDMP_Governing_equations}
  \begin{align}
  \label{Eqn:TransientDMP_PDE}
    &\frac{\partial c(\mathbf{x},t)}{\partial t} -\mathrm{div}
    [\mathbf{D}(\mathbf{x}) \mathrm{grad}[c(\mathbf{x},t)]] = 
    f(\mathbf{x},t) \quad \mathrm{in} \; \Omega \times (0, 
    \mathcal{I})  \\
    & c(\mathbf{x},t) = c_p(\mathbf{x},t) \quad \mathrm{on} \; 
    \Gamma^{\mathrm{D}} \times (0, \mathcal{I}) \\
   \label{Eqn:TransientDMP_Neumann}
    & \mathbf{\hat{n}}(\mathbf{x}) \cdot \mathbf{D}(\mathbf{x}) 
   \mathrm{grad}[c(\mathbf{x},t)] = q_p(\mathbf{x},t) \quad 
   \mathrm{on} \; \Gamma^{\mathrm{N}} \times (0, \mathcal{I}) \\
    \label{Eqn:TransientDMP_IC}
    & c(\mathbf{x},t=0) = c_0(\mathbf{x}) \quad \mathrm{in} \; \Omega
  \end{align}
\end{subequations}
where $\mathbf{D}(\mathbf{x})$ is the diffusivity tensor, 
$f(\mathbf{x},t)$ is the volumetric source/sink, $c_p(\mathbf{x},t)$ 
is the prescribed concentration on the boundary, $q_p(\mathbf{x},t)$ 
is the prescribed flux on the boundary, and $c_0(\mathbf{x})$ is the 
prescribed initial condition. The diffusivity tensor is symmetric, 
and is assumed to be bounded above and uniformly elliptic. That is, 
there exists two constants $\mathrm{0 < \xi_1 \leq \xi_2 < +\infty}$ 
such that 
\begin{align}
\label{Eqn:Helmholtz_positive_definiteness_D}
  \xi_1 \mathbf{y}^{\mathrm{T}} \mathbf{y}\leq \mathbf{y}^{\mathrm{T}} 
  \mathbf{D}(\mathbf{x}) \mathbf{y} \leq \xi_2 \mathbf{y}^{\mathrm{T}} 
  \mathbf{y} \quad \forall \mathbf{x} \in \Omega \; \mathrm{and} \; 
  \forall \mathbf{y} \in \mathbb{R}^{nd}
\end{align}
The above initial boundary value problem given by equations 
\eqref{Eqn:TransientDMP_PDE}--\eqref{Eqn:TransientDMP_IC} is 
a \emph{linear parabolic partial differential equation}. From 
the theory of partial differential equations, such equations 
are known to satisfy maximum principles under appropriate 
regularity assumptions on the input data and on the domain 
\cite{Protter_Weinberger,McOwen}. 

\begin{remark}
  It should be noted that a consequence of Fickian/Fourier 
  mathematical model is that a thermal/chemical disturbance 
  at a point will be felt at other points instantaneously. 
  This is because of the parabolic nature of the resulting 
  partial differential equations. To put it differently, 
  these mathematical models predict that the information 
  travels at infinite speed, which is against the current 
  accepted laws of Physics. Several modifications have 
  been suggested in the area of heat conduction to have 
  finite speeds for thermal disturbances, and most of 
  these models are hyperbolic partial differential 
  equations. Some notable works on this topic are by 
  Maxwell \cite{Maxwell_PTRSL_1866_vA157_p26}, Catteneo 
  \cite{Cattaneo_CR_1958_v247_p431}, and Gurtin and 
  Pipkin \cite{Gurtin_Pikin_ARMA_1968_v31_p252}.
  A more detailed discussion with respect to finite 
  speed thermoelasticity can be found in Reference 
  \cite{Ignaczak_Starzewski}. It is noteworthy that 
  hyperbolic partial differential equations do not 
  possess maximum principles ``similar'' to the ones 
  possessed by elliptic and parabolic partial 
  differential equations. This area of research 
  is far from settled, and is beyond the scope 
  of this paper. 
\end{remark}

\subsection{Maximum principles for parabolic equations}
\label{Subsec:TransientDMP_maximum_principle}
Maximum principles for parabolic partial differential equations 
can be traced back to Levi \cite{Levi_AMPA_1908_v14_p187} and 
Picone \cite{Picone_AMPA_1929_v7_p145}. A brief history and 
other references on maximum principles for parabolic partial 
differential equations can be found in the book by Protter 
and Weinberger \cite{Protter_Weinberger}. Herein, we shall 
employ an approach similar to that of Nirenberg 
\cite{Nirenberg_CPAM_1953_v6_p167}.
Before we state a maximum principle for linear parabolic partial 
differential equations, we shall introduce relevant notation and 
definitions. The parabolic cylinder is defined as $\Omega_{\mathcal{I}} 
:= \mathrm{\Omega} \times (0,\mathcal{I})$. The parabolic boundary 
is defined as follows: 
\begin{align}
  \Gamma_{\mathcal{I}} := \left\{(\mathbf{x},t) \in 
    \overline{\Omega}_{\mathcal{I}} \; \Big| \; \mathbf{x} 
    \in \partial \Omega \; \mathrm{or} \; t = 0\right\}
\end{align}
The parabolic cylinder and parabolic boundary are pictorially 
described in Figure \ref{Fig:TransientDMP_parabolic_plot}. 
\begin{figure}[t]
  \centering	
  \psfrag{R}{$\mathbb{R}$}
  \includegraphics[scale=1.0]{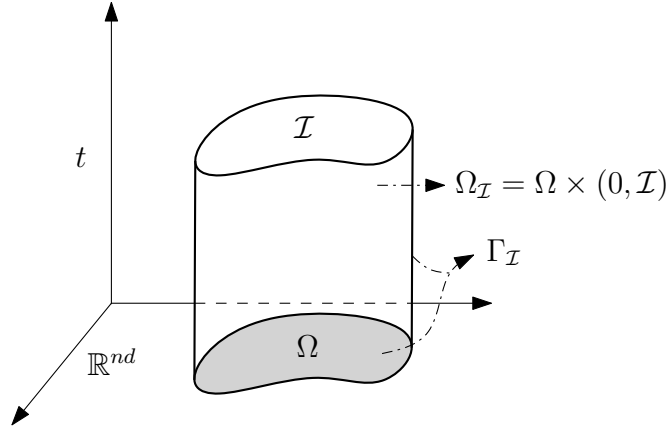}
  \caption{A pictorial description of parabolic cylinder $\Omega_{\mathcal{I}}$ 
    and parabolic boundary $\Gamma_{\mathcal{I}}$.} 
  \label{Fig:TransientDMP_parabolic_plot}
\end{figure}
Let $C^{m}(\Omega)$ denotes the set of functions defined on 
$\Omega$ that are continuously differentiable up to $m$-th 
order. We shall introduce the following function space with 
differing smoothness in the $\mathbf{x}$- and $t$-variables:
\begin{align}
  C_{1}^{2}(\Omega_{\mathcal{I}}) := \left\{c : \Omega_{\mathcal{I}} 
    \rightarrow \mathbb{R} \; | \; c, \frac{\partial c}{\partial x_i}, 
    \frac{\partial^2 c}{\partial x_i \partial x_j}, \frac{\partial c}{\partial t} 
    \in C(\Omega_{\mathcal{I}}); i, j = 1, \cdots, nd \right\}
\end{align}
\begin{theorem}[maximum principle]
  \label{Theorem:TransientDMP_MP_parabolic}
  Let $c(\mathbf{x},t)$ $\in C^{2}_{1}(\Omega_{\mathcal{I}}) \cap 
  C(\overline{\Omega}_{\mathcal{I}})$ satisfy $\partial c/\partial t - 
  \mathrm{div}[\mathbf{D}(\mathbf{x}) \mathrm{grad}[c]] \geq 0$ 
  in $\Omega_{\mathcal{I}}$. Then $c(\mathbf{x},t)$ achieves its 
  minimum on the parabolic boundary of $\Omega_{\mathcal{I}}$. 
  That is, 
  \begin{align}
    \label{Eqn:TransientDMP_Maximum_principle}
    \mathop{\mathrm{min}}_{(\mathbf{x},t) \in \overline{\Omega}_{\mathcal{I}}} 
    \; c(\mathbf{x},t) = \mathop{\mathrm{min}}_{(\mathbf{x},t) \in 
      \Gamma_{\mathcal{I}}} \; c(\mathbf{x},t)
  \end{align}
\end{theorem}
\begin{proof}
  A proof can be found in standard books on partial differential 
  equations (e.g., see \cite{Protter_Weinberger,McOwen,Evans_PDE}).
\end{proof}

\begin{remark}
  The above maximum principle implies that if one has volumetric 
  source everywhere and at all times (i.e., $f(\mathbf{x},t) 
  \geq 0$) then the minimum will occur on the boundary of the 
  domain or in the initial condition. A logically equivalent 
  statement of the above theorem can be written as follows: 
  If $c(\mathbf{x},t)$ satisfies $\partial c/\partial t - 
  \mathrm{div}[\mathbf{D}(\mathbf{x})\mathrm{grad}[c]] 
  \leq 0$, the maximum occurs on the parabolic boundary. 
  That is, 
  \begin{align}
    \mathop{\mathrm{max}}_{(\mathbf{x},t) \in \overline{\Omega}_{\mathcal{I}}} 
    \; c(\mathbf{x},t) = \mathop{\mathrm{max}}_{(\mathbf{x},t) \in 
      \Gamma_{\mathcal{I}}} \; c(\mathbf{x},t)
  \end{align}
\end{remark}

Maximum principles play a central role in the study of 
partial differential equations. Many uniqueness theorems 
and powerful estimates for elliptic and parabolic partial 
differential equations utilize some form of maximum principles 
\cite{Gilbarg_Trudinger,Pao}. Maximum principles also have 
important physical implications in mathematical modeling, 
as they place restrictions on physical quantities. One 
such implication is the non-negative constraint. We now 
show that, under certain assumptions, the non-negative 
constraint is a consequence of the maximum principle 
given by Theorem \ref{Theorem:TransientDMP_MP_parabolic}.
For the present discussion, let us assume that $\Gamma^{\mathrm{D}} 
= \partial \Omega$ (that is, we prescribe Dirichlet boundary 
conditions on the whole boundary). If $f(\mathbf{x},t) \geq 0$ 
(i.e., we have volumetric source), $c_p(\mathbf{x},t) \geq 0$ 
(i.e., we have non-negative prescribed Dirichlet boundary 
conditions on the whole boundary), and $c_{0}(\mathbf{x}) 
\geq 0$ (i.e., we have non-negative prescribed initial 
concentration); then the maximum principle given by 
Theorem \ref{Theorem:TransientDMP_MP_parabolic} implies 
that the quantity $c(\mathbf{x},t)$ is non-negative in 
the whole domain and at all times. That is, 
\begin{align}
  c(\mathbf{x},t) \geq 0 \quad \forall \mathbf{x} 
  \in \overline{\Omega} \; \mathrm{and} \; \forall t 
  \in [0, \mathcal{I}]
\end{align}
It should be noted that the above discussion on maximum 
principles and the non-negative constraint is in continuum 
setting. For most practical problems (which will involve 
complex geometries and spatially varying coefficients), 
it is not possible to find analytical solutions. Therefore, 
one has to resort to numerical solutions. This leads to 
the following questions, which are central to this paper. 
\emph{
Whether numerical formulations satisfy maximum principles 
and the non-negative constraint for transient diffusion 
equation. If so, under what conditions? If not, is it 
possible to fix a given numerical formulation to meet 
these important principles?} This area of research is 
popularly referred to as discrete maximum principles. 

\begin{remark}
  Some recent efforts \cite{Liska_Shashkov_CiCP_2008_v3_p852,
    Nakshatrala_Valocchi_JCP_2009_v228_p6726,
    Nagarajan_Nakshatrala_IJNMF_2011_v67_p820} have 
  addressed similar questions with respect to maximum 
  principles and the non-negative constraint, but all 
  these studies have considered \emph{steady} diffusion 
  equation. 
\end{remark}

\subsection{Discrete maximum principles}
The discrete analogy of maximum principles is commonly 
referred to as \emph{discrete maximum principles} (DMP). 
Some main factors which affect numerical solutions with 
respect to discrete maximum principles are:
\begin{enumerate}[(i)]
\item topology of the domain (e.g., shape of the 
  domain, features like holes in the domain), 
\item type of mesh (e.g., Delaunay, well-centered, 
  structured vs. unstructured), 
\item element type (simplicial vs. non-simplicial elements), 
\item mesh size (i.e., aspect ratio), 
\item medium properties (e.g., anisotropy, heterogeneity), 
\item order of approximation (i.e., low-order vs. 
  high-order), and 
\item temporal discretization (e.g., time 
  stepping scheme, selection of the time step). 
\end{enumerate}
The first six factors are equally applicable to steady 
anisotropic diffusion equation. Systematic studies on 
the effect of first five factors on maximum principles 
and the non-negative constraint can be found in references 
\cite{Nakshatrala_Valocchi_JCP_2009_v228_p6726,
  Nagarajan_Nakshatrala_IJNMF_2011_v67_p820,
  Mudunuru_Nakshatrala_IJNME_2012_v89_p1144}. Reference 
\cite{Payette_Nakshatrala_Reddy_IJNME_2012} discusses 
in detail about the sixth factor. The last factor (in 
combination with other six factors) is the subject 
matter of this paper. 

This leads to the problem statement of this paper: 
\emph{Develop a finite element methodology for linear 
  transient tensorial diffusion equation that satisfies 
  maximum principles and the non-negative constraint 
  on general computational grids for low-order finite 
  elements with no additional restrictions on the time 
  step.} 
To the best of our knowledge, such a methodology does not 
exist in the literature. In the next section, we shall 
extend the optimization-based methodologies that are 
presented in references \cite{Nakshatrala_Valocchi_JCP_2009_v228_p6726,
  Nagarajan_Nakshatrala_IJNMF_2011_v67_p820} for steady diffusion 
equations to transient diffusion equation. We shall explicitly 
enforce constraints on the nodal concentrations to satisfy 
maximum principles and the non-negative. We shall restrict 
to low-order finite elements, which include two-node line 
element, three-node triangular element, four-node quadrilateral 
element, four-node tetrahedron element, eight-node brick element, 
and six-node wedge element. 
However, it should be noted that the proposed methodology is 
\emph{not} applicable to high-order elements, as enforcing 
non-negative constraints at nodes does not imply non-negative 
concentrations throughout the domain for high-order elements 
(e.g., three-node line element, six-node triangular element) 
\cite{Payette_Nakshatrala_Reddy_IJNME_2012}.

\section{PROPOSED METHODOLOGY: DERIVATION AND IMPLEMENTATION DETAILS}
\label{Sec:TransientDMP_Derivation}
Herein, we shall employ the method of horizontal lines (also known 
as the Rothe method) \cite{Rothe_MA_1930_v102_p650} as opposed to 
the commonly employed method of vertical lines \cite{Hughes}. The 
method of horizontal lines is a discretization sequence in which 
the time is discretized first followed by spatial discretization. 
To this end, we shall define two sets of time levels: \emph{integral} 
and \emph{weighted} time levels. The time interval of interest $[0, 
  \mathcal{I}]$ is divided into $N$ non-overlapping subintervals 
such that 
\begin{align}
  [0, \mathcal{I}] = \bigcup_{n = 1}^{N} [t_{n-1}, t_n]
\end{align}
where $t_n \; (n = 0, \cdots, N)$ are referred to 
as integral time levels. For convenience, we shall 
assume that the time step $\Delta t$ to be uniform, 
which implies that 
\begin{align}
  \Delta t = \frac{\mathcal{I}}{N} \; \mathrm{and} 
  \; t_n = n \Delta t
\end{align}
However, it should be noted that the proposed 
methodology can be easily extended to non-uniform 
time steps. We shall apply the method of horizontal 
lines at weighted time levels, which are defined as 
follows:
\begin{align}
  t_{n+\eta} := (1 - \eta) t_n + \eta t_{n+1}
\end{align}
where the parameter $\eta \in [0, 1]$. The 
concentration and its rate at integral time 
levels are respectively denoted as follows:
\begin{subequations}
  \begin{align}
    c^{(n)}(\mathbf{x}) &= c(\mathbf{x}, t = t_n) \\
    v^{(n)}(\mathbf{x}) &= \frac{\partial c}{\partial t} 
    (\mathbf{x}, t = t_n)
  \end{align}
\end{subequations}
The following notation is used to denote quantities at 
weighted time levels: 
\begin{subequations}
  \label{Def:TransientDMP_time_step_definitions}
  \begin{align} 
    c^{(n+\eta)}(\mathbf{x}) &:= (1 - \eta) c^{(n)}(\mathbf{x}) 
    + \eta c^{(n+1)}(\mathbf{x}) \approx c(\mathbf{x},t_{n+\eta}) \\ 
    v^{(n+\eta)}(\mathbf{x}) &:= (1 - \eta) v^{(n)}(\mathbf{x}) 
    + \eta v^{(n+1)}(\mathbf{x}) \approx \frac{\partial c}{\partial t}
    (\mathbf{x},t = t_{n+\eta}) \\
    c_p^{(n+\eta)}(\mathbf{x}) &:= c_p(\mathbf{x},t_{n+\eta}) \\
    f^{(n+\eta)}(\mathbf{x}) &:= f(\mathbf{x},t_{n+\eta}) \\
    q_p^{(n+\eta)}(\mathbf{x}) &:= q_p(\mathbf{x},t_{n+\eta}) 
  \end{align}
\end{subequations}

\subsection{Derivation}
In designing the proposed methodology, attention will be 
exercised on two different aspects. The first aspect is 
to make sure that the non-negative constraint and maximum 
principles are preserved after both temporal and spatial 
discretizations. The second aspect is to achieve numerical 
stability in solving the resulting differential-algebraic 
equations. 
As we shall see in subsection \ref{Sec:TransientDMP_proposed_formulation}, 
we will be adding additional equations in the form of lower and upper 
bounds (i.e., inequality constraints). This implies that we will be 
dealing with differential-algebraic equations. It is important to 
note that numerical time integration schemes that are designed for 
ordinary differential equations may not be stable and accurate for 
solving differential-algebraic equations. This point has been 
discussed adequately in the literature (e.g., see references 
\cite{Ascher_Petzold,Hairer_Lubich_Roche,Hairer_Wanner}). An 
important work on numerical time integration of differential-algebraic 
equations is by Petzold \cite{Petzold_SIAMJSciStatComp_1982_v3_p367}, 
and the title of this paper (``Differential/algebraic equations are 
not ODEs'') succinctly summarizes the above discussion.

We shall employ the generalized-$\alpha$ method for temporal 
discretization. The generalized-$\alpha$ method was first 
proposed for second-order transient systems in Reference 
\cite{Chung_Hulbert_JAM_1993_v60_p371}, and later modified 
for first-order transient systems in Reference 
\cite{Jansen_Whiting_Hulbert_CMAME_2000_v190_p305}. It 
will be shown that the entire family of time stepping 
schemes under the generalized-$\alpha$ method may not 
be made to satisfy maximum principles and the 
non-negative constraint. The following derivation 
will reveal that only a subset of time stepping 
schemes satisfying $\alpha_f = 1$ and $\alpha_m 
= \gamma \in (0, 1]$ will be suitable. 
  \begin{remark}
    The elimination of some time stepping schemes under 
    the generalized-$\alpha$ family to meet maximum 
    principles should not be considered as a limitation 
    of the proposed methodology. It should be interpreted 
    that some time stepping schemes are not suitable to 
    enforce maximum principles and the non-negative 
    constraint. This in similar to fact that all the time 
    stepping schemes under the Newmark family are not 
    energy preserving \cite{Hughes}. If one want to 
    preserve energy and employ a time stepping scheme 
    from the Newmark family, the only choice will be 
    the trapezoidal rule. Similarly, one can employ 
    other time stepping schemes (i.e., time stepping 
    schemes not satisfying $\alpha_f = 1$ and $\alpha_m 
    = \gamma$) to solve transient diffusion equations, 
    but the numerical results need not satisfy maximum 
    principles and the non-negative constraint.
  \end{remark}

After applying the generalized-$\alpha$ method 
to the governing equations 
\eqref{Eqn:TransientDMP_PDE}--\eqref{Eqn:TransientDMP_Neumann}, 
we obtain the following equations:
\begin{subequations}
  \begin{alignat}{2}
    \label{Eqn:Transient_DMP_PDE_time_discrete}
    &v^{(n+\alpha_m)}(\mathbf{x}) - \mathrm{div}[\mathbf{D}
      (\mathbf{x})\mathrm{grad}[c^{(n+\alpha_f)}]] = 
    f^{(n+\alpha_f)}(\mathbf{x}) &\quad& \mathrm{in} 
    \; \Omega \\
    &c^{(n+\alpha_f)}(\mathbf{x}) = c_p^{(n+\alpha_f)}
    (\mathbf{x}) &\quad& \mathrm{on} \; 
    \Gamma^{\mathrm{D}} \\
    &\mathbf{\hat{n}}(\mathbf{x}) \cdot \mathbf{D}(\mathbf{x}) 
    \mathrm{grad}[c^{(n+\alpha_f)}] = q_p^{(n+\alpha_f)}(\mathbf{x}) 
    &\quad& \mathrm{on} \; \Gamma^{\mathrm{N}}
  \end{alignat}
\end{subequations}
where the parameters $\alpha_m, \; \alpha_f \in 
[0, 1]$. In addition, we have the following 
relationship: 
\begin{align}
  c^{(n+1)}(\mathbf{x}) = c^{(n)}(\mathbf{x}) + 
  \Delta t \left((1 - \gamma)v^{(n)}(\mathbf{x}) 
  + \gamma v^{(n+1)}(\mathbf{x})\right)
\end{align}
where the parameter $\gamma \in [0, 1]$. The 
initial condition takes the following form:
\begin{align}
  c^{(0)}(\mathbf{x}) = c_0(\mathbf{x}) 
  \quad \mathrm{in} \; \Omega
\end{align}
\begin{remark}
  Many popular time stepping schemes are special case 
  of generalized-$\alpha$ method. For example, forward 
  Euler $(\alpha_m = 1, \alpha_f = 1, \gamma = 0)$, 
  trapezoidal rule $(\alpha_m = 1, \alpha_f = 1, 
  \gamma = 1/2)$, and backward Euler $(\alpha_m = 1, 
  \alpha_f = 1, \gamma = 1)$. 
\end{remark}

Herein, we shall take $\alpha_m = \gamma$. This selection is 
intended to inherit the non-negative property for the resulting 
time discrete equations. The time discrete equations in terms 
of concentration take the following form: Find $c^{(n+\alpha_f)}
(\mathbf{x})$ such that we have
\begin{subequations}
  \label{Eqn:Transient_DMP_diffusion_with_decay_BVP}
  \begin{alignat}{2}
    \label{Eqn:Transient_DMP_diffusion_with_decay}
    &\frac{1}{\alpha_f \Delta t} c^{(n+\alpha_f)}
    (\mathbf{x}) - \mathrm{div}[\mathbf{D}(\mathbf{x})
      \mathrm{grad}[c^{(n+\alpha_f)}]] = f^{(n+\alpha_f)}
    (\mathbf{x}) + \frac{1}{\alpha_f \Delta t} 
    c^{(n)}(\mathbf{x}) &\quad& \mathrm{in} \; \Omega \\
    \label{Eqn:Transient_DMP_diffusion_with_decay_Dirichlet}
    &c^{(n+\alpha_f)}(\mathbf{x}) = c_p^{(n+\alpha_f)}
      (\mathbf{x}) &\quad& \mathrm{on} \; 
      \Gamma^{\mathrm{D}} \\
    \label{Eqn:Transient_DMP_diffusion_with_decay_Neumann}
    &\mathbf{\hat{n}}(\mathbf{x}) \cdot \mathbf{D}(\mathbf{x}) 
    \mathrm{grad}[c^{(n+\alpha_f)}] = q_p^{(n+\alpha_f)}(\mathbf{x}) 
    &\quad& \mathrm{on} \; \Gamma^{\mathrm{N}}
  \end{alignat}
\end{subequations}
The above boundary value problem 
\eqref{Eqn:Transient_DMP_diffusion_with_decay}--\eqref{Eqn:Transient_DMP_diffusion_with_decay_Neumann} is a second-order inhomogeneous 
elliptic partial differential equation with Dirichlet 
and Neumann boundary conditions. Specifically, equation 
\eqref{Eqn:Transient_DMP_diffusion_with_decay} is the 
well-known steady-state anisotropic diffusion equation 
with decay, as $\alpha_f \Delta t$ will be always 
positive. The decay coefficient can be identified 
as $1/(\alpha_f \Delta t)$, and the volumetric 
source term is $f^{(n+\alpha_f)}(\mathbf{x}) + \frac{1}
{\alpha_f \Delta t}c^{(n)}(\mathbf{x})$. This boundary 
value problem is also known to satisfy maximum principles 
and the non-negative constraint. The selection $\alpha_m 
= \gamma$ made it possible to preserve maximum principles 
and the non-negative constraint by ensuring the decay 
coefficient to be positive, and the volumetric source 
at discrete time levels to be non-negative. 

It should be emphasized that an arbitrary temporal 
discretization will not preserve maximum principles 
and the non-negative constraint. An important aspect 
is to ensure is that the resulting equation after 
a temporal discretization of transient diffusion 
equation \eqref{Eqn:TransientDMP_PDE} is a diffusion 
equation with decay instead of a Helmholtz equation. 
Diffusion equation with decay takes the following form:
\begin{align}
  \alpha(\mathbf{x}) c(\mathbf{x}) - \mathrm{div}
        [\mathbf{D}(\mathbf{x}) \mathrm{grad}[c]] 
        = f(\mathbf{x})
\end{align}
with $\alpha(\mathbf{x}) \geq 0$. If $\alpha(\mathbf{x}) < 0$, 
the equation is referred to as Helmholtz equation. It should 
be noted that Helmholtz equation does not have a maximum 
principle similar to the one possessed by diffusion equation 
with decay \cite{Gilbarg_Trudinger}. Hence, in order to 
preserve maximum principles and the non-negative constraint, 
the temporal discretization based on the method of horizontal 
lines should be carried out in such a way that the resulting 
decay coefficient is non-negative.
In Appendix, we shall outline several other ways 
of carrying out temporal discretization, and 
discuss the drawbacks of such approaches in 
meeting maximum principles and the non-negative 
constraint. 

Recently, Nagarajan and Nakshatrala 
\cite{Nagarajan_Nakshatrala_IJNMF_2011_v67_p820} have 
proposed a procedure for enforcing maximum principles 
and the non-negative constraint for steady diffusion 
with decay equation, which we shall modify to solve 
equations 
\eqref{Eqn:Transient_DMP_diffusion_with_decay}--\eqref{Eqn:Transient_DMP_diffusion_with_decay_Neumann}. 
We start by applying Galerkin formalism to equations 
\eqref{Eqn:Transient_DMP_diffusion_with_decay}--\eqref{Eqn:Transient_DMP_diffusion_with_decay_Neumann}. The corresponding weak form takes 
the following form: Find $c^{(n+\alpha_f)}(\mathbf{x}) \in 
\mathcal{P}_{n+\alpha_f}$ such that we have 
\begin{align}
  \label{Eqn:Transient_DMP_weak_form}
  \int_{\Omega} w(\mathbf{x}) \frac{1}{\alpha_f \Delta t} 
  c^{(n+\alpha_f)}(\mathbf{x}) \; \mathrm{d} \Omega &+ 
  \int_{\Omega} \mathrm{grad}[w] \cdot \mathbf{D} 
  (\mathbf{x}) \mathrm{grad}\left[c^{(n+\alpha_f)}\right] 
  \; \mathrm{d} \Omega \nonumber \\
  &= \int_{\Omega} w(\mathbf{x}) f^{(n+\alpha_f)}
  (\mathbf{x}) \; \mathrm{d} \Omega + 
  \int_{\Omega} w(\mathbf{x}) \frac{1}{\alpha_f \Delta t} 
  c^{(n)}(\mathbf{x}) \; \mathrm{d} \Omega \nonumber \\
  &+ \int_{\Gamma^{\mathrm{N}}} w(\mathbf{x}) 
  q_p^{(n+\alpha_f)}(\mathbf{x}) \; \mathrm{d} \Gamma 
  \quad \forall w(\mathbf{x}) \in \mathcal{Q}
\end{align}
where the function spaces $\mathcal{P}_{n+\alpha_f}$ 
and $\mathcal{Q}$ are defined as follows:
\begin{subequations}
  \begin{align}
    \label{Eqn:Transient_DMP_Pngamma_space}
    &\mathcal{P}_{n+\alpha_f} := \left\{c(\mathbf{x}) 
    \in H^{1}(\Omega) \; \big| \; c(\mathbf{x}) = 
    c_p^{(n+\alpha_f)}(\mathbf{x}) \; \mathrm{on} \; 
    \Gamma^{\mathrm{D}} \right\} \\
    \label{Eqn:Transient_DMP_Q_space}
    &\mathcal{Q} := \left\{w(\mathbf{x}) \in H^{1}
    (\Omega) \; \big| \; w(\mathbf{x}) = 0 \; 
    \mathrm{on} \; \Gamma^{\mathrm{D}}\right\}
  \end{align}
\end{subequations}
After executing the usual steps of the 
finite element method, the above weak 
form \eqref{Eqn:Transient_DMP_weak_form} 
can be converted to a system of linear 
equations of the following form:
\begin{align}
  \label{Eqn:Transient_DMP_system_of_LE}
  \boldsymbol{K} \boldsymbol{c}^{(n+\alpha_f)} 
  = \boldsymbol{f}^{(n+\alpha_f)}
\end{align}
where ``$ndofs$'' denotes the number of (free) degrees-of-freedom, 
$\boldsymbol{c}^{(n+\alpha_f)} \in \mathbb{R}^{ndofs}$ denotes the 
unknown vector containing nodal concentrations at the weighted 
time level $t_{n+\alpha_f}$, $\boldsymbol{f}^{(n+\alpha_f)} \in 
\mathbb{R}^{ndofs}$ is a known vector, and $\boldsymbol{K}$ 
is a symmetric and positive definite matrix. 
It will be shown in a subsequent section that the finite element 
solution obtained by solving the system of linear equations 
\eqref{Eqn:Transient_DMP_system_of_LE} may not satisfy 
maximum principles and the non-negative constraint. Using 
optimization-based techniques, we now modify the above 
solution procedure to meet these important physical 
constraints. 

\subsection{Enforcing maximum principles and the non-negative constraint}
\label{Sec:TransientDMP_proposed_formulation}
We shall denote the standard inner product on finite dimensional 
Euclidean spaces by $\langle \cdot ; \cdot \rangle$. We shall 
use the symbols $\preceq$ and $\succeq$ to denote component-wise 
inequalities for vectors. That is, for given any two (finite 
dimensional) vectors $\boldsymbol{a}$ and $\boldsymbol{b}$ 
\begin{align}
  \boldsymbol{a} \preceq \boldsymbol{b} 
  \quad \mbox{means that} \quad a_i \leq 
  b_i \; \forall i
\end{align}
Similarly, one can define the symbol $\succeq$. 
The optimization problem can then be written as 
follows: 
\begin{subequations}
  \label{Eqn:Transient_DMP_optimization}
  \begin{align}
    \label{Eqn:Transient_DMP_objective}
    &\mathop{\mathrm{minimize}}_{\boldsymbol{c}^{(n+\alpha_f)} 
      \in \mathbb{R}^{ndofs}} \quad \frac{1}{2} \left<
    \boldsymbol{c}^{(n+\alpha_f)}; \boldsymbol{K} 
    \boldsymbol{c}^{(n+\alpha_f)}\right> - \left< 
    \boldsymbol{c}^{(n+\alpha_f)}; 
    \boldsymbol{f}^{(n+\alpha_f)} \right> \\
    \label{Eqn:Transient_DMP_constraint}
    &\mbox{subject to} \quad c_{\mathrm{min}}^{(n+\alpha_f)} 
    \boldsymbol{1} \preceq \boldsymbol{c}^{(n+\alpha_f)} 
    \preceq c_{\mathrm{max}}^{(n+\alpha_f)} \boldsymbol{1}
  \end{align}
\end{subequations}
where $\boldsymbol{1}$ is a vector containing ones of 
size $ndofs \times 1$, and $c_{\mathrm{min}}^{(n+\alpha_f)}$ 
and $c_{\mathrm{max}}^{(n+\alpha_f)}$ are respectively the 
lower and upper bounds. For enforcing maximum principles, 
$c_{\mathrm{min}}^{(n+\alpha_f)}$ and $c_{\mathrm{max}}^{(n+\alpha_f)}$ 
can be taken as follows:
\begin{subequations}
  \label{Eqn:Transient_DMP_constraints_ngamma}
  \begin{align}
    \label{Eqn:Transient_DMP_cmin_ngamma}
    c_{\mathrm{min}}^{(n+\alpha_f)} &:= \mathrm{min} 
    \left\{\min_{\mathbf{x} \in \Omega} \; c_{0}(\mathbf{x}), 
    \; \min_{\mathbf{x} \in \partial \Omega} \; c_{p}^{(n+\alpha_f)} 
    (\mathbf{x})\right\} \\
    \label{Eqn:Transient_DMP_cmax_nalpha_f}
    c_{\mathrm{max}}^{(n+\alpha_f)} &:= \mathrm{max} 
    \left\{\max_{\mathbf{x} \in \Omega} \; c_{0}(\mathbf{x}), 
    \; \max_{\mathbf{x} \in \partial \Omega} \; c_{p}^{(n+\alpha_f)} 
    (\mathbf{x})\right\}
  \end{align}
\end{subequations}
For problems involving only the non-negative 
constraint, one can employ the following:
\begin{align}
  \label{Eqn:Transient_DMP_non_negative_constraint}
  c_{\mathrm{min}}^{(n+\alpha_f)} = 0 \; \mathrm{and} \; 
  c_{\mathrm{max}}^{(n+\alpha_f)} = +\infty
\end{align}
Alternatively, for enforcing the non-negative 
constraint, one can replace the constraint 
\eqref{Eqn:Transient_DMP_constraint} with 
the following:
\begin{align}
  \label{Eqn:Transient_DMP_non_negative_vector_constraint}
  \mathbf{0} \preceq \boldsymbol{c}^{(n+\alpha_f)}
\end{align}
where $\boldsymbol{0}$ denotes the vector of size 
$ndofs \times 1$ containing zeros. It should be 
noted that the above optimization problem 
\eqref{Eqn:Transient_DMP_optimization} belongs 
to \emph{quadratic programming}. Since, for the 
problem at hand, the matrix $\boldsymbol{K}$ is 
positive definite (which makes the objective 
function \eqref{Eqn:Transient_DMP_objective} 
convex) the optimization problem belongs to 
\emph{convex quadratic programming}. A sound 
mathematical theory is already in place for 
studying convex quadratic programming 
\cite{Boyd_convex_optimization}, and 
several efficient algorithms are available 
in the literature \cite{Pang_CCE_1983_v7_p583,
  Ye_TSE_MP_1989_v44_p157,Boyd_convex_optimization}. 
In this paper, we shall employ the built-in optimization 
solver available in MATLAB \cite{MATLAB_2012a}. Some 
other popular packages that can handle convex 
quadratic programming optimization problems are 
GAMS \cite{GAMS}, TAO \cite{tao-user-ref}, and 
DAKOTA \cite{Dakota}.

Once the nodal concentrations are obtained at 
weighted time level, one can obtain the nodal 
concentrations at integral time levels as follows:
\begin{align}
  \label{Eqn:TransientDMP_nodal_concs}
  \boldsymbol{c}^{(n+1)} = \frac{\boldsymbol{c}^{(n+\alpha_f)} - 
    (1 - \alpha_f) \boldsymbol{c}^{(n)}}{\alpha_f}
\end{align}
Although $\boldsymbol{c}^{(n+\alpha_f)} \succeq \boldsymbol{0}$, 
the nodal concentrations at integral time levels based on 
equation \eqref{Eqn:TransientDMP_nodal_concs} need not be 
non-negative if $\alpha_f \neq 1$. To put it differently, 
one is assured of satisfying maximum principles and the 
non-negative constraint under the proposed methodology 
if $\alpha_{m} = \gamma \in (0, 1]$ and $\alpha_{f} = 1$. 

  \subsubsection{Calculation of the rate of nodal concentrations} 
  There are two ways one could calculate the nodal 
  rates of concentration. The first method is to 
  directly calculate the rate of nodal concentrations 
  at integral time levels using the following expression: 
  \begin{align}
    \label{Eqn:TransientDMP_rates_integral_first}
    \boldsymbol{v}^{(n+1)} = \frac{\boldsymbol{c}^{(n+1)} 
      - \boldsymbol{c}^{(n)} - (1 - \gamma) \Delta t
      \boldsymbol{v}^{(n)}}{\gamma \Delta t}
  \end{align}
  It should also be emphasized that $\gamma = 0$ cannot 
  be employed under the proposed methodology to meet 
  maximum principles and the non-negative constraint.
  The second method is to first calculate the rates 
  of concentration at weighted time levels using the 
  following expression: 
  \begin{align}
    \label{Eqn:TransientDMP_rates_weighted}
    \boldsymbol{v}^{(n+\gamma)} = \frac{\boldsymbol{c}^{(n+1)} - 
      \boldsymbol{c}^{(n)}}{\Delta t}
  \end{align}
  The rate of nodal concentrations at the integral 
  time levels can then be calculated using the 
  following consistent approximation:
  \begin{align}
    \label{Eqn:TransientDMP_rates_integral_second}
    \boldsymbol{v}^{(n+1)} = \gamma \boldsymbol{v}^{(n+\gamma)} 
    + (1 - \gamma) \boldsymbol{v}^{(n+1+\gamma)}  
  \end{align}
  Both these consistent ways of obtaining the rates 
  of concentration are pictorially described in Figure 
  \ref{Fig:Transient_DMP_weighted_quantities}.
  
  The various steps involved in the numerical implementation 
  of the proposed methodology to satisfy maximum principles 
  and the non-negative constraint are summarized in Algorithm 
  \ref{Algo:TransientDMP_implementation}, which could serve 
  as a quick reference during computer code design and 
  implementation. 

\begin{algorithm}[h]
  \caption{Implementation of the proposed methodology based on $\alpha_f = 1$.}
  \label{Algo:TransientDMP_implementation}
  \begin{algorithmic}[1]
    \STATE Input: Initial condition $c(\mathbf{x})$, 
    Dirichlet boundary conditions $c_{p}(\mathbf{x},t)$, 
    Neumann boundary conditions $q_p(\mathbf{x},t)$, 
    time step $\Delta t$, total time of interest 
    $\mathcal{I}$, $\alpha_m = \gamma \in (0, 1]$.
    \STATE Construct initial nodal concentrations 
    $\boldsymbol{c}^{(0)}$
    \STATE Set $\boldsymbol{c}^{(n)} \longleftarrow 
    \boldsymbol{c}^{(0)}$, $t \longleftarrow 0$, 
    $n \longleftarrow 0$
    \WHILE{$t < \mathcal{I}$}
    \STATE Calculate $c_{\mathrm{min}}^{(n+1)}$ 
    and $c_{\mathrm{max}}^{(n+1)}$ (see equations 
    \eqref{Eqn:Transient_DMP_constraints_ngamma}--\eqref{Eqn:Transient_DMP_non_negative_constraint})
    \STATE Call non-negative solver to obtain 
    $\boldsymbol{c}^{(n+1)}$ 
    \begin{align*}
      \mathop{\mathrm{minimize}}_{\boldsymbol{c}^{(n+1)} \in \mathbb{R}^{ndofs}} 
      &\quad \frac{1}{2} \langle {\boldsymbol{c}^{(n+1)}}; \boldsymbol{K} 
      \boldsymbol{c}^{(n+1)} \rangle - \langle {\boldsymbol{c}^{(n+1)}}; 
      \boldsymbol{f}^{(n+1)} \rangle \\
      \mbox{subject to} &\quad 
      c_{\mathrm{min}}^{(n+1)} \boldsymbol{1} \preceq 
      \boldsymbol{c}^{(n+1)} \preceq c_{\mathrm{max}}^{(n+1)} 
      \boldsymbol{1}
    \end{align*}
    \STATE If needed, obtain the rate of nodal concentrations 
    at both integral and weighed time levels (see equations 
    \eqref{Eqn:TransientDMP_rates_integral_first}--\eqref{Eqn:TransientDMP_rates_integral_second} and Figure 
    \ref{Fig:Transient_DMP_weighted_quantities})
    \begin{align*}
      \boldsymbol{v}^{(n+1)} &= \frac{\boldsymbol{c}^{(n+1)} 
        - \boldsymbol{c}^{(n)} - (1 - \gamma) \Delta t 
        \boldsymbol{v}^{(n)}}{\gamma \Delta t} \\
      or \\
      \boldsymbol{v}^{(n+\gamma)} &= \frac{\boldsymbol{c}^{(n+1)} 
        - \boldsymbol{c}^{(n)}}{\Delta t} 
      \quad \mathrm{and} \quad  
      \boldsymbol{v}^{(n+1)} = \gamma \boldsymbol{c}^{(n+\gamma)} 
        + (1 - \gamma) \boldsymbol{c}^{(n+1+\gamma)} 
    \end{align*}
    \STATE Set $\boldsymbol{c}^{(n)} \longleftarrow 
    \boldsymbol{c}^{(n+1)}$, $t \longleftarrow t + 
    \Delta t$, $n \longleftarrow n + 1$
    \ENDWHILE
  \end{algorithmic}
\end{algorithm}

\section{REPRESENTATIVE NUMERICAL RESULTS}
\label{Sec:TransientDMP_NR}
In this section, we shall illustrate the performance of 
the proposed methodology for enforcing maximum principles 
and the non-negative constraint using several canonical 
problems. We shall also perform numerical convergence 
studies on the proposed methodology. We shall restrict 
our numerical studies to one- and two-dimensional problems. 
It should be, however, noted that the proposed methodology 
is equally applicable for solving three-dimensional problems. 
We do not solve any three-dimensional problem here as, in 
comparison with one- and two-dimensional problems, there 
are no additional difficulties other than the usual book 
keeping that is associated with most three-dimensional 
problems. In all our numerical simulations we have 
employed low-order finite elements, and have taken 
$\alpha_f = 1$. It is assumed that $\alpha_m = \gamma 
= 1$, unless stated otherwise. The specific selection 
of $\gamma$ does not appear in the calculation of nodal 
concentrations. But it will be needed to calculate the rate 
of nodal concentrations, which is discussed in the previous 
section. 

\subsection{One-dimensional problem with uniform initial condition}
\label{Subsec:TransientDMP_uniform_IC}
The following one-dimensional problem is taken from Reference 
\cite{Carslaw_Jaeger}, which is also used as a test problem in 
Reference \cite{Harari_CMAME_2004_v193_p1491} in the context 
of discrete maximum principles. The computational domain is 
$\Omega := (0,1)$. The governing equations of the test problem 
take the following form: 
\begin{subequations}
  \label{Eqn:TransientDMP_1D_problem}
  \begin{align}
    \label{Eqn:TransientDMP_1D_GE}
    &\frac  {\partial{c(\mathrm{x},t)}} {\partial{t}} - 
    \frac{\partial^2{c(\mathrm{x},t)}}{\partial{\mathrm{x}^2}}  
    =  0 \quad \mathrm{in} \; \Omega_{\mathcal{I}} := (0,1) 
    \times (0, \mathcal{I}) \\
    & \frac{\partial c(\mathrm{x}=0,t)}{\partial \mathrm{x}} 
    = 0, \; c(\mathrm{x}=1,t) = 0 \quad \forall t \in 
    (0, \mathcal{I}] \\
      \label{Eqn:TransientDMP_1D_IC}
      & c(\mathrm{x},0) = 1 \quad \forall \mathrm{x} \in [0,1]
  \end{align}
\end{subequations}
The analytical solution can be written as follows:
\begin{align}
  \label{Eqn:TransientDMP_1D_analytical}
  c(\mathrm{x},t) = \frac{4}{\pi} \sum_{n=0}^{\infty} 
  \frac{(-1)^n}{(2 n + 1)} \exp\left[- \frac{(2 n + 1)^2 
      \pi^2 t}{4}\right]\cos \left[\frac{(2n + 1) \pi 
      \mathrm{x}}{2}\right]
\end{align}
The analytical solution is bounded between zero and unity. In 
the numerical simulation, we have divided the computational 
domain into five equal linear finite elements, and have taken 
the time step to be $\Delta t = 0.001$ (which is chosen 
arbitrarily). Figure \ref{Fig:OneD_uniform_IC_u_x_dot6} 
compares the analytical solution with the numerical solutions 
obtained using the single-field formulation and the proposed 
methodology. The single-field formulation violates the maximum 
principle, as the obtained numerical solution is greater than unity. 
The proposed methodology satisfies the maximum principle for 
all times. 
The rate of nodal concentrations under the proposed methodology 
are shown in Figure \ref{Fig:OneD_Harari_vn_second_method} for 
various values of $\gamma = 0.1$, $0.5$ and $1.0$. Note that 
under the proposed methodology $\alpha_m = \gamma$. We 
have employed the second method for calculating the rates 
(i.e., equations 
\eqref{Eqn:TransientDMP_rates_weighted}--\eqref{Eqn:TransientDMP_rates_integral_second}).

\begin{remark}
  For this problem, the initial condition is not compatible with the 
  boundary conditions. That is, the (homogeneous) Dirichlet boundary 
  condition at the right end of the domain at time $t = 0$ is not equal 
  to the initial condition. Hence, there is \emph{no} classical solution 
  to the initial boundary value problem given by equations 
  \eqref{Eqn:TransientDMP_1D_GE}--\eqref{Eqn:TransientDMP_1D_IC} in 
  the sense that $c(\mathrm{x},t) \in C^{2}_{1}(\Omega_{\mathcal{I}}) 
  \cap C(\overline{\Omega}_{\mathcal{I}})$. The analytical solution given 
  in equation \eqref{Eqn:TransientDMP_1D_analytical} should be 
  interpreted in Lebesgue measurable sense. 
\end{remark}

\subsection{One-dimensional problem with non-uniform initial condition}
\label{One_dimensional_problem}
Consider the following simple one-dimensional problem with 
homogeneous forcing function. This problem is a modification 
to one of the examples given in Reference \cite{Crank_1980}. 
The initial boundary value problem can be written as follows:
\begin{subequations}
  \label{Eqn:TransientDMP_1D_problem}
  \begin{align}
    &\frac  {\partial{c(\mathrm{x},t)}} {\partial{t}} - 
    \frac{\partial^2{c(\mathrm{x},t)}}{\partial{\mathrm{x}^2}} 
    =  0 \quad \mathrm{in} \; \Omega_{\mathcal{I}} := (0,1) \times 
    (0, \mathcal{I}) \\
    & c(\mathrm{x}=0,t) = c(\mathrm{x}=1,t) = 0 \quad 
    \forall t \in (0, \mathcal{I}] \\
    & c(\mathrm{x},0) =
    \begin{cases}
      \label{Eqn:TransientDMP_1D_problem_IC}
      1 &\text{if $ \mathrm{x} \in [a,b]$}\\
      0 &\text{otherwise}
    \end{cases}
  \end{align}
\end{subequations}
The analytical solution to the above problem is given by 
\begin{align}
  \label{Eqn:TransientDMP_1D_problem_analytical}
  c(\mathrm{x},t) = \frac{2}{\pi} \sum_{n = 1}^{\infty} 
  \frac{1}{n} (\cos(n\pi a) - \cos(n\pi b)) 
  \sin(n\pi \mathrm{x}) \exp\left[-n^{2} \pi^{2} t\right]
\end{align}
Herein, we have taken $a = 0.4$ and $b = 0.6$. 

Figure \ref{Fig:TransientDMP_1D_proposed_formulation} shows 
that the numerical solution from the proposed methodology 
compares well point-wise with the analytical solution, and 
satisfies the maximum principle and the non-negative constraint. 
Figure \ref{Fig:TransientDMP_1D_convergence} shows 
the numerical convergence of the proposed methodology 
and the standard single-field formulation in $L_2$-norm 
and $H^1$-seminorm. Note that the convergence in $L_2$-norm 
and $H^{1}$-seminorm is in integral sense and need not imply 
point-wise convergence. The convergence study is 
carried out by employing simultaneous spatial and temporal refinements 
satisfying the condition $\Delta t \propto (\Delta \mathrm{x})^2$. The 
coarsest mesh has 11 nodes, and the corresponding time step used for 
this mesh is $\Delta t = 10^{-3}$.

Figure \ref{Fig:TransientDMP_1D_min_max_conc} shows the 
variation of the minimum and maximum concentrations in 
the domain with respect to time for a fixed mesh but 
for different time steps under the standard single-field 
formulation. Note that for this problem the minimum 
concentration should be zero, and the maximum concentration 
should be unity. Clearly, the results from the standard 
single-field formulation violated both the upper and 
lower bounds. 
Figure \ref{Fig:TransientDMP_1D_nonuniform_min_max_fixed_mesh} 
shows the effect of mesh refinement for a fixed time step on 
the violation of the maximum principle under the single-field 
formulation. For a given mesh, the extent of the violation 
will be greater for smaller time steps. On the other hand, 
for a given time step, the extent of the violation decreases 
with mesh refinement, which will \emph{not} be the trend in 
the case of anisotropy. (For example, see the test problems 
given in subsections \ref{Subsec:TransientDMP_plate_with_hole} 
and \ref{Subsec:TransientDMP_2D_heterogeneous}.) 
In all the cases considered, the proposed methodology 
produced concentrations that satisfy the maximum 
principle and the non-negative constraint. 

\subsection{Two-dimensional problem with non-uniform initial condition}
\label{Two_dimensional_problem}
This test problem is a two-dimensional extension 
of the problem described earlier in subsection 
\ref{One_dimensional_problem}. The governing 
equations take the following form: 
\begin{subequations}
  \label{Eqn:TransientDMP_2D_problem}
  \begin{align}
    &\frac{\partial{c(\mathrm{x},\mathrm{y},t)}}{\partial{t}} 
    -  \left(\frac{\partial^2{c(\mathrm{x},\mathrm{y},t)}}{\partial{\mathrm{x}^2}}
    +\frac{\partial^2{c(\mathrm{x},\mathrm{y},t)}}{\partial{\mathrm{y}^2}}\right) = 0
    \quad \mathrm{in} \; \Omega_{\mathcal{I}} := (0,1) \times (0,1) \times (0, \mathcal{I}) \\
    & c(\mathrm{x}=0,\mathrm{y},t) = c(\mathrm{x}=1,\mathrm{y},t) = 0, \; 
    c(\mathrm{x},\mathrm{y}=0,t) = c(\mathrm{x},\mathrm{y}=1,t) = 0  \\
    & c(\mathrm{x},\mathrm{y},0) =
    \begin{cases}
      \label{Eqn:TransientDMP_2D_problem_IC}
      1 &\text{if $ \mathrm{x} \in [a,b] \times [a,b]$}\\
      0 &\text{otherwise}
    \end{cases}
  \end{align}
\end{subequations}
where $a = 0.4$ and $b = 0.6$. Figure 
\ref{Fig:TransientDMP_2D_pictorial_rep} gives a 
pictorial description of the test problem. The 
analytical solution can be written as follows:
\begin{align}
  \label{Eqn:TransientDMP_2D_problem_analytical}
  c(\mathrm{x},\mathrm{y},t) = &\frac{4}{\pi^2} 
  \sum_{m = 1}^{\infty} \sum_{n = 1}^{\infty} \frac{1}{mn} 
  (\cos(m\pi a) - \cos(m\pi b)) (\cos(n\pi a) - 
  \cos(n\pi b)) \nonumber \\
  &\sin(n\pi \mathrm{x}) \sin(m\pi \mathrm{y}) 
  \exp\left[-(m^{2}+n^{2}) \pi^{2} t\right]
\end{align}
Hierarchical meshes are employed in this numerical 
$h$-convergence study, which are illustrated in 
Figure \ref{Fig:TransientDMP_2D_convergence_meshes}.  
The numerical $h$-convergence of the proposed methodology is 
shown in Figure \ref{Fig:TransientDMP_2D_proposed_convergence}. 
The performance of the proposed methodology is compared with 
that of the single-field formulation and MATLAB's PDE Toolbox 
\cite{MATLAB_2012a} in Figures \ref{Fig:TransientDMP_2D_min_conc} 
and \ref{Fig:TransientDMP_2D_pdetool}.

\subsection{Transient anisotropic diffusion in square plate with a hole}
\label{Subsec:TransientDMP_plate_with_hole}
The computational domain is given by $\Omega := (0,1) 
\times (0,1) - [0.45,0.55] \times [0.45,0.55]$. The 
initial concentration in the domain is taken to be 
zero (i.e., $c_0(\mathbf{x}) = 0$). The volumetric 
source is zero (i.e., $f(\mathbf{x},t) = 0$). The 
inner hole is prescribed with a constant concentration 
of unity, and the outer hole is prescribed with a 
constant concentration of zero. The diffusivity 
tensor is taken as follows:
\begin{align}
  \mathbf{D}(\mathbf{x}) = \mathbf{R} \mathbf{D}_0 
  \mathbf{R}^{\mathrm{T}}
\end{align}
where $\mathbf{D}_0$ and the rotation tensor 
are, respectively, defined as follows:
\begin{subequations}
  \begin{align}
    \mathbf{D}_0 &= \left(\begin{array}{cc} 
      k_1 & 0 \\
      0 & k_2
    \end{array}\right) \\
    \mathbf{R} &= \left(\begin{array}{cc} 
      +\cos(\theta) & -\sin(\theta) \\
      +\sin(\theta) & +\cos(\theta) \\
    \end{array}\right)
  \end{align}
\end{subequations}
with the values $k_1 = 10$, $k_2 = 10^{-3}$ and $\theta = 
-\pi/6$. Using the maximum principle given by Theorem 
\ref{Theorem:TransientDMP_MP_parabolic}, it can be 
concluded that the concentration in the domain should 
be between zero and unity. This test problem is used 
to illustrate the following aspects: 
\begin{enumerate}[(i)]
\item The numerical results from COMSOL 
  \cite{multiphysics2012version} (which 
  is a popular commercial finite element 
  software package) do not satisfy the 
  maximum principle and the non-negative 
  constraint for transient anisotropic 
  diffusion. 
\item The proposed methodology satisfies the 
  maximum principle and the non-negative 
  constraint even on unstructured meshes 
  with no additional restrictions on the 
  time step. 
\item The approach of using the backward Euler 
  time stepping scheme with lumped capacity 
  matrix does not guarantee non-negative 
  solutions in the case of anisotropic 
  diffusion.
\end{enumerate}

Using numerical simulations it has been found 
that the transient solution is very close to 
the steady-state solution for time greater than 
0.05. Therefore, the time steps for this test 
problem are chosen to be smaller than or equal 
to 0.05 so that they are appropriate for transient 
analyses. 

We first show the results obtained using COMSOL 
\cite{multiphysics2012version}. Two different meshes 
are employed in the numerical simulations, which are 
shown in Figure \ref{Fig:TransientDMP_COMSOL_meshes}. 
The variation of the minimum concentration with time 
is shown in Figure \ref{Fig:TransientDMP_COMSOL_min_conc}, 
and the numerical results from COMSOL did not satisfy the 
non-negative constraint. Figures \ref{Fig:TransientDMP_COMSOL_Q4} 
and \ref{Fig:TransientDMP_COMSOL_T3} show the spread of the 
violation of the non-negative constraint and the concentration 
profiles using COMSOL for four-node structured mesh and 
three-node unstructured mesh, respectively.  From these 
figures, the following two observations can be made: 
\begin{enumerate}[(a)]
\item The magnitude of the violation of the 
  non-negative constraint increases as the 
  time step decreases. 
\item The violation reaches a steady-state value 
  after sufficient time, which is around $t = 0.05$ 
  for this problem. It should be emphasized that this 
  steady-state value for minimum concentration is a 
  significant non-negative number, and the violation 
  of the non-negative constraint is nearly 5\%.
\end{enumerate}

The aforementioned problem is also solved using the proposed 
methodology. Figure \ref{Fig:TransientDMP_QP_code_meshes} 
shows the unstructured computational meshes used in the 
numerical simulation. The concentration profiles obtained 
under the proposed methodology using these computational 
meshes are shown in Figures \ref{Fig:TransientDMP_QP_code_conc_Q4} 
and \ref{Fig:TransientDMP_QP_code_conc_T3}. Clearly, the proposed 
methodology satisfies the maximum principle and the non-negative 
constraint at all time levels. 
%
Figure \ref{Fig:TransientDMP_BE_lumped} clearly shows that 
the approach of employing the backward Euler time stepping 
scheme with lumped capacity matrix is not sufficient to meet 
the maximum principle and the non-negative constraint in 
the case of transient anisotropic diffusion. This approach 
will work in the case of transient isotropic diffusion 
provided some restrictions on the mesh are met, which 
is discussed briefly in Appendix of this paper. 

\subsection{Diffusion in heterogeneous anisotropic medium}
\label{Subsec:TransientDMP_2D_heterogeneous}
This problem considers transient diffusion in a bi-unit 
square domain with heterogeneous anisotropic diffusivity. 
Homogeneous Dirichlet boundary condition is applied on 
the entire boundary. The initial concentration is taken 
to be zero (i.e., $c_{0}(\mathbf{x}) = 0$). The volumetric 
source is taken as follows:
\begin{align}
  f(\mathbf{x},t) = \left\{\begin{array}{ll} 
  1 & \mathrm{if} \; (\mathrm{x},\mathrm{y}) 
  \in [3/8,5/8]^2 \\
  0 & \mathrm{otherwise} \end{array} \right.
\end{align}
The diffusivity tensor is taken as follows:
\begin{align}
  \mathbf{D}(\mathbf{x}) = 
  \left(\begin{array}{cc}
    \mathrm{y}^{2} + \epsilon \mathrm{x}^2 
    & -(1 - \epsilon)\mathrm{xy} \\
    -(1 - \epsilon)\mathrm{xy} 
    & \epsilon \mathrm{y}^2 + \mathrm{x}^2
  \end{array}\right)
\end{align}
with $\epsilon = 0.001$. 
Note that the diffusivity tensor is positive 
  definite in the open set $\Omega:= (0,1) \times (0,1)$. 
  This diffusivity tensor is widely used to test the 
  robustness of numerical formulations in the context 
  of maximum principles (e.g., Le Potier 
  \cite{LePotier_CRM_2005_v341_p787}).

Four-node quadrilateral finite elements are employed in the 
numerical simulation. The numerical results are generated 
for two different meshes (XSeed = YSeed = 51 and 101). 
Note that XSeed and YSeed denote the number of nodes 
along the x-direction and y-direction, respectively. 
Figure \ref{Fig:TransientDMP_2D_heterogeneous_mesh} 
shows a typical computational mesh used in the 
numerical simulations with XSeed = YSeed = 51. 
Various time steps ($\Delta t = 0.05, \; 0.1, \; 0.5$ 
and $1$) are employed in the numerical simulations. 
The rationale behind the choice of the time steps 
is that the transient solution is very close to the 
steady-state response for times greater than 2. Hence, 
any time step bigger than the ones used in the numerical 
simulations does not capture the transient features of 
the problem, and will not be appropriate for a transient 
analysis. Any smaller time step will result in bigger 
violation of the non-negative constraint, which will 
be evident from the numerical results presented in 
this paper and which is also reported in Reference 
\cite{Harari_CMAME_2004_v193_p1491}. 

For the aforementioned parameters, the variation of the 
minimum concentration with time under the single-field 
formulation is shown in Figure 
\ref{Fig:TransientDMP_2D_heterogeneous_min_conc}. The 
proposed methodology produced non-negative values for 
the concentration under all the considered cases, and 
the minimum concentration is zero. On the other hand, 
the variations of maximum concentration with time 
under the single-field formulation and the proposed 
methodology are similar, which is illustrated in Figure 
\ref{Fig:TransientDMP_2D_heterogeneous_max_conc}. 

Figure \ref{Fig:TransientDMP_2D_heterogeneous_contours} 
compares the contours of the concentration obtained using 
the single-field formulation and the proposed methodology 
for XSeed = 51 and $\Delta t = 0.5$. Even for a problem 
involving transient diffusion in a heterogeneous 
anisotropic medium, the proposed methodology did 
not violate the non-negative constraint. Figure 
\ref{Fig:TransientDMP_2D_heterogeneous_Elapsed_time} 
compares the elapsed computational time of the 
proposed methodology with that of the Galerkin 
single-field formulation. The elapsed time 
is measured using tic-toc feature available 
in MATLAB \cite{MATLAB_2012a}. From this figure, 
one can conclude that the additional cost incurred 
by the proposed methodology in meeting maximum 
principles and the non-negative is negligible.

\section{CONCLUDING REMARKS}
\label{Sec:TransientDMP_Conclusions}
We have presented a novel methodology for transient anisotropic 
diffusion equations that satisfies maximum principles and the 
non-negative constraint on computational grids with \emph{no 
additional restrictions on the time step}. The methodology 
has been developed using the method of horizontal lines, and 
techniques from convex programming. 
We have shown that the semi-discrete procedure based on the 
standard single-field formulation gives unphysical negative 
concentrations and violates maximum principles. Using several 
representative numerical examples we have shown that the proposed 
methodology satisfies maximum principles and the non-negative 
constraint on general computational grids with anisotropic and 
heterogeneous diffusion. The proposed methodology performs 
gives physically meaningful non-negative concentrations even 
on coarse computational grids and for small time steps.
\emph{
We shall conclude the paper by discussing two possible 
future research endeavors in the area of discrete maximum 
principles. We also briefly outline potential challenges 
one may have to overcome in addressing these research 
problems.}
\begin{enumerate}[(i)]
\item A possible future work is to incorporate advection in 
  addition to diffusion, and devise a non-negative methodology 
  for both steady-state and transient advection-diffusion 
  equation. However, one cannot directly implement the 
  procedure presented in this paper and in references 
  \cite{Nakshatrala_Valocchi_JCP_2009_v228_p6726, 
  Nagarajan_Nakshatrala_IJNMF_2011_v67_p820} for 
  advection-diffusion equation, as the advection 
  term makes the spatial differential operator 
  non-self-adjoint.
\item Another interesting research problem is to devise a 
  non-negative methodology for both steady and transient 
  \emph{nonlinear} diffusion-type equations. The obvious 
  challenges will be handling nonlinearity, and to ensure 
  that the computational cost in obtaining non-negative 
  solutions is not prohibitively expensive. 
\end{enumerate}

\section{APPENDIX}
\label{Sec:TransientDMP_Conclusions}
We now discuss other possible ways of implementing the 
methods of horizontal and vertical lines for transient 
diffusion-type equations. We will also provide reasons 
why these approaches may not satisfy maximum principles 
and the non-negative constraint. This discussion will 
shed light on the rationale behind the proposed methodology, 
and can guide future efforts in developing robust solvers 
for other important parabolic partial differential equations 
(e.g., transient diffusive-reactive systems).
All the approaches presented in this appendix employ 
trapezoidal family of time integrators, which can be 
written as follows:
\begin{align}
  \label{Eqn:Transient_DMP_trapezoidal}
  \boldsymbol{c}^{(n+1)} = \boldsymbol{c}^{(n)} 
  + \Delta t \left((1 - \gamma) \boldsymbol{v}^{(n)} 
  + \gamma \boldsymbol{v}^{(n+1)}\right)
\end{align}
where $\gamma \in [0,1]$. (Recall that the parameter $\gamma$ 
used in Section \ref{Sec:TransientDMP_Derivation} is different 
from the parameter in trapezoidal family of time integrators.)
The discussion and conclusions in this appendix will 
hinge on the following result from Matrix Algebra. 
Given any vector $\boldsymbol{b} \succeq \boldsymbol{0}$, 
the solution of a system of linear equations of the form 
\begin{align}
  \boldsymbol{A} \boldsymbol{x} = \boldsymbol{b}
\end{align}
will be non-negative (i.e., $\boldsymbol{x} \succeq 
\boldsymbol{0}$) \emph{if and only} if the matrix 
$\boldsymbol{A}$ is a monotone. (Recall that $\succeq$ 
denotes component-wise inequality.) A matrix is called 
a monotone if the matrix is invertible and all the 
entries of its inverse are non-negative. For further 
details on monotone matrices refer to the classic 
texts \cite{Fiedler,Berman_Plemmons,Saad}.

\subsection{Method of vertical lines at integral time steps}
In this paper, this method is referred to as the 
\emph{standard single-field formulation}.
This is the most commonly used method for solving transient 
diffusion equation, and can be found in many introductory 
texts on finite element methods (e.g., \cite{Hughes,Reddy,
Zienkiewicz}).
The method is based on standard semi-discrete methodology 
and Galerkin formalism. The corresponding weak form reads: 
Find $c(\mathbf{x},t) \in \mathcal{P}_t$ such that we have
\begin{align}
  \label{Eqn:TransientDMP_single_field_formulation}
  \int_{\Omega} w(\mathbf{x}) \frac{\partial 
    c(\mathbf{x},t)}{\partial t} \; \mathrm{d} 
  \Omega 
  &+\int_{\Omega} \mathrm{grad}[w(\mathbf{x})] 
  \cdot \mathbf{D}(\mathbf{x}) \mathrm{grad}[c(\mathbf{x},t)] 
  \; \mathrm{d} \Omega \nonumber \\
  &= \int_{\Omega} w(\mathbf{x}) \; f(\mathbf{x},t) \; 
  \mathrm{d} \Omega + \int_{\Gamma^{\mathrm{N}}} w(\mathbf{x}) 
  \; q_p(\mathbf{x},t) \; \mathrm{d} \Gamma \quad \forall 
  w(\mathbf{x}) \in \mathcal{Q}
\end{align}
where 
\begin{align}
  \mathcal{P}_t &:= \{c(\mathbf{x},t) \in H^{1}(\Omega) \; 
  \big| \; c(\mathbf{x},t) = c_p(\mathbf{x},t) \; 
  \mathrm{on} \; \Gamma^{\mathrm{D}}\} 
\end{align}
and the function space $\mathcal{Q}$ is defined previously 
in equation \eqref{Eqn:Transient_DMP_Q_space}. After spatial 
discretization using the finite element method, one obtains 
a system of ordinary differential equations of following 
form:
\begin{align}
  \boldsymbol{C} \frac{d \boldsymbol{c}(t)}{dt} + 
  \boldsymbol{K} \boldsymbol{c}(t) = \boldsymbol{f}(t)
\end{align}
The capacity matrix $\boldsymbol{C}$ is symmetric and 
positive definite, and all the entries of the matrix 
are non-negative. The matrix $\boldsymbol{K}$ is symmetric 
and positive semi-definite. More importantly, the matrix 
$\boldsymbol{K}$ will not be a monotone if the medium 
(i.e., the diffusion process) is not isotropic. (If the 
medium is isotropic, it is easy to check that the matrix 
$\boldsymbol{K}$ is diagonally dominant, and hence it will 
be a monotone matrix.) 
If a time stepping scheme from the trapezoidal family 
is employed to solve the above ordinary differential 
equations, one can obtain a system of linear equations 
of the following form:
\begin{align}
  \label{Eqn:Transient_DMP_trapezoidal_linear}
  \left(\frac{1}{\gamma \Delta t} \boldsymbol{C} + 
  \boldsymbol{K}\right) \boldsymbol{c}^{(n+1)} = 
  \boldsymbol{f}^{(n+1)} + \frac{1}{\gamma \Delta t} 
  \boldsymbol{C} \left(\boldsymbol{c}^{(n)} + \Delta 
  t (1 - \gamma) \boldsymbol{v}^{(n)}\right)
\end{align}

There are two potential scenarios that can contribute 
to the violation of the non-negative constraint and 
maximum principle under the method of vertical lines 
at integral time steps. 
Firstly, the vector on the right side of equation 
\eqref{Eqn:Transient_DMP_trapezoidal_linear} need 
not be non-negative, as there is no physical constraint 
requiring that $\boldsymbol{v}^{(n)}$ should be non-negative. 
Even if the volumetric source is non-negative (i.e., 
$\boldsymbol{f}^{(n+1)} \succeq \boldsymbol{0}$), 
$\boldsymbol{c}^{(n)} \succeq \boldsymbol{0}$, 
$\gamma \geq 0$, $\Delta t> 0$, and all the 
entries of the capacity matrix are non-negative; 
the resulting vector on the right side of the 
above equation need not be non-negative. One 
possible exception is when $\gamma = 1$ (that 
is, when the backward Euler is employed).
Secondly, the matrix on the left side of equation 
\eqref{Eqn:Transient_DMP_trapezoidal_linear} may not be a 
monotone. Even for an isotropic medium, the matrix will be 
monotone \emph{only} if the time step is greater than a 
critical time step or by employing lumped capacity matrix. 
Based on the above discussion, the sufficient conditions 
for the method of vertical lines at integral time levels 
to satisfy maximum principles and the non-negative 
constraint are as follows:
\begin{itemize}
\item isotropic diffusion, 
\item low-order finite elements, 
\item backward Euler scheme (i.e., $\gamma = 1$), 
\item lumped capacity matrix, 
\item select a time step \emph{greater} than the 
  critical time step, and 
\item place constraints on the mesh and element shapes (e.g., 
  well-centered triangular elements, rectangular elements 
  with aspect ratio between $1/\sqrt{2}$ and $\sqrt{2}$). 
\end{itemize}
It is important to note that the above conditions are too 
restrictive to be able to obtain physically meaningful 
results for practical problems. But this method is 
commonly employed in many numerical simulations, and 
in many commercial finite element packages. Few other 
remarks about this method are in order. 

\begin{remark}
  For a discussion on necessary constraints on a 
  finite element mesh to satisfy maximum principles 
  and the non-negative constraint, see references 
  \cite{Ciarlet_Raviart_CMAME_1973_v2_p17, 
    Christie_Hall_IJNME_1984_v20_p549,
    Horvath_IJCMA_2008_v55_p2306,
    Nakshatrala_Valocchi_JCP_2009_v228_p6726,
    Nagarajan_Nakshatrala_IJNMF_2011_v67_p820}. 
  However, all these constraints are for isotropic 
  diffusion. It is noteworthy that, in the case of 
  anisotropy, a computational mesh may not even 
  exist that will ensure the satisfaction of 
  maximum principles and the non-negative 
  constraint. 
\end{remark}

\begin{remark}
  Several studies derived critical time steps with 
  respect to maximum principles. For example, see 
  references \cite{Thomas_Zhou_CNME_1998_p809,
    Ilinca_Hetu_CMAME_2002_v191_p3073}. But these 
  derivations for critical time steps are restricted 
  to one-dimensional problems, isotropic diffusion, 
  and backward Euler.
\end{remark}

\begin{remark}
  It is noteworthy that there is no obvious way of modifying 
  the non-negative formulations that has been shown recently shown 
  to be successful for steady-state diffusion equations (e.g., 
  see references \cite{Nakshatrala_Valocchi_JCP_2009_v228_p6726,
    Nagarajan_Nakshatrala_IJNMF_2011_v67_p820}) to obtain a 
  non-negative formulation for transient diffusion equation 
  under the method of vertical lines at integral time steps. 
  This is the reason why this method has not been considered 
  as the basis in Section \ref{Sec:TransientDMP_Derivation}.
\end{remark}

\subsection{Method of horizontal lines at integral time steps}
By applying the method of horizontal lines at integral time levels and 
eliminating $v^{(n+1)}(\mathbf{x})$ using the time discretization of 
trapezoidal family given by equation \eqref{Eqn:Transient_DMP_trapezoidal}, 
the time discretized equations take the following form:
\begin{subequations}
  \label{Eqn:Transient_DMP_MHL_trapezoidal}
  \begin{align}
    \label{Eqn:Transient_DMP_MHL_trapezoidal_GE}
    &\frac{1}{\gamma \Delta t} c^{(n+1)}(\mathbf{x}) - 
    \mathrm{div}[\mathbf{D}(\mathbf{x}) \mathrm{grad}
      [c^{(n+1)}]] = f^{(n+1)}(\mathbf{x}) + \frac{1}
           {\gamma \Delta t}\left(c^{(n)}(\mathbf{x}) + 
           (1 - \gamma) \Delta t v^{(n)}(\mathbf{x}) \right) 
    \; \mathrm{in} \; \Omega \\
    &c^{(n+1)}(\mathbf{x}) = c_p^{(n+1)}(\mathbf{x}) \quad 
    \mathrm{on} \; \Gamma^{\mathrm{D}} \\
    \label{Eqn:Transient_DMP_MHL_trapezoidal_Neumann}
    &\mathbf{\hat{n}}(\mathbf{x}) \cdot \mathbf{D}(\mathbf{x})
    \mathrm{grad}[c^{(n+1)}(\mathbf{x})] = q_p^{(n+1)}
    (\mathbf{x}) \quad \mathrm{on} \; \Gamma^{\mathrm{N}}
  \end{align}
\end{subequations}
In going from equations 
\eqref{Eqn:TransientDMP_PDE}--\eqref{Eqn:TransientDMP_IC} to equations 
\eqref{Eqn:Transient_DMP_MHL_trapezoidal_GE}--\eqref{Eqn:Transient_DMP_MHL_trapezoidal_Neumann}, 
the temporal discretization may not preserve the non-negative constraint, which should be 
interpreted in the following sense. One may not get a non-negative solution under equations 
\eqref{Eqn:Transient_DMP_MHL_trapezoidal_GE}--\eqref{Eqn:Transient_DMP_MHL_trapezoidal_Neumann} 
even when the solution to the original time continuous problem given by equations 
\eqref{Eqn:TransientDMP_PDE}--\eqref{Eqn:TransientDMP_IC} is non-negative. This is 
again due to the fact that the right side of equation \eqref{Eqn:Transient_DMP_MHL_trapezoidal_GE} 
can be negative, as there is no physical constraint requiring 
that the rate of concentration $v^{(n)}(\mathbf{x})$ should be 
non-negative. However, it does not mean that the time discrete 
equation does not satisfy maximum principles and the non-negative 
equation. The above equation is diffusion with decay, and as 
mentioned earlier, this equation also satisfies maximum principles 
and the non-negative constraint. But, the requirement for the 
non-negative constraint is that $f^{(n+1)}(\mathbf{x}) + 
\frac{1}{\gamma \Delta t} \left(c^{(n)}(\mathbf{x}) + 
(1 - \gamma) \Delta t v^{(n)}(\mathbf{x})\right) \geq 0$.

\subsection{Method of horizontal lines at weighted time levels}
We shall perform temporal discretization at the weighted time 
level $t_{n+\gamma}$, which gives rise to the following equations:
\begin{subequations}
  \begin{align}
    &\frac{1}{\gamma \Delta t} c^{(n+\gamma)}(\mathbf{x}) - 
    \mathrm{div}[\mathbf{D}(\mathbf{x})\mathrm{grad}
      [c^{(n+\gamma)}]] = f^{(n+\gamma)}(\mathbf{x}) + 
    \frac{1}{\gamma \Delta t} c^{(n)}(\mathbf{x})
    \quad \mathrm{in} \; \Omega \\
    &c^{(n+\gamma)}(\mathbf{x}) = c_{p}(\mathbf{x},t_{n+\gamma}) 
    \quad \mathrm{on} \; \Gamma^{\mathrm{D}} \\
    &\mathbf{\hat{n}}(\mathbf{x}) \cdot \mathbf{D}(\mathbf{x}) 
    \mathrm{grad}[c^{(n+\gamma)}] = q_{p}^{(n+\gamma)}
    (\mathbf{x}) \quad \mathrm{on} \; \Gamma^{\mathrm{N}} 
  \end{align}
\end{subequations}
One can obtain nodal concentrations at weighted time levels (i.e., 
$\boldsymbol{c}^{(n+\gamma)}$) by employing the optimization-based 
solver presented in Section \ref{Sec:TransientDMP_Derivation}.
Noting the results presented in Reference 
\cite{Nakshatrala_Prakash_Hjelmstad_JCP_2009_v228_p7957} on 
stability issues associated with numerical time integration 
of differential-algebraic equations, the concentration at 
integral time levels is approximated in terms of corresponding 
quantities at weighted time levels. The interpolation scheme 
is pictorially described in Figure 
\ref{Fig:Transient_DMP_weighted_quantities}, and can be 
mathematically written as follows: 
\begin{align}
  \label{Eqn:Transient_DMP_integral_in_terms_of_weights}
  \boldsymbol{c}^{(n+1)} &= \gamma \boldsymbol{c}^{(n + \gamma)} 
  + (1 - \gamma) \boldsymbol{c}^{(n + 1 + \gamma)}
\end{align}
The rate of concentration at weighted time levels 
can be calculated as follows:
\begin{align}
  \boldsymbol{v}^{(n+\gamma)} = \frac{\boldsymbol{c}^{(n+1)} 
    - \boldsymbol{c}^{(n)}}{\Delta t} 
\end{align}
The corresponding quantity at integral time levels 
are calculated as follows: 
\begin{align}
  \label{Eqn:Transient_DMP_interpolation_rate}
  \boldsymbol{v}^{(n+1)} = \gamma \boldsymbol{v}^{(n+\gamma)} 
  + (1 - \gamma) \boldsymbol{v}^{(n+1+\gamma)}
\end{align}
The interpolation given by equation 
\eqref{Eqn:Transient_DMP_integral_in_terms_of_weights} is 
different from the usual way of interpolating the quantities 
at weighted time levels in terms of integral time levels. 
That is, 
\begin{align}
  \boldsymbol{c}^{(n+\gamma)} = (1 - \gamma) \boldsymbol{c}^{(n)} 
  + \gamma \boldsymbol{c}^{(n+1)}
\end{align}
Figure \ref{Fig:Transient_DMP_weighted_quantities} compares both 
these interpolation schemes. The only drawback of the method 
presented in this subsection is that it is not self-starting, 
as we do not have $\boldsymbol{c}^{(n-1+\gamma)}$ when $n = 1$ 
unless $\gamma = 1$. But this drawback can be easily overcome 
by employing the backward Euler scheme (i.e., $\gamma = 1$) 
for the first time level, and then employ the method for 
subsequent time levels. Therefore, the method presented 
in this subsection can be considered as an alternate to 
the method presented in Section 
\ref{Sec:TransientDMP_Derivation} to satisfy maximum 
principles and the non-negative constraint for transient 
diffusion-type equations. 

\begin{figure}
  \psfrag{gm}{$\gamma \Delta t$}
  \psfrag{gm1}{$(1- \gamma)\Delta t$}
  \psfrag{tn}{$t_{n}$}
  \psfrag{tnm}{$t_{n-1}$}
  \psfrag{tnp}{$t_{n+1}$}
  \psfrag{tngamma}{$t_{n+\gamma}$}
  \psfrag{tnmgamma}{$t_{n - 1 + \gamma}$}
  \psfrag{cn}{$c^{(n)}$}
  \psfrag{cnm}{$c^{(n-1)}$}
  \psfrag{cnp}{$c^{(n+1)}$}
  \psfrag{cngamma}{$c^{(n+\gamma)}$}
  \psfrag{cnmgamma}{$c^{(n - 1 + \gamma)}$}
  \includegraphics[scale=0.5]{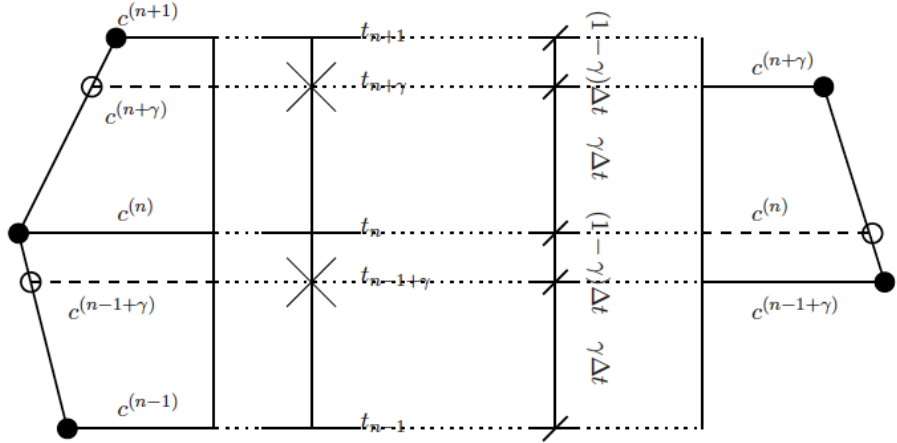}
  \caption{The left part of the figure shows the usual way of 
    interpolating quantities at integral time levels to obtain 
    the corresponding quantities at weighted time levels. That 
    is, $c^{(n+\gamma)} = (1 - \gamma) c^{(n)} + \gamma c^{(n+1)}$. 
    The right part of the figure shows the interpolation of 
    quantities at weighted time levels to obtain the corresponding 
    quantities at integral time levels, which is adopted in this 
    paper. That is, $c^{(n)} = \gamma c^{(n - 1 + \gamma)} + (1 - 
    \gamma) c^{(n+\gamma)}$. The interpolated quantities are 
    indicated using hollow circles. 
    \label{Fig:Transient_DMP_weighted_quantities}}
\end{figure}

\section*{ACKNOWLEDGMENTS}
K.B.N. and M.S. acknowledge the funding received from the 
DOE Office of Nuclear Energy's Nuclear Energy University 
Programs (NEUP).  The first author (K.B.N.) also acknowledges 
the support from the National Science Foundation under Grant 
No. CMMI 1068181. The opinions expressed in this paper are 
those of the authors and do not necessarily reflect that of the 
sponsors.

\bibliographystyle{plain}
\bibliography{Master_References/Master_References,Master_References/Books}

\begin{thebibliography}{10}

\bibitem{gmsh.www}
{\em {Gmsh}: A three-dimensional finite element mesh generator with pre- and
  post-processing facilities}.
\newblock URL: http://www.geuz.org/gmsh/.

\bibitem{Tecplot360}
{\em {Tecplot 360: User's Manual}}.
\newblock URL: http://www.tecplot.com, Bellevue, Washington, USA, 2008.

\bibitem{GAMS}
{\em {General Algebraic Modeling System (GAMS)}}.
\newblock Version 23.8, GAMS Development Corporation, Washington DC, USA, 2012.

\bibitem{MATLAB_2012a}
{\em {MATLAB 2012a}}.
\newblock The MathWorks, Inc., Natick, Massachusetts, USA, 2012.

\bibitem{Dakota}
B.~M. Adams, W.~J. Bohnhoff, K.~R. Dalbey, J.~P. Eddy, M.~S. Eldred, D.~M. Gay,
  K.~Haskell, P.~D. Hough, and L.~P. Swiler.
\newblock {\em {DAKOTA, A Multilevel Parallel Object-Oriented Framework for
  Design Optimization, Parameter Estimation, Uncertainty Quantification, and
  Sensitivity Analysis: Version 5.2 User's Manual}}.
\newblock Sandia Technical Report SAND2010-2183, 2011.

\bibitem{Ascher_Petzold}
U.~M. Ascher and L.~R. Petzold.
\newblock {\em {Computer {M}ethods for {O}rdinary {D}ifferential {E}quations
  and {D}ifferential-{A}lgebraic {E}quations}}.
\newblock SIAM, Philadelphia, USA, 1998.

\bibitem{Bartle_Sherbert}
R.~G. Bartle and D.~R. Sherbert.
\newblock {\em {Introduction to Real Analysis}}.
\newblock John Wiley \& Sons, Inc., New Jersey, USA, 2011.

\bibitem{Berman_Plemmons}
A.~Berman and R.~J. Plemmons.
\newblock {\em {Nonnegative Matrices in the Mathematical Sciences}}.
\newblock Classics in Applied Mathematics, Society for Industrial and Applied
  Mathematics, Philadelphia, Pennsylvania, 1987.

\bibitem{Berzins_CNME_2001_p659}
M.~Berzins.
\newblock {Modified mass matrices and positivity preservation for hyperbolic
  and parabolic PDEs}.
\newblock {\em Communications in Numerical Methods in Engineering},
  17:659--666, 2001.

\bibitem{Bornemann_ICSE_1990_v02_p279}
F.~A. Bornemann.
\newblock {An adaptive multilevel approach to parabolic equations I.~General
  theory and 1D implementation}.
\newblock {\em Impact of Computing in Science and Engineering}, 2:279--317,
  1990.

\bibitem{Boyd_convex_optimization}
S.~Boyd and L.~Vandenberghe.
\newblock {\em {Convex Optimization}}.
\newblock Cambridge University Press, Cambridge, UK, 2004.

\bibitem{Carslaw_Jaeger}
H.~S. Carslaw and J.~C. Jaeger.
\newblock {\em {Conduction of Heat in Solids}}.
\newblock Oxford University Press, New York, USA, second edition, 1986.

\bibitem{Cattaneo_CR_1958_v247_p431}
C.~Cattaneo.
\newblock {Sur une forme de l'\'equation de la chaleur \'eliminant le paradoxe
  d'une propagation instantan\'ee}.
\newblock {\em Comptes Rendus}, 247:431--433, 1958.

\bibitem{Chapko_Kress_JIEA_1997_v09_p47}
R.~Chapko and R.~Kress.
\newblock {Rothe's method for the heat equation and boundary integral
  equations}.
\newblock {\em Journal of Integral Equations and Applications}, 09:47--69,
  1997.

\bibitem{Chen_Thomee_JAMS_1985_v26_p329}
C.~M. Chen and V.~Thomee.
\newblock {The lumped mass finite element method for a parabolic problem}.
\newblock {\em Journal of the Australian Mathematical Society}, 26:329--354,
  1985.

\bibitem{Christie_Hall_IJNME_1984_v20_p549}
I.~Christie and C.~Hall.
\newblock {The maximum principle for bilinear elements}.
\newblock {\em International Journal for Numerical Methods in Engineering},
  20:549--553, 1984.

\bibitem{Chung_Hulbert_JAM_1993_v60_p371}
J.~Chung and G.~M. Hulbert.
\newblock {A time integration algorithm for structural dynamics with improved
  numerical dissipation: The generalized-$\alpha$ method}.
\newblock {\em Journal of Applied Mechanics}, 60:371--375, 1993.

\bibitem{Ciarlet_Raviart_CMAME_1973_v2_p17}
P.~G. Ciarlet and P-A. Raviart.
\newblock {Maximum principle and uniform convergence for the finite element
  method}.
\newblock {\em Computer Methods in Applied Methods and Engineering}, 2:17--31,
  1973.

\bibitem{Crank_1980}
J.~Crank.
\newblock {\em {The Mathematics of Diffusion}}.
\newblock Oxford University Press, New York, USA, second edition, 1980.

\bibitem{Douglas_Dupont_SIAMJNA_1970_v07_p575}
J.~Douglas and T.~Dupont.
\newblock {Galerkin methods for parabolic equations}.
\newblock {\em SIAM Journal on Numerical Analysis}, 07:575--626, 1970.

\bibitem{Elshebli_AMM_2008_v32_p1530}
M.~A.~T. Elshebli.
\newblock {Discrete maximum principle for the finite element solution of linear
  non-stationary diffusionÐreaction problems}.
\newblock {\em Applied Mathematical Modeling}, 32:1530--1541, 1998.

\bibitem{Evans_PDE}
L.~C. Evans.
\newblock {\em {P}artial {D}ifferential {E}quations}.
\newblock American Mathematical Society, Providence, Rhode Island, USA, 1998.

\bibitem{Farago_Horvath_Korotov_ANM_2005_v53_p249}
I.~Farago, R.~Horvath, and S.~Korotov.
\newblock Discrete maximum principle for linear parabolic problems solved on
  hybrid meshes.
\newblock {\em Applied Numerical Mathematics}, 53:249--264, 2005.

\bibitem{Fiedler}
M.~Fiedler.
\newblock {\em {Special Matrices and Their Applications in Numerical
  Mathematics}}.
\newblock Martinus Nijhoff Publishers, Dordrecht, The Netherlands, 1986.

\bibitem{Gilbarg_Trudinger}
D.~Gilbarg and N.~S. Trudinger.
\newblock {\em {Elliptic Partial Differential Equations of Second Order}}.
\newblock Springer, New York, USA, 2001.

\bibitem{Gurtin_Pikin_ARMA_1968_v31_p252}
M.~E. Gurtin and A.~C. Pipkin.
\newblock {A general theory of heat conduction with finite speed}.
\newblock {\em Archive for Rational Mechanics and Analysis}, 31:113--126, 1968.

\bibitem{Hairer_Lubich_Roche}
E.~Hairer, C.~Lubich, and M.~Roche.
\newblock {\em {The Numerical Solution of Differential-Algebraic Systems by
  Runge-Kutta Methods}}.
\newblock Lecture Notes in Mathematics. Springer-Verlag, New York, USA, 1989.

\bibitem{Hairer_Wanner}
E.~Hairer and G.~Wanner.
\newblock {\em {Solving Ordinary Differential Equations II: Stiff and
  Differential-Algebraic Problems}}.
\newblock Springer-Verlag, New York, USA, 1996.

\bibitem{Harari_CMAME_2004_v193_p1491}
I.~Harari.
\newblock {Stability of semidiscrete formulations for parabolic problems at
  small time steps}.
\newblock {\em Computer Methods in Applied Mechanics and Engineering},
  193:1491--1516, 2004.

\bibitem{Herrera_Valocchi_GW_2006_v44_p803}
P.~Herrera and A.~Valocchi.
\newblock {Positive solution of two-dimensional solute transport in
  heterogeneous aquifers}.
\newblock {\em Ground Water}, 44:803--813, 2006.

\bibitem{Horvath_IJCMA_2008_v55_p2306}
R.~Horvath.
\newblock {Sufficient conditions of the discrete maximum-minimum principle for
  parabolic problems on rectangular meshes}.
\newblock {\em International Journal of Computers and Mathematics with
  Applications}, 55:2306--2317, 2008.

\bibitem{Hughes}
T.~J.~R. Hughes.
\newblock {\em {The {F}inite {E}lement {M}ethod: {L}inear {S}tatic and
  {D}ynamic {F}inite {E}lement {A}nalysis}}.
\newblock Prentice-Hall, Englewood Cliffs, New Jersey, USA, 1987.

\bibitem{Ignaczak_Starzewski}
J.~Ignaczak and M.~O. Starzewski.
\newblock {\em {Thermoelasticity with Finite Wave Speeds}}.
\newblock Oxford Science Publications, New York, USA, 2009.

\bibitem{Ilinca_Hetu_CMAME_2002_v191_p3073}
F.~Ilinca and J.~F. Hetu.
\newblock {Galerkin gradient least-squares formulations for transient
  conduction heat transfer}.
\newblock {\em Computer Methods in Applied Mechanics and Engineering},
  191:3073--3097, 2002.

\bibitem{Jansen_Whiting_Hulbert_CMAME_2000_v190_p305}
K.~E. Jansen, C.~H. Whiting, and G.~H. Hulbert.
\newblock {A generalized-$\alpha$ method for integrating the filtered
  Navier-Stokes equations with a stabilized finite element method}.
\newblock {\em Computer Methods in Applied Mechanics and Engineering},
  190:305--319, 2000.

\bibitem{Kittel_Kroemer}
C.~Kittel and H.~Kroemer.
\newblock {\em {Thermal Physics}}.
\newblock W.~H.~Freeman and Company, New York, USA, 1980.

\bibitem{Lang_Walter_ANM_1993_v13_p135}
J.~Lang and A.~Walter.
\newblock {An adaptive Rothe method for nonlinear reaction-diffusion systems}.
\newblock {\em Applied Numerical Mathematics}, 13:135--146, 1993.

\bibitem{Levi_AMPA_1908_v14_p187}
E.~E. Levi.
\newblock {Sull' equazione del calore}.
\newblock {\em Annali di Matematica Pura ed Applicata}, 14:187--264, 1908.

\bibitem{Lipnikov_Manzini_Svyatskiy_JCP_2011_v230_p2620}
K.~Lipnikov, G.~Manzini, and D.~Svyatskiy.
\newblock {Analysis of the monotonicity conditions in the mimetic finite
  difference method for elliptic problems}.
\newblock {\em Journal of Computational Physics}, 230:2620--2642, 2011.

\bibitem{Lipnikov_Shashkov_Svyatskiy_Vassilevski_JCP_2007_v227_p492}
K.~Lipnikov, M.~Shashkov, D.~Svyatskiy, and Y.~Vassilevski.
\newblock {Monotone finite volume schemes for diffusion equations on
  unstructured triangular and shape-regular polygonal meshes}.
\newblock {\em Journal of Computational Physics}, 227:492--512, 2007.

\bibitem{Lipnikov_Svyatskiy_Vassilevski_JCP_2009_v228_p703}
K.~Lipnikov, D.~Svyatskiy, and Y.~Vassilevski.
\newblock {Interpolation-free monotone finite volume method for diffusion
  equations on polygonal meshes}.
\newblock {\em Journal of Computational Physics}, 228:703--716, 2009.

\bibitem{Lipnikov_Svyatskiy_Vassilevski_JCP_2010_v229_p4017}
K.~Lipnikov, D.~Svyatskiy, and Y.~Vassilevski.
\newblock {A monotone finite volume method for advectionÐdiffusion equations on
  unstructured polygonal meshes}.
\newblock {\em Journal of Computational Physics}, 229:4017--4032, 2010.

\bibitem{Liska_Shashkov_CiCP_2008_v3_p852}
R.~Liska and M.~Shashkov.
\newblock {Enforcing the discrete maximum principle for linear finite element
  solutions for elliptic problems}.
\newblock {\em Communications in Computational Physics}, 3:852--877, 2008.

\bibitem{Maxwell_PTRSL_1866_vA157_p26}
J.~C. Maxwell.
\newblock On the dynamical theory of gases.
\newblock {\em Philosophical Transactions of Royal Society of London},
  A157:26--78, 1866.

\bibitem{McOwen}
R.~McOwen.
\newblock {\em {P}artial {D}ifferential {E}quations: {M}ethods and
  {A}pplications}.
\newblock Prentice Hall, New Jersey, USA, 1996.

\bibitem{Mizukami_CMAME_1986_v59_p101}
A.~Mizukami.
\newblock {Variable explicit finite element methods for unsteady heat
  conduction equations}.
\newblock {\em Computer Methods in Applied Mechanics and Engineering},
  59:101--109, 1986.

\bibitem{Mudunuru_Nakshatrala_IJNME_2012_v89_p1144}
M.~K. Mudunuru and K.~B. Nakshatrala.
\newblock {A framework for coupled deformation-diffusion analysis with
  application to degradation/healing}.
\newblock {\em International Journal for Numerical Methods in Engineering},
  89:1144--1170, 2012.

\bibitem{multiphysics2012version}
COMSOL Multiphysics.
\newblock Version 4.3 a.
\newblock {\em COMSOL Inc, Burlington, MA}, 2012.

\bibitem{tao-user-ref}
T.~Munson, J.~Sarich, S.~Wild, S.~Benson, and L.~C. McInnes.
\newblock {TAO 2.0 Users Manual}.
\newblock Technical Report ANL/MCS-TM-322, Mathematics and Computer Science
  Division, Argonne National Laboratory, 2012.
\newblock http://www.mcs.anl.gov/tao.

\bibitem{Nagarajan_Nakshatrala_IJNMF_2011_v67_p820}
H.~Nagarajan and K.~B. Nakshatrala.
\newblock {Enforcing the non-negativity constraint and maximum principles for
  diffusion with decay on general computational grids}.
\newblock {\em International Journal for Numerical Methods in Fluids},
  67:820--847, 2011.

\bibitem{Nakshatrala_Prakash_Hjelmstad_JCP_2009_v228_p7957}
K.~B. Nakshatrala, A.~Prakash, and K.~D. Hjelmstad.
\newblock {On dual Schur domain decomposition method for linear first-order
  transient problems}.
\newblock {\em Journal of Computational Physics}, 228:7957--7985, 2009.

\bibitem{Nakshatrala_Valocchi_JCP_2009_v228_p6726}
K.~B. Nakshatrala and A.~J. Valocchi.
\newblock {Non-negative mixed finite element formulations for a tensorial
  diffusion equation}.
\newblock {\em Journal of Computational Physics}, 228:6726--6752, 2009.

\bibitem{Nirenberg_CPAM_1953_v6_p167}
L.~Nirenberg.
\newblock {A strong maximum principle for parabolic equations}.
\newblock {\em Communications on Pure and Applied Mathematics}, 6:167--177,
  1953.

\bibitem{Ozisik}
M.~N. Ozisik.
\newblock {\em {Heat Conduction}}.
\newblock John Wiley \& Sons, Inc., New York, USA, second edition, 1993.

\bibitem{Pang_CCE_1983_v7_p583}
J.-S. Pang.
\newblock {Methods for quadratic programming: A survey}.
\newblock {\em Computers and Chemical Engineering}, 5:583--594, 1983.

\bibitem{Pao}
C.~V. Pao.
\newblock {\em {Nonlinear Parabolic and Elliptic Equations}}.
\newblock Springer-Verlag, New York, USA, 1993.

\bibitem{Payette_Nakshatrala_Reddy_IJNME_2012}
G.~S. Payette, K.~B. Nakshatrala, and J.~N. Reddy.
\newblock {On the performance of high-order finite elements with respect to
  maximum principles and the non-negative constraint for diffusion-type
  equations}.
\newblock {\em International Journal for Numerical Methods in Engineering},
  91:742--771, 2012.

\bibitem{Petzold_SIAMJSciStatComp_1982_v3_p367}
L.~Petzold.
\newblock {Differential/algebraic equations are not ODEs}.
\newblock {\em SIAM Journal on Scientific and Statistical Computing},
  3:367--384, 1982.

\bibitem{Picone_AMPA_1929_v7_p145}
M.~Picone.
\newblock {Maggiorazione degli integrali delle equazioni totalmente paraboliche
  alle derivate parziali del secondo ordine}.
\newblock {\em Annali di Matematica Pura ed Applicata}, 7:145--192, 1929.

\bibitem{Porru_Serra_JAMS_1994_v56_p41}
G.~Porru and S.~Serra.
\newblock {Maximum principles for parabolic equations}.
\newblock {\em Journal of the Australian Mathematical Society}, 56:41--52,
  1994.

\bibitem{LePotier_CRM_2005_v341_p787}
C.~Le Potier.
\newblock {Finite volume monotone scheme for highly anisotropic diffusion
  operators on unstructured triangular meshes}.
\newblock {\em Comptes Rendus Mathematique}, 341:787--792, 2005.

\bibitem{Protter_Weinberger}
M.~H. Protter and H.~F. Weinberger.
\newblock {\em {Maximum Principles in Differential Equations}}.
\newblock Springer-Verlag, New York, USA, 1999.

\bibitem{Rank_Katz_Werner_IJNME_1983_v19_p1771}
E.~Rank, C.~Katz, and H.~Werner.
\newblock {On the importance of the discrete maximum principle in transient
  analysis using finite element methods}.
\newblock {\em International Journal for Numerical Methods in Engineering},
  19:1771--1782, 1983.

\bibitem{Reddy}
J.~N. Reddy.
\newblock {\em {Finite Element Method}}.
\newblock McGraw Hill, New York, USA, 1993.

\bibitem{Rothe_MA_1930_v102_p650}
E.~Rothe.
\newblock Zweidimensionale parabolische randwertaufgaben als grenzfall
  eindimensionaler randwertaufgaben.
\newblock {\em Mathematische Annalen}, 102:650--670, 1930.

\bibitem{Saad}
Y.~Saad.
\newblock {\em {Iterative Methods for Sparse Linear Systems}}.
\newblock SIAM, Philadelphia, USA, 2003.

\bibitem{Thomas_Zhou_CNME_1998_p809}
H.~R. Thomas and Z.~Zhou.
\newblock {An analysis of factors that govern the minimum time step size to be
  used in the finite element analysis of diffusion problems}.
\newblock {\em Communications in Numerical Methods in Engineering},
  14:809--819, 1998.

\bibitem{Ye_TSE_MP_1989_v44_p157}
Y.~Ye and E.~Tse.
\newblock {An extension of Karmarkar's projective algorithm for convex
  quadratic programming}.
\newblock {\em Mathematical Programming}, 44:157--179, 1989.

\bibitem{Zienkiewicz}
O.~C. Zienkiewicz and R.~L. Taylor.
\newblock {\em {The {F}inite {E}lement {M}ethod : Vol.1}}.
\newblock McGraw-Hill, New York, USA, 1989.

\end{thebibliography}

\clearpage
\newpage


\begin{figure}
  \includegraphics[clip,scale=0.4]{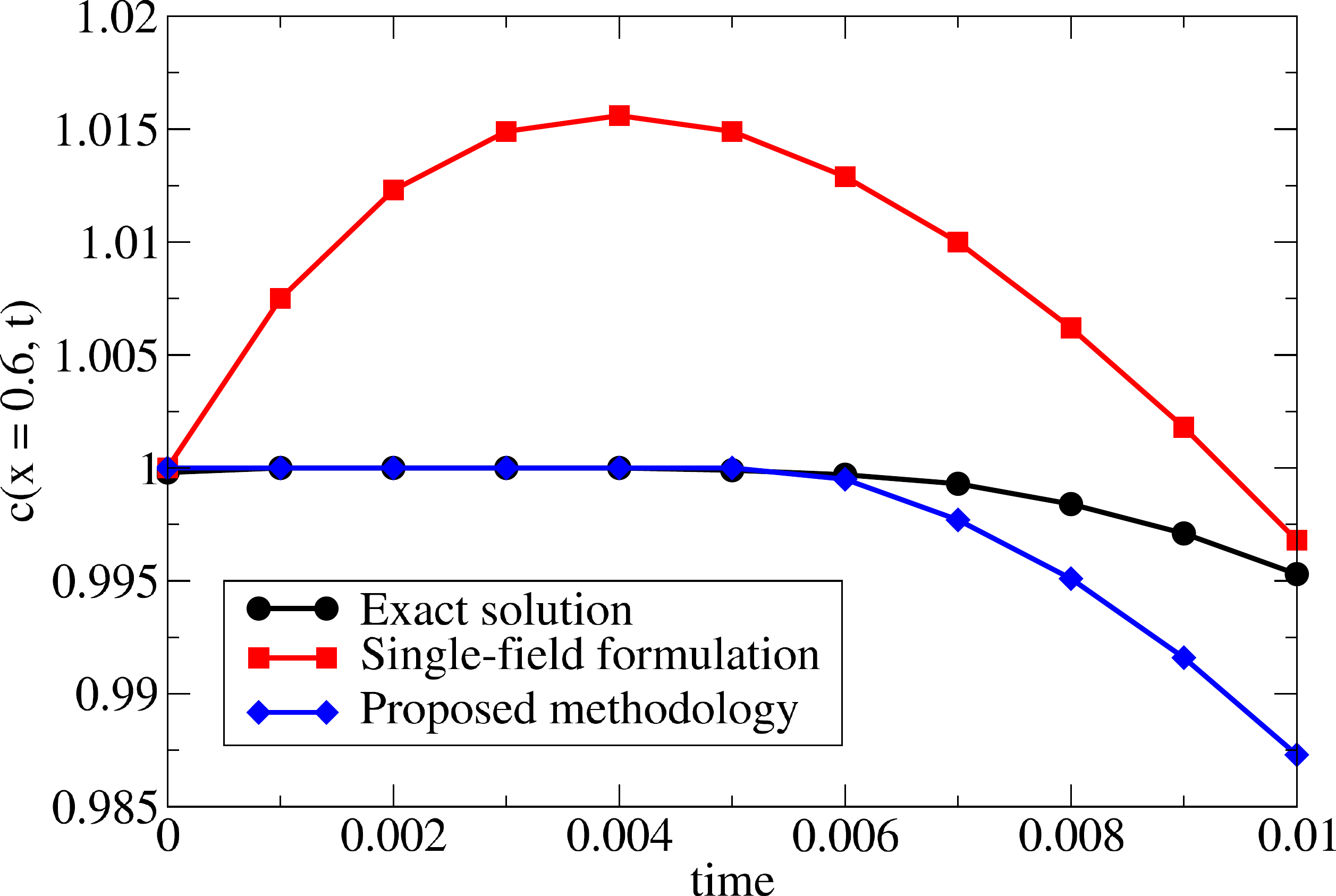}
  \caption{One dimensional problem with uniform initial condition: 
    This figure shows the concentration at $\mathrm{x} = 0.6$ as a 
    function of time. The time step is taken as 
    $\Delta t = 0.001$, and five equally spaced two-node 
    finite elements are employed. The numerical solutions obtained 
    from the single-field formulation and the proposed methodology 
    are compared with the analytical solution. From the maximum 
    principle, it is known that the analytical solution is bounded 
    above by unity. The numerical solution from the single-field 
    formulation exceeds unity while the proposed methodology 
    satisfies the maximum principle. \label{Fig:OneD_uniform_IC_u_x_dot6}}
\end{figure}

\begin{figure}
\psfrag{c}{$\dot{c}(\mathrm{x},t = 0.009)$}
  \includegraphics[clip,scale=0.4]{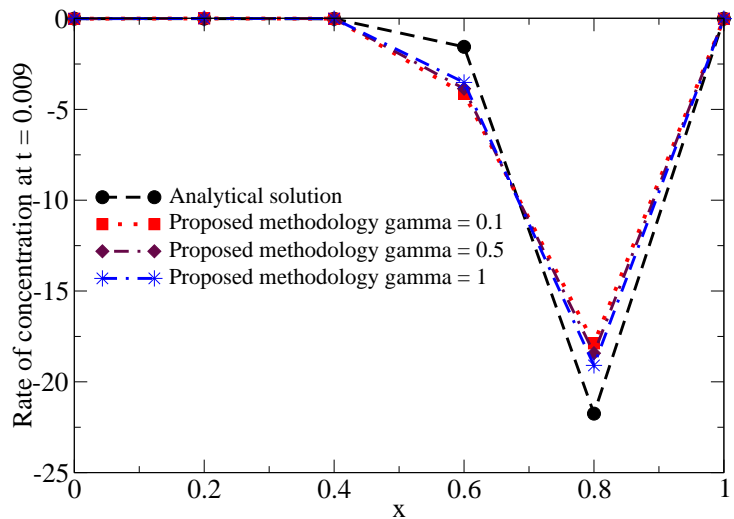}
  \caption{One dimensional problem with uniform initial condition: 
    This figure shows the rate of nodal concentrations at $t = 0.009$ 
    as a function of $\mathrm{x}$. Various values of $\gamma$ are 
    used, which are indicated in the figure. The time step is taken as 
    $\Delta t = 0.001$, and five equally spaced two-node finite elements 
    are employed. The numerical solutions obtained from the proposed 
    methodology are compared with the analytical solution.  
    \label{Fig:OneD_Harari_vn_second_method}}
\end{figure}

\begin{figure}[htp]
  \centering
  \includegraphics[scale=0.45]{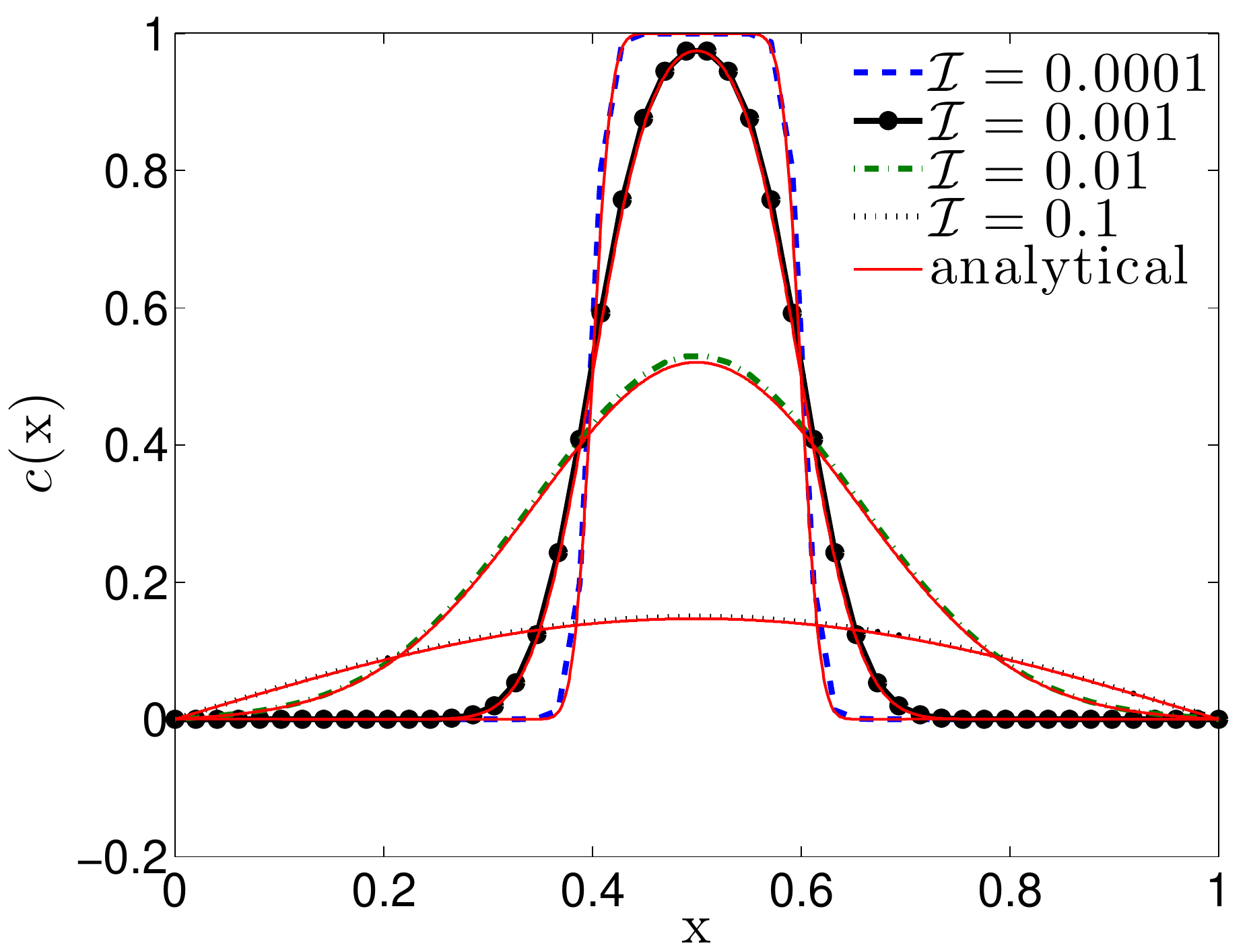}
  \caption{One-dimensional problem with non-uniform initial condition: 
    This figure compares the concentration obtained using the proposed 
    methodology with the analytical solution at various instants of 
    time. For this test problem, the solution should be between zero 
    and unity. The time step used in the numerical simulation is 
    $\Delta t = 10^{-4}$. As one can see from the figure, 
    the proposed methodology performed well, and it did not violate 
    the maximum principle and the non-negative constraint. 
    \label{Fig:TransientDMP_1D_proposed_formulation}}
\end{figure}

\begin{figure}[h]
  \centering
  \subfigure[single-field formulation]{
    \includegraphics[scale=0.3]{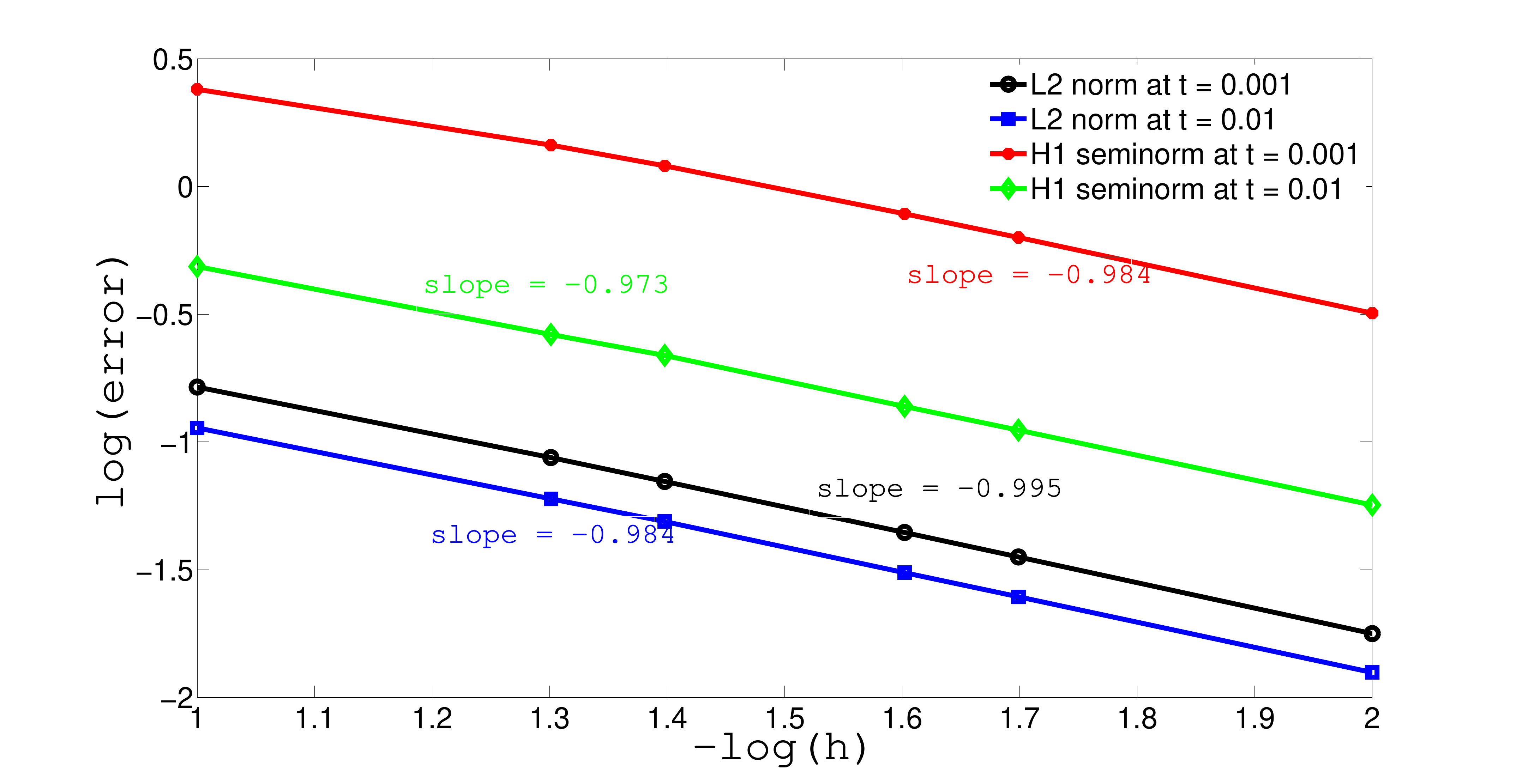}}        
  \subfigure[proposed methodology]{
    \includegraphics[scale=0.3]{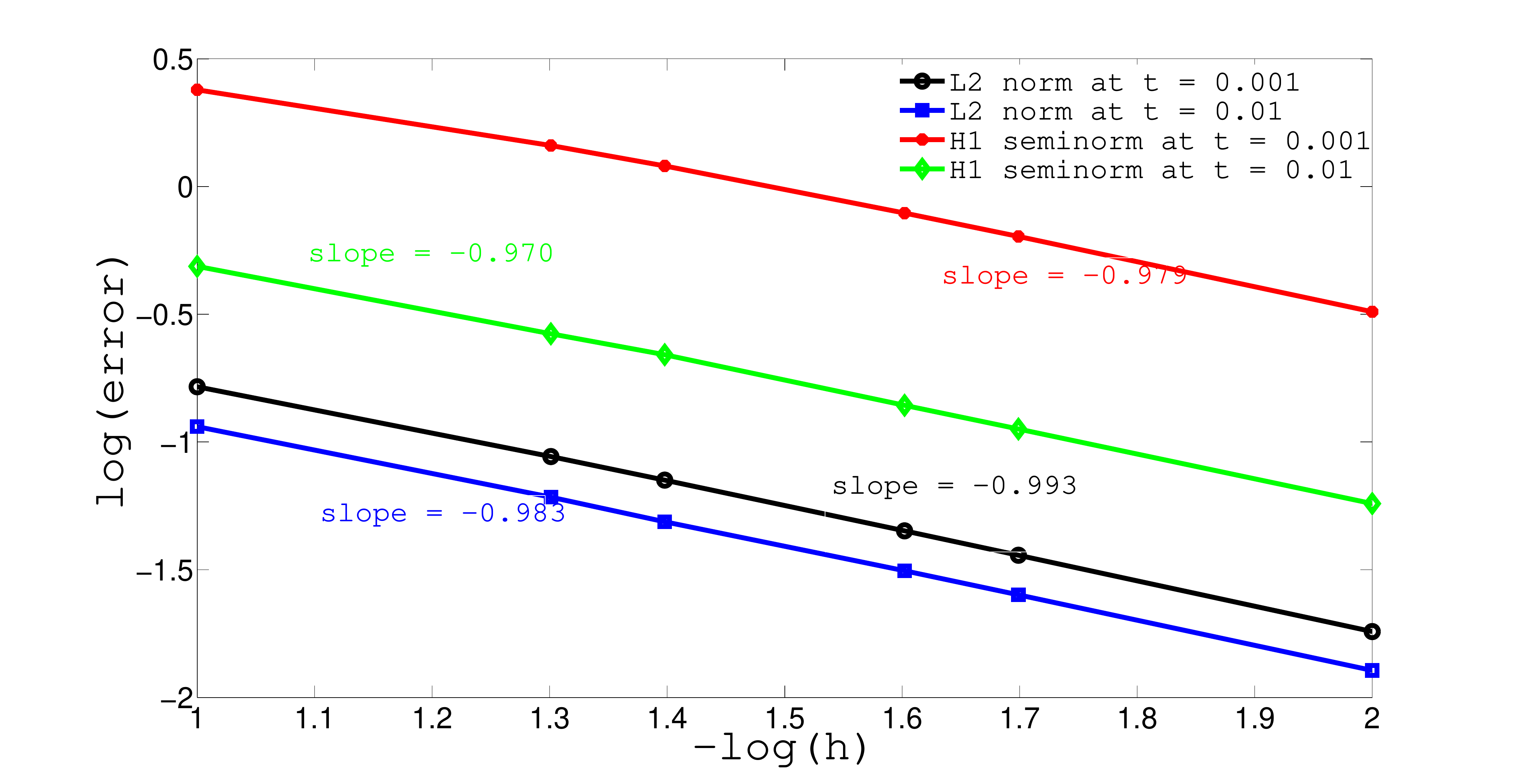}} 
  \caption{One-dimensional problem with non-uniform initial 
    condition: This figure compares the numerical convergence 
    of the single-field formulation (top figure) and the 
    proposed non-negative methodology (bottom figure) with 
    simultaneous spatial and temporal refinements such that 
    $\Delta t \propto {(\Delta \mathrm{x})}^2$. In this 
    numerical simulation, we have taken $\gamma = 1$. The 
    convergence is carried out at two different time levels: 
    $t = 0.001$ and $t = 0.01$. The coarsest mesh has $11$ 
    nodes, and the corresponding time step used for this 
    mesh is $\Delta t = 10^{-3}$. The terminal rates of 
    convergence in $L_{2}$-norm and $H^{1}$-seminorm are 
    also indicated in the figure. 
    \label{Fig:TransientDMP_1D_convergence}}
\end{figure}

\begin{figure}[!h]
  \centering
  \subfigure{
    \includegraphics[clip,scale=0.42]{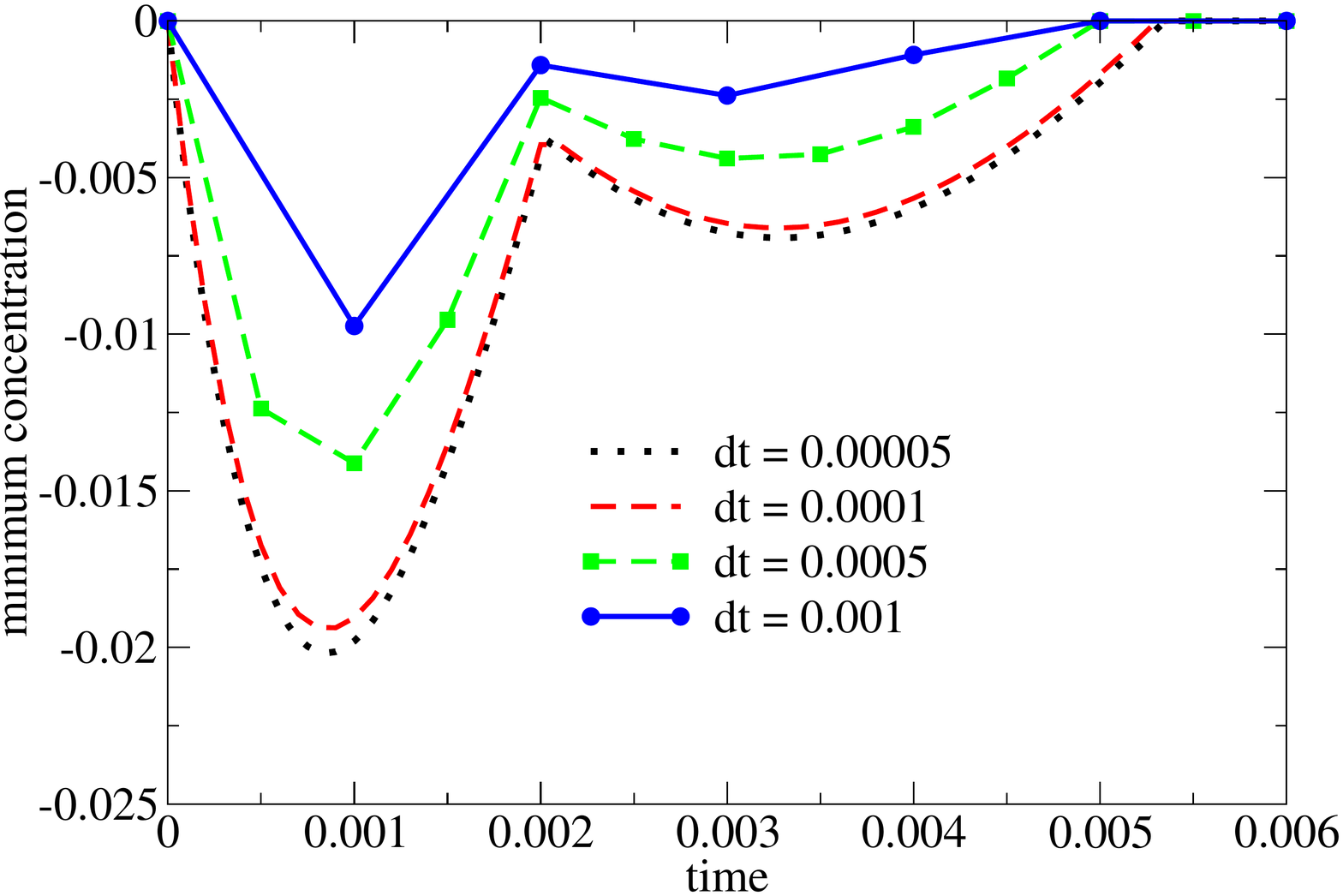}}        
  \subfigure{
    \includegraphics[clip,scale=0.42]{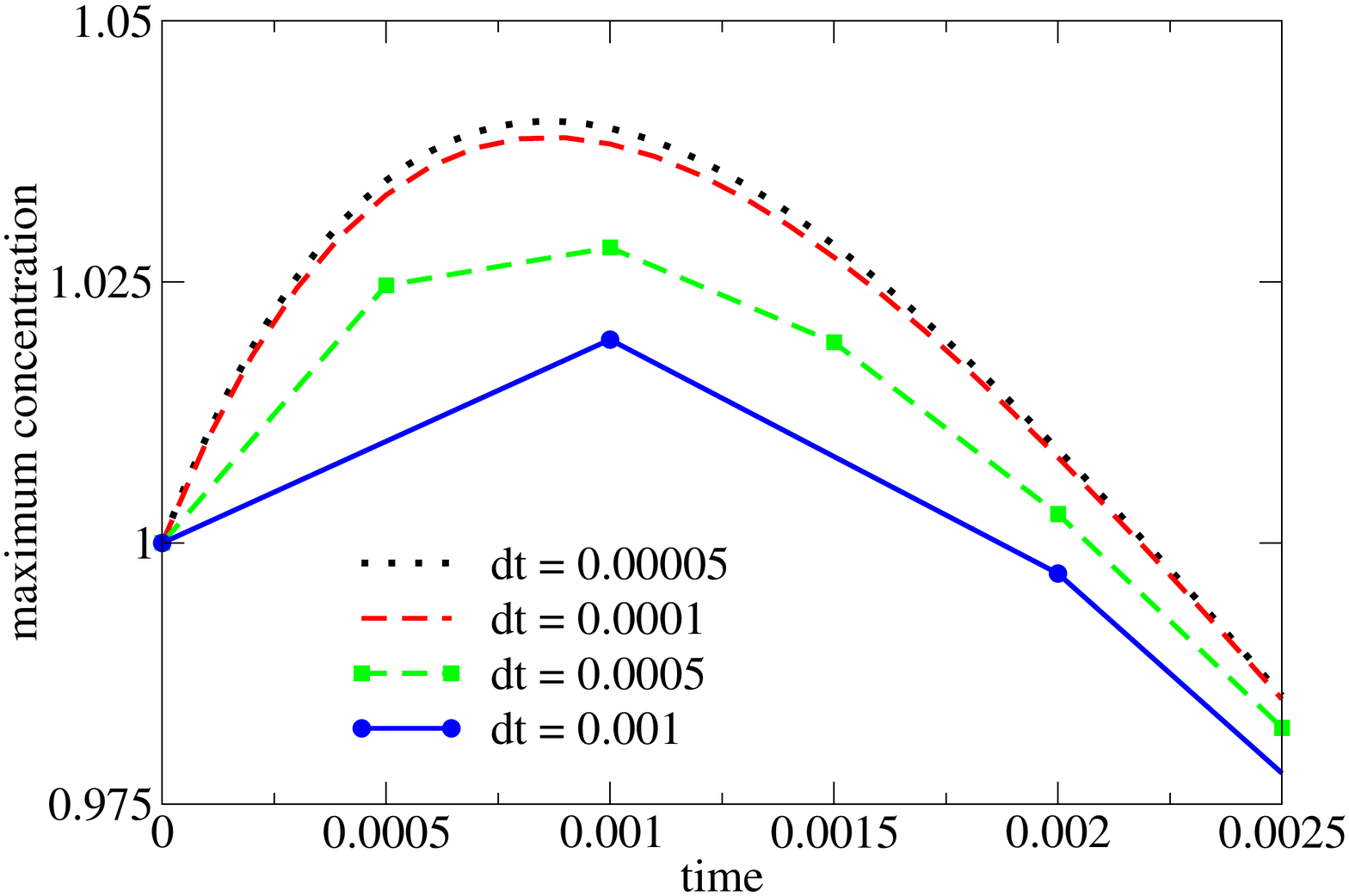}} 
  \caption{One-dimensional problem with non-uniform initial condition: 
    This figure shows the variation of the \emph{minimum} (top) 
    and \emph{maximum} (bottom) concentrations with respect to 
    time for a given mesh and for different time steps under 
    the \emph{single-field formulation}. According to the 
    maximum principle, the concentration should lie between 
    0 and 1. The domain is divided into 10 equal-sized two-node 
    finite elements. The time steps used in the numerical 
    simulation are indicated in the figure. The single-field 
    formulation violated both the minimum (which is the 
    non-negative constraint) and the maximum values that 
    are given by the maximum principle. \emph{It should 
      also be noted that the violations increase in 
      magnitude with a decrease in the time step.} 
    \label{Fig:TransientDMP_1D_min_max_conc}}
\end{figure}

\begin{figure}[!h]
  \centering
  \subfigure{
    \includegraphics[clip,scale=0.42]{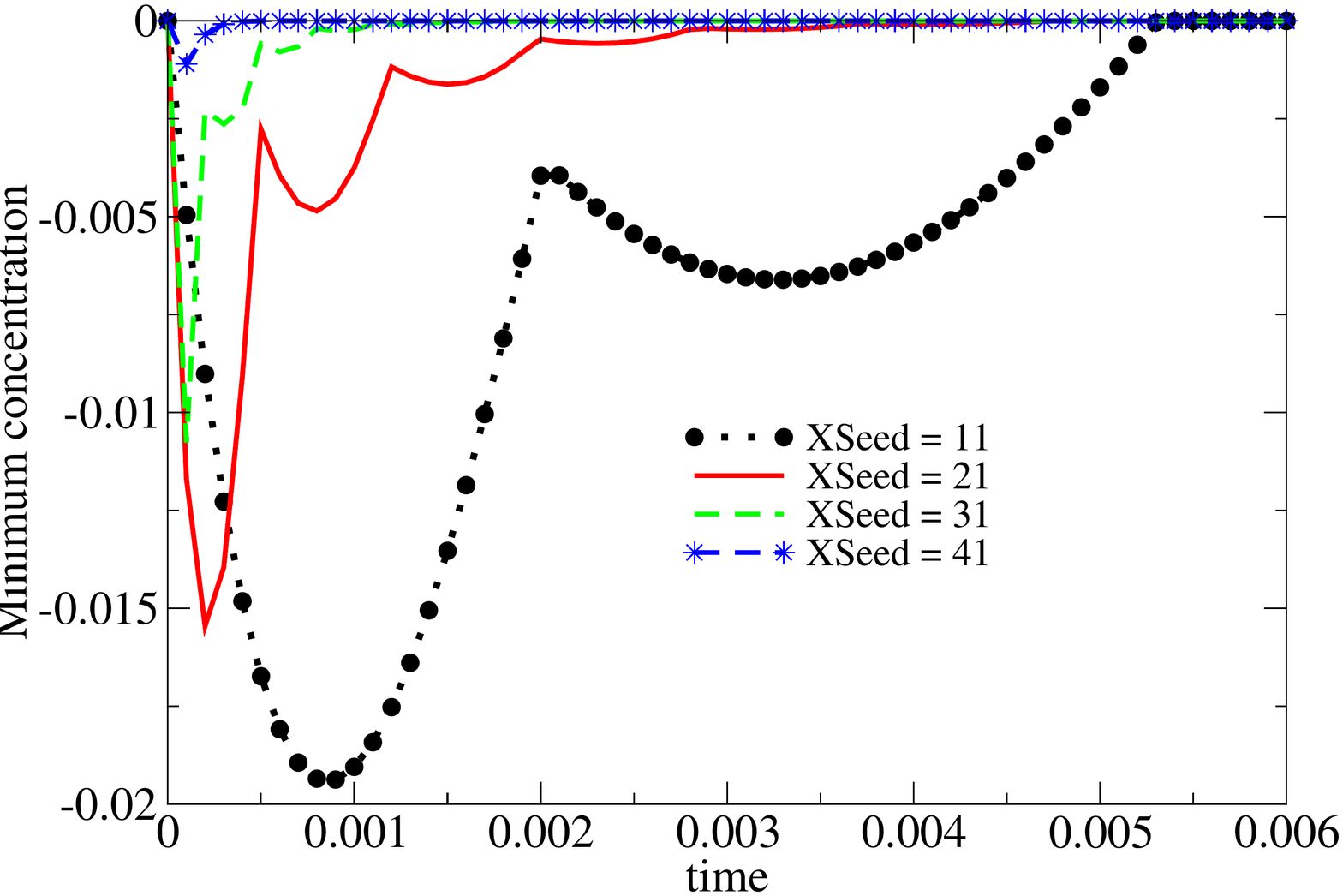}}        
    \subfigure{
    \includegraphics[clip,scale=0.42]{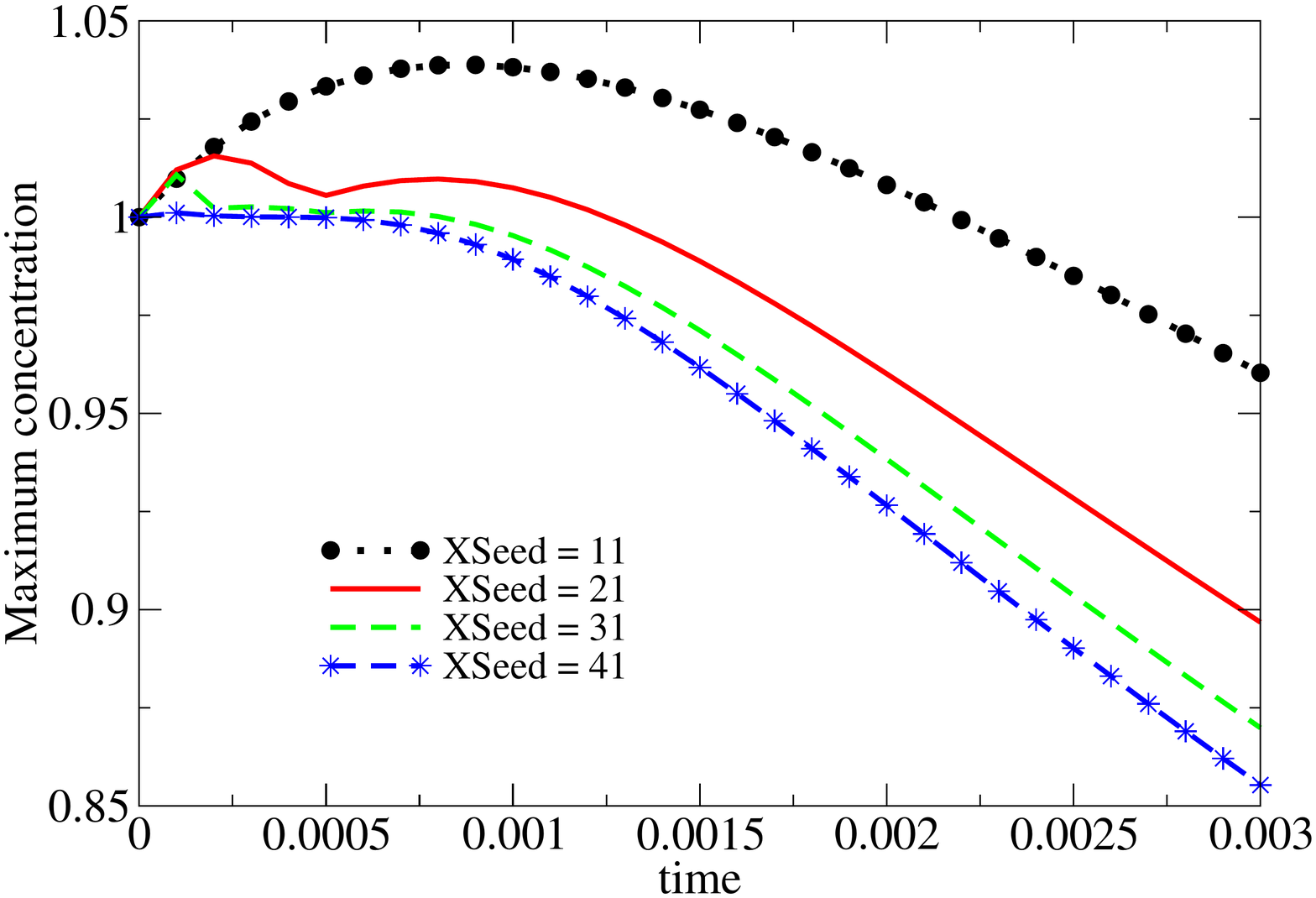}}         
  \caption{One-dimensional problem with non-uniform initial condition: 
    This figure shows the variation of the \emph{minimum} (top) 
    and \emph{maximum} (bottom) concentrations with respect to 
    time for a given time step and for different mesh refinements 
    under the \emph{single-field formulation}. The time step is 
    taken as $\Delta t = 10^{-4}$. According to the maximum 
    principle, the concentration should lie between 0 and 1. 
    Various meshes are used (XSeed = 11, 21, 31 and 41). 
    Note that XSeed denotes the number of nodes. The 
    single-field formulation violated both the minimum (which 
    is the non-negative constraint) and the maximum values 
    that are given by the maximum principle. 
    \emph{It should be noted that the violations of the maximum 
    principle decrease with mesh refinement if the diffusion 
    is anisotropic \cite{Nagarajan_Nakshatrala_IJNMF_2011_v67_p820}.} 
    \label{Fig:TransientDMP_1D_nonuniform_min_max_fixed_mesh}}
\end{figure}


\begin{figure}[htp]
  \centering
  \includegraphics[scale=0.85]{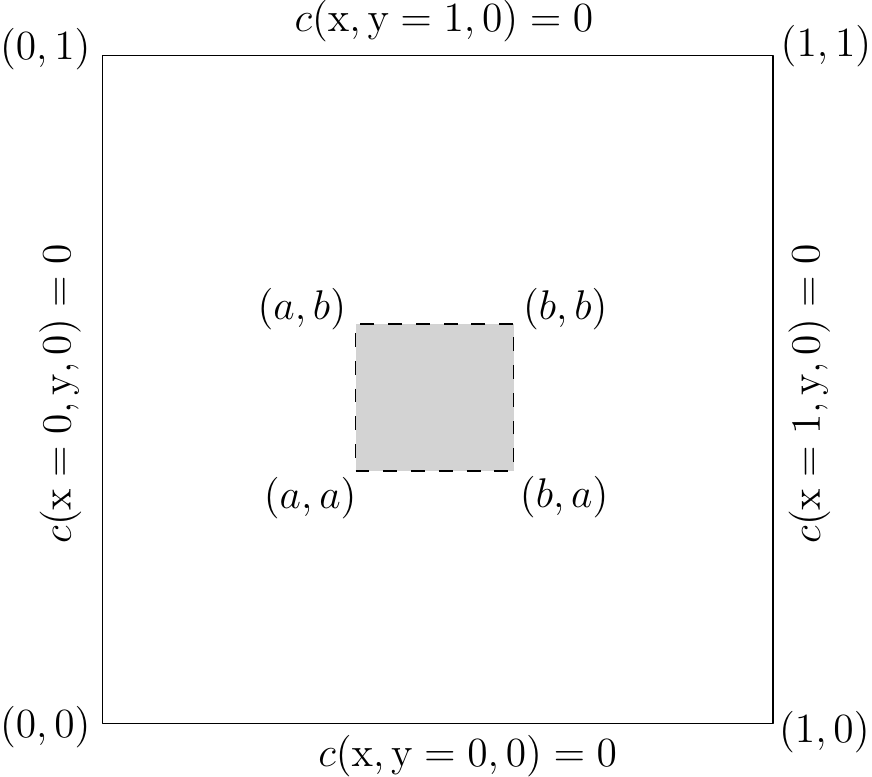}
  \caption{Two-dimensional problem with non-uniform initial condition: 
    A pictorial description of the problem described in subsection 
    \ref{Two_dimensional_problem}. The shaded region has an initial 
    concentration of $c(\mathrm{x},\mathrm{y},t=0)=1$, and the 
    remaining part of the domain has an initial condition of 
    $c(\mathrm{x},\mathrm{y},t=0)=0$. Homogeneous Dirichlet 
    boundary condition is prescribed on the entire boundary. 
    \label{Fig:TransientDMP_2D_pictorial_rep}}
\end{figure}

\begin{figure}[htp]
  \centering
  \subfigure[T3 finite element mesh]{
  \includegraphics[scale=0.4]{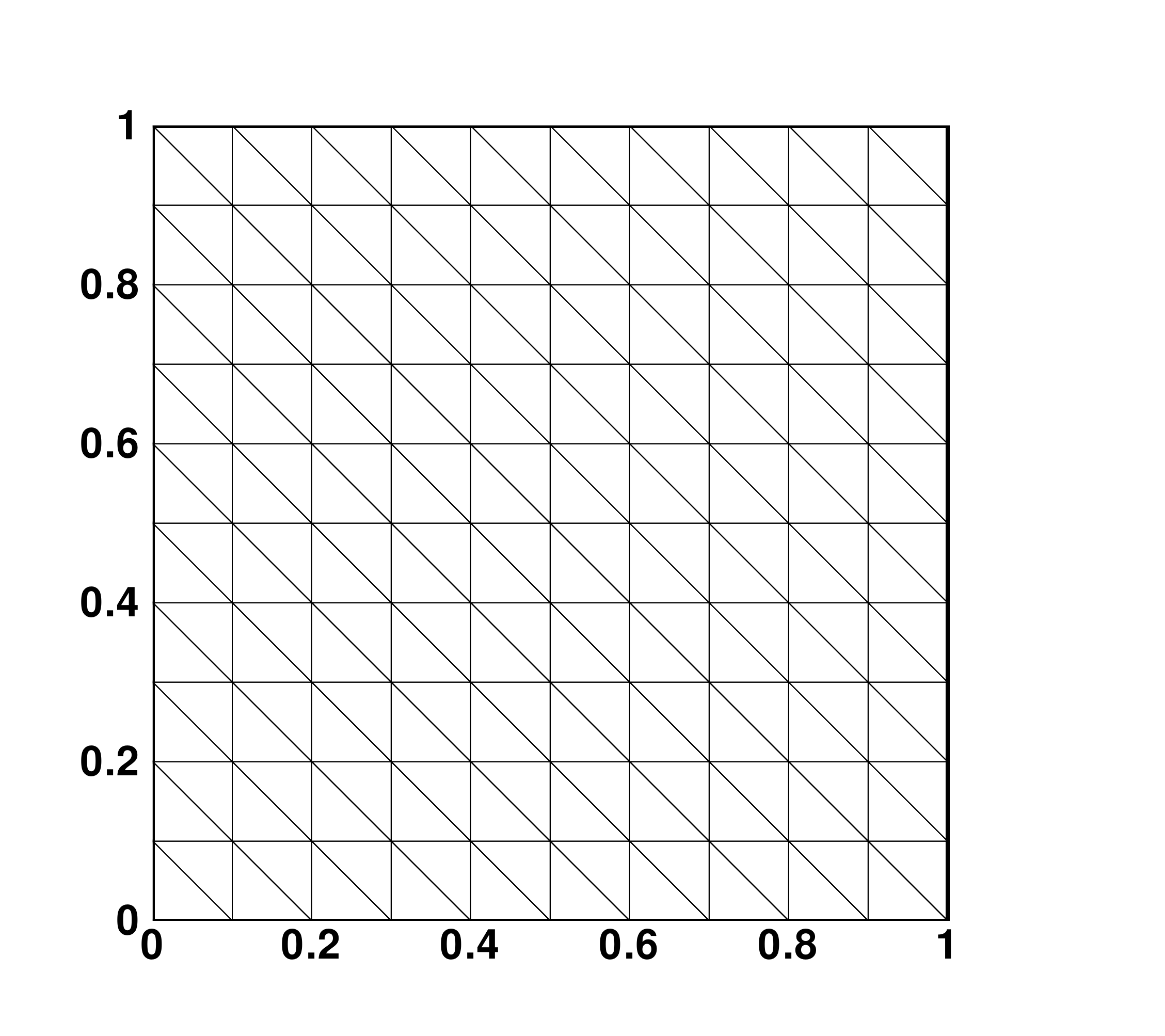}}
  \subfigure[Q4 finite element mesh]{
  \includegraphics[scale=0.4]{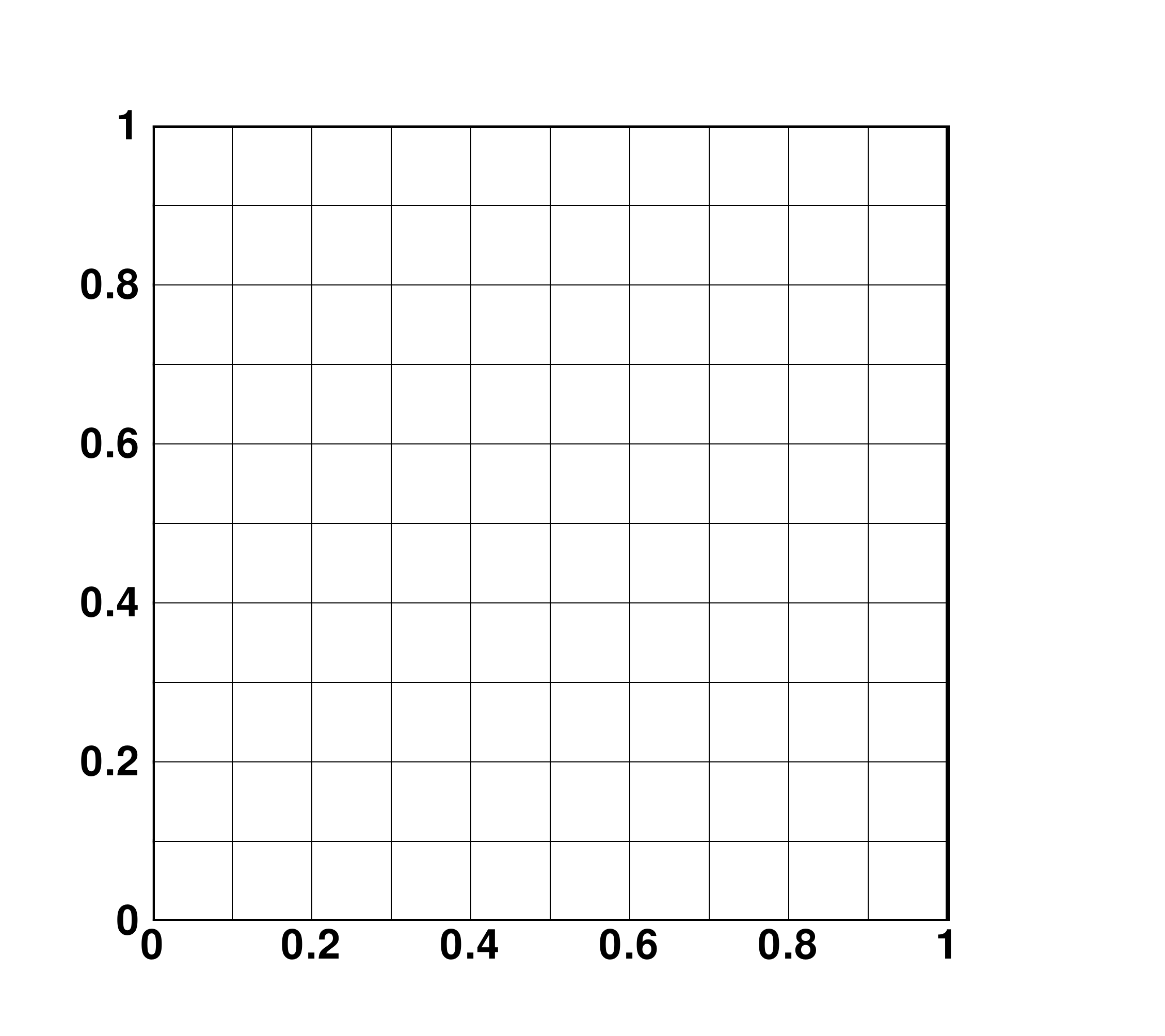}}
  \caption{Two-dimensional problem with non-uniform initial condition: 
    This figure shows the typical meshes used in the $h$-numerical 
    convergence studies. The meshes in this figure have XSeed = 
    YSeed = 11, which denote the number of nodes along the 
    x-direction and y-direction. The convergence study employs 
    hierarchical meshes. 
    \label{Fig:TransientDMP_2D_convergence_meshes}}
\end{figure}

\begin{figure}[htp]
  \centering
  \subfigure[Galerkin T3]{
  \includegraphics[scale=0.27,clip]{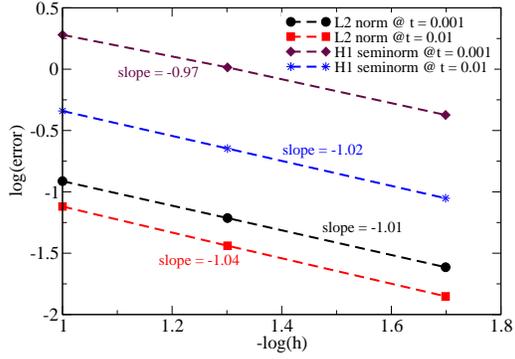}}
  \subfigure[Proposed methodology T3]{
  \includegraphics[scale=0.27,clip]{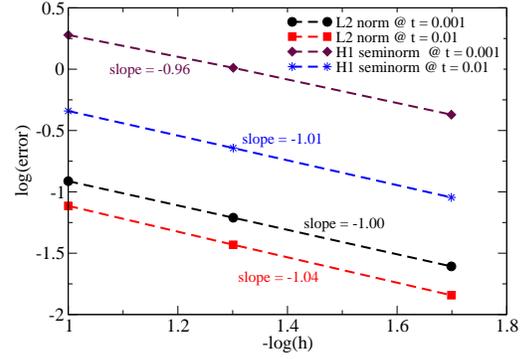}}
    \subfigure[Galerkin Q4]{
  \includegraphics[scale=0.27,clip]{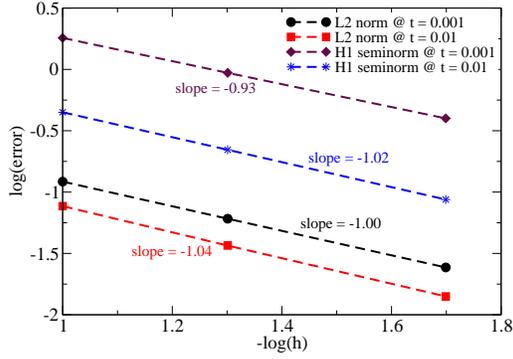}}
    \subfigure[Proposed methodology Q4]{
  \includegraphics[scale=0.27,clip]{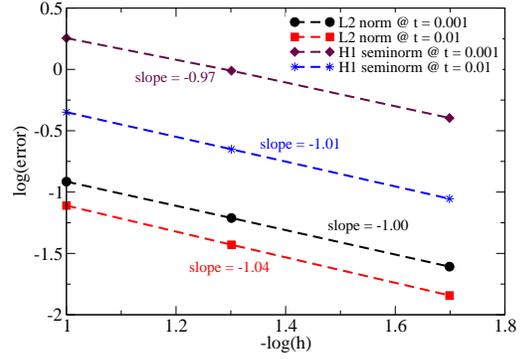}}
  \caption{Two-dimensional problem with non-uniform initial condition: 
    This figure illustrates the numerical convergence of the Galerkin 
    single-field formulation and the proposed methodology for 
    three-node triangular element (T3) and four-node quadrilateral 
    element (Q4). We have taken $\gamma = 1$. The numerical 
    convergence is performed at two time levels $t = 0.001$ and 
    $t = 0.01$. A hierarchy of meshes are employed in the numerical 
    study. The initial mesh has XSeed = YSeed = 11, which denote 
    the number of nodes along the x-direction and y-direction. The 
    corresponding time step for this mesh is taken as $\Delta t = 0.001$. 
    The mesh and the time step are simultaneously refined as $\Delta t 
    \propto {(\Delta \mathrm{x})}^2$. The terminal rates of convergence 
    in $L_{2}$-norm and $H^{1}$-seminorm are indicated in the figure. 
    \label{Fig:TransientDMP_2D_proposed_convergence}}
\end{figure}

\begin{figure}[!h]
  \centering
  \subfigure[20 $\times$ 20 equally spaced three-node triangular mesh]{
    \includegraphics[scale=0.35]{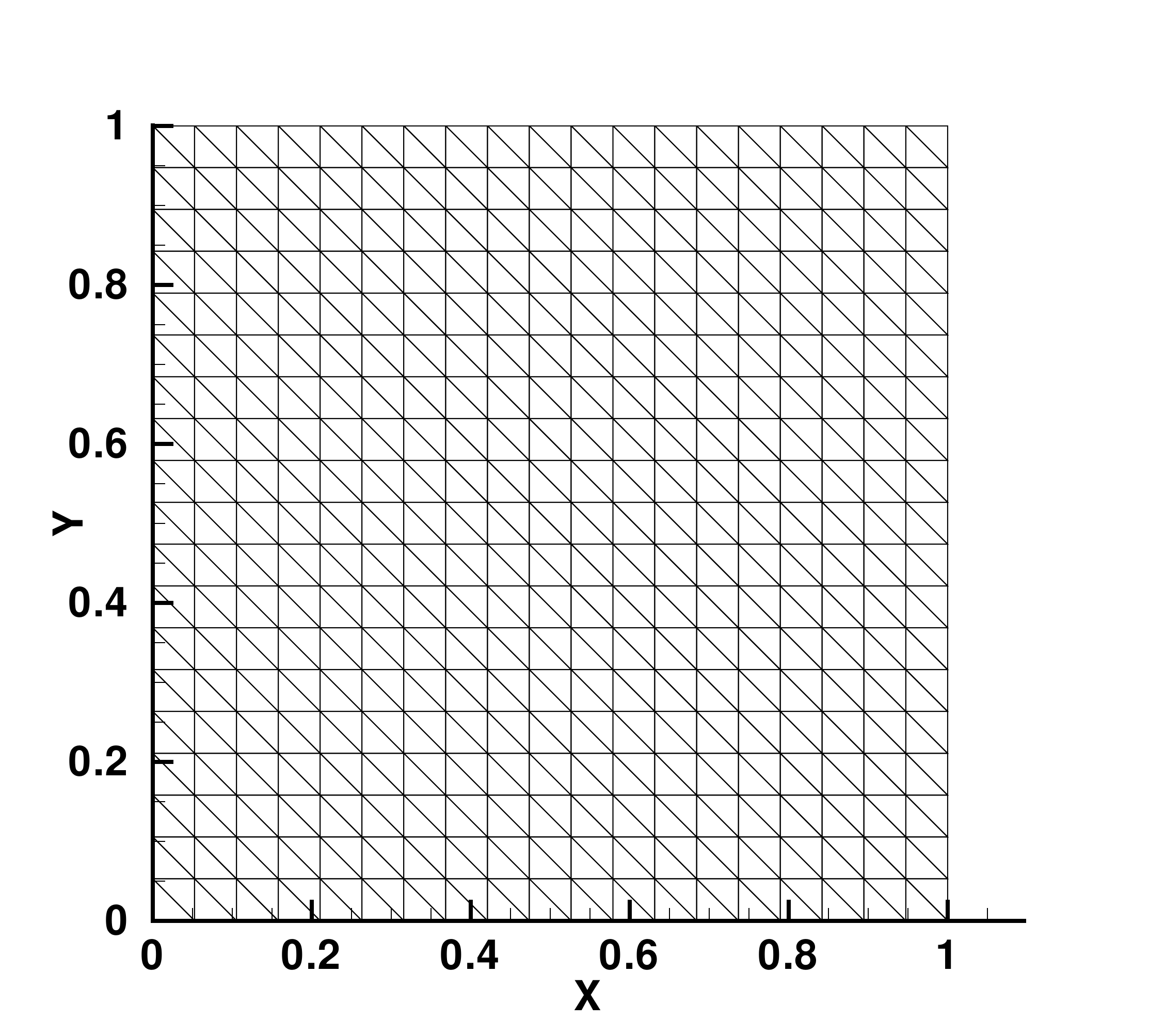}}        
  \subfigure[Violation of the non-negative constraint]{
    \includegraphics[scale=0.35]{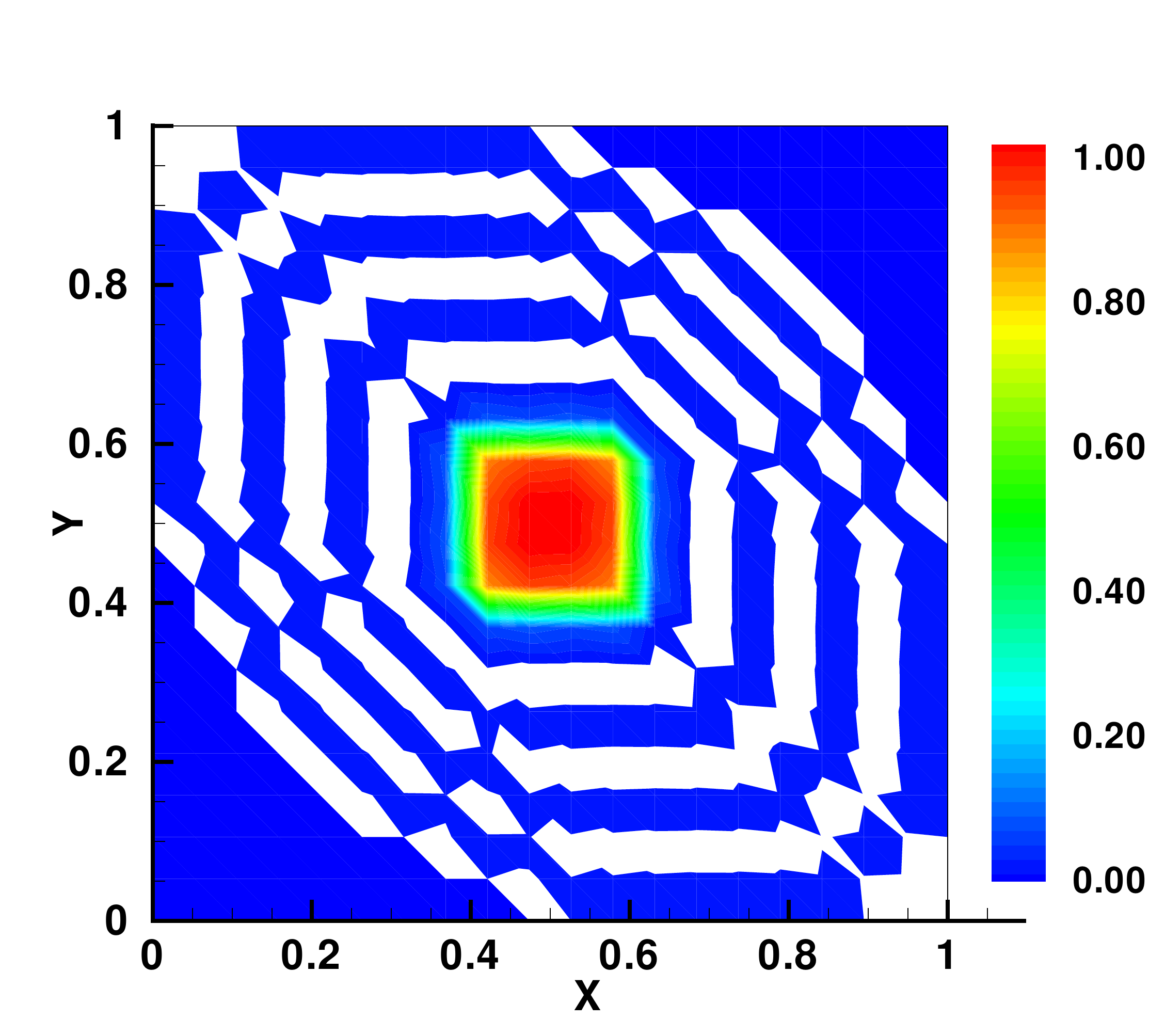}} 
  \subfigure[Violation of the maximum principle]{
    \includegraphics[scale=0.35]{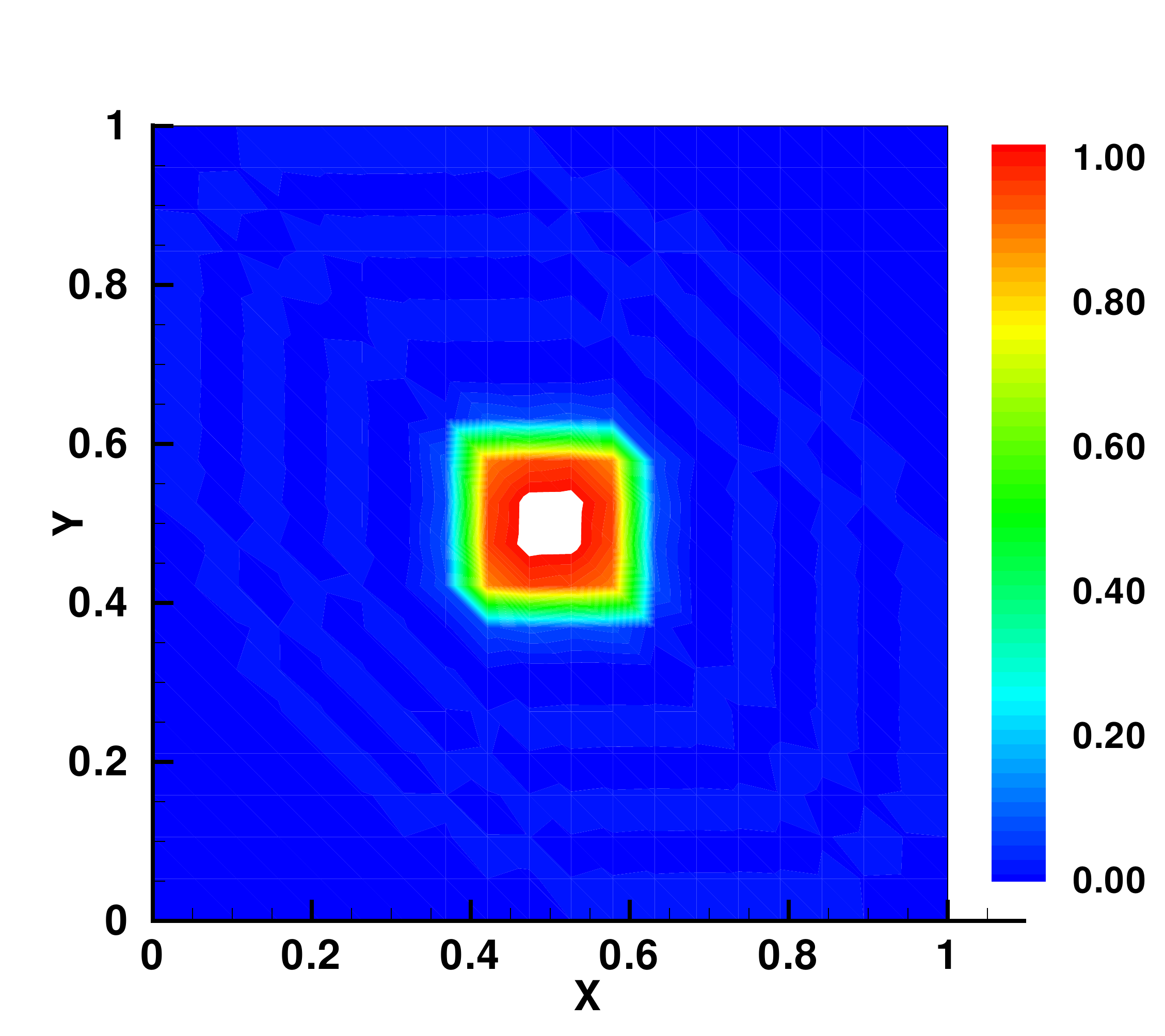}} 
  \subfigure[Proposed methodology]{
    \includegraphics[scale=0.35]{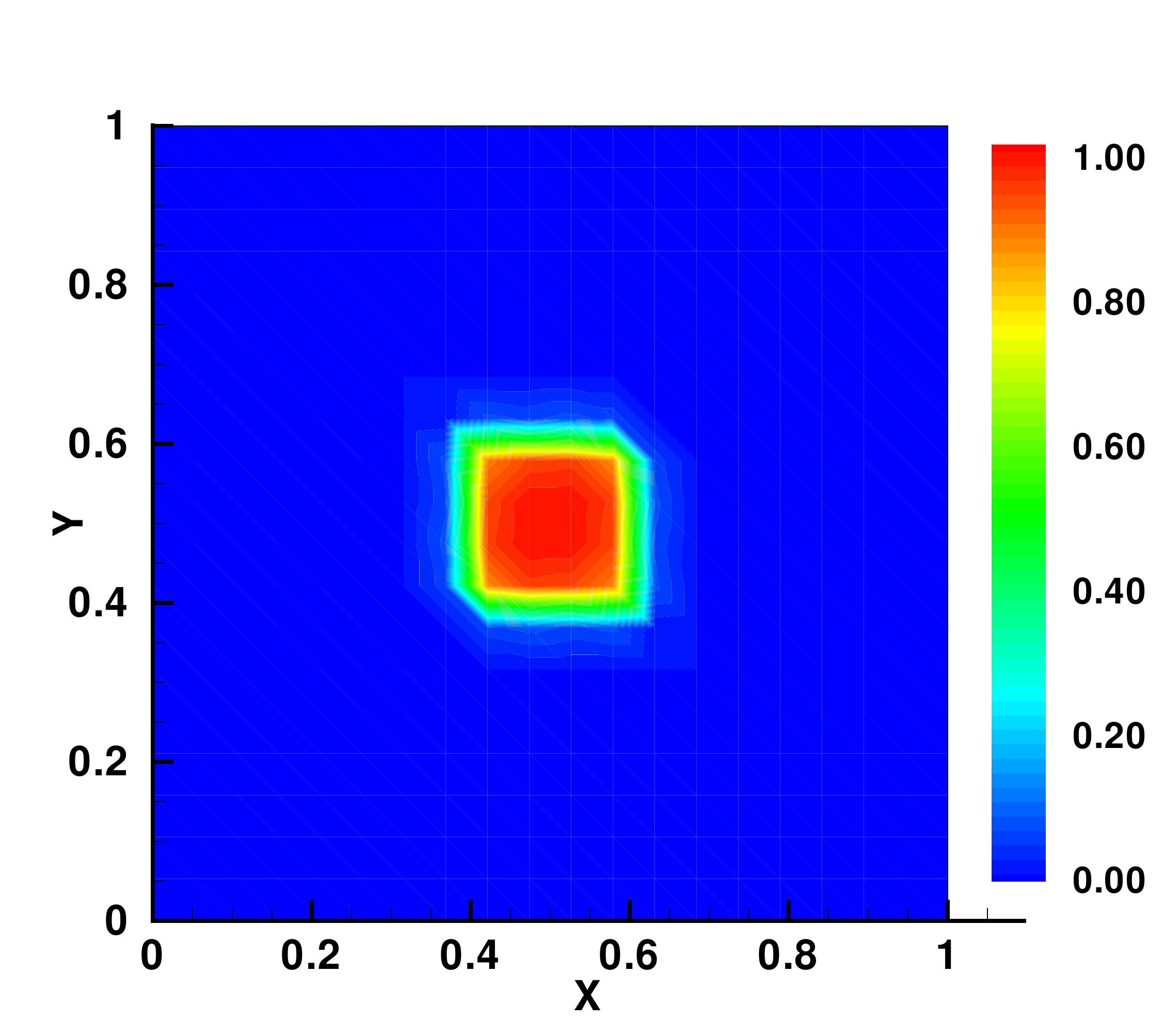}}        
  \caption{Two-dimensional problem with non-uniform initial 
    condition: This figure compares the concentrations obtained 
    from the single-field formulation and the proposed methodology 
    with the analytical solution at time level $t = \Delta t = 
    10^{-4}$. 
    Subfigure (b) shows that the single-field formulation 
    violates the non-negative constraint, as 36\% of nodes 
    have negative concentrations. The obtained minimum 
    concentration is $-0.01221$. 
    Subfigure (c) shows that the single-field formulation 
    violates the maximum principle, as  1\% of nodes having 
    concentrations greater than unity. The obtained maximum 
    concentration is $1.02039$. 
    Subfigure (d) shows that the concentration obtained from 
    the proposed methodology satisfies the maximum principle, 
    and the non-negative constraint. In subfigure (b), the 
    regions with negative concentrations are indicated in 
    white color. In subfigure (c), the regions with 
    concentrations greater than unity are indicated 
    in white color. 
    \label{Fig:TransientDMP_2D_min_conc}}
\end{figure}

\begin{figure}[!h]
  \centering
  \subfigure[Computational mesh]{
    \includegraphics[scale=0.35]{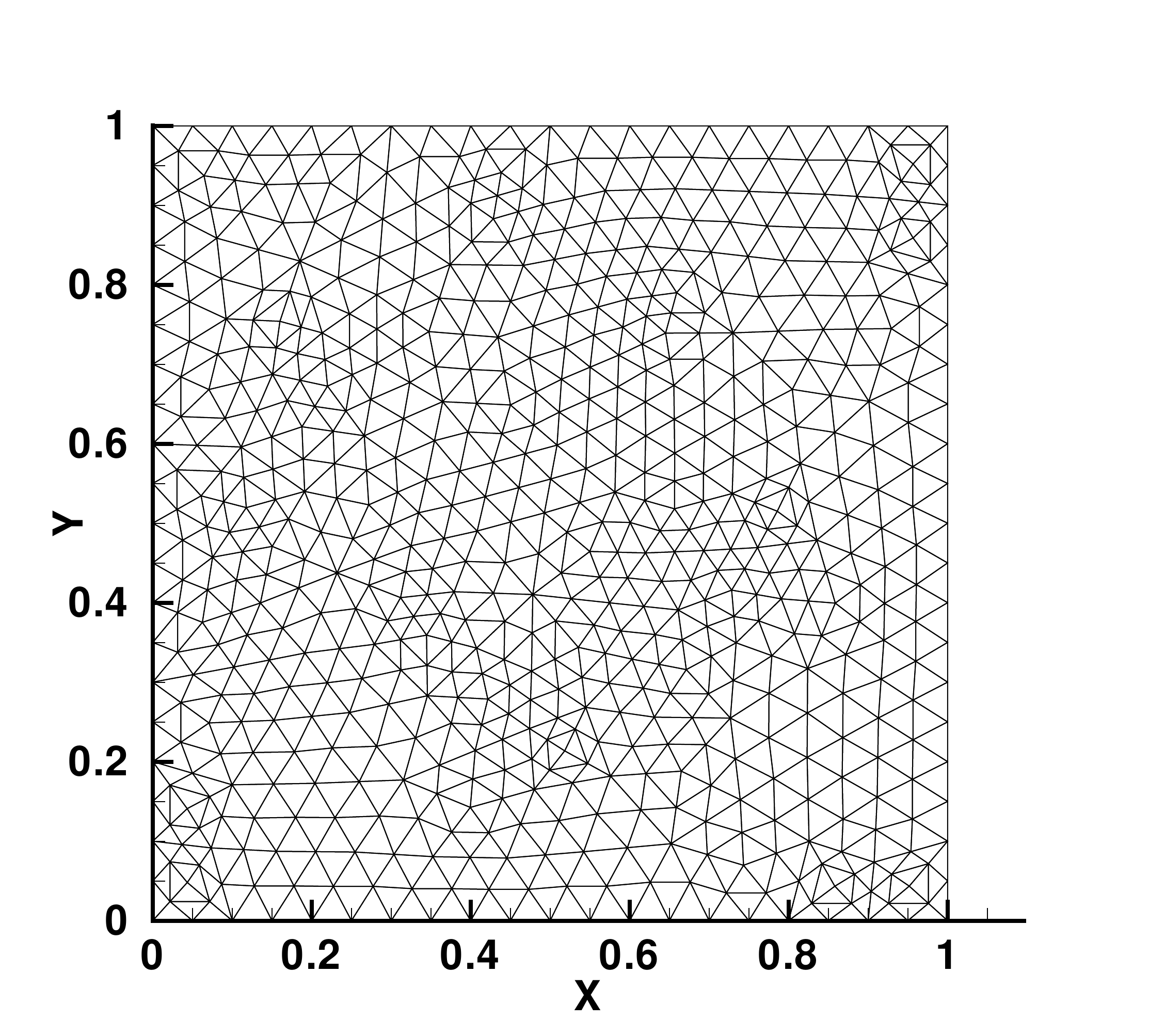}}        
  \subfigure[Violation of the non-negative constraint]{
    \includegraphics[scale=0.35]{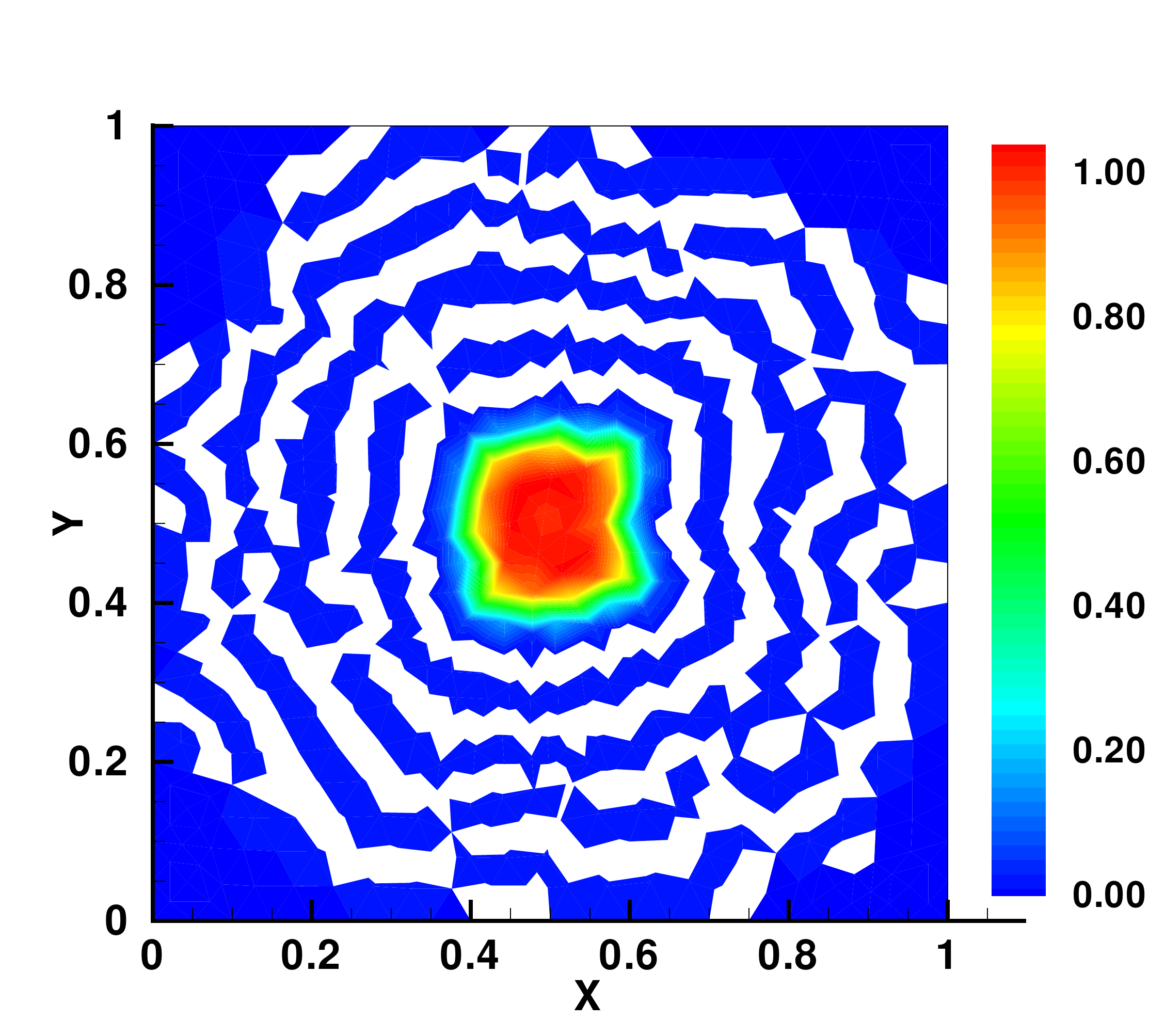}} 
  \subfigure[Violation of the maximum principle]{
    \includegraphics[scale=0.35]{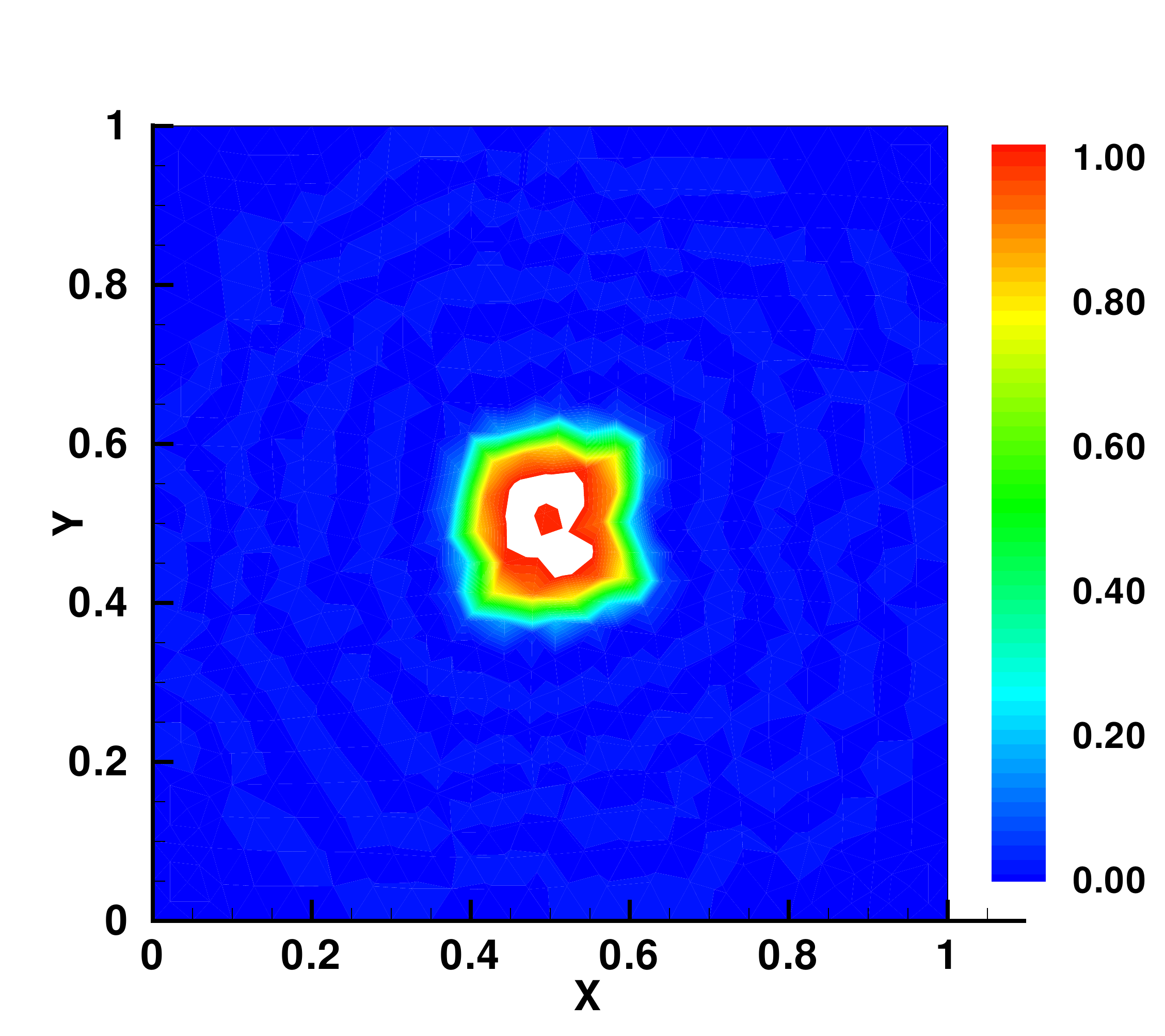}} 
    \subfigure[Proposed methodology]{
    \includegraphics[scale=0.35]{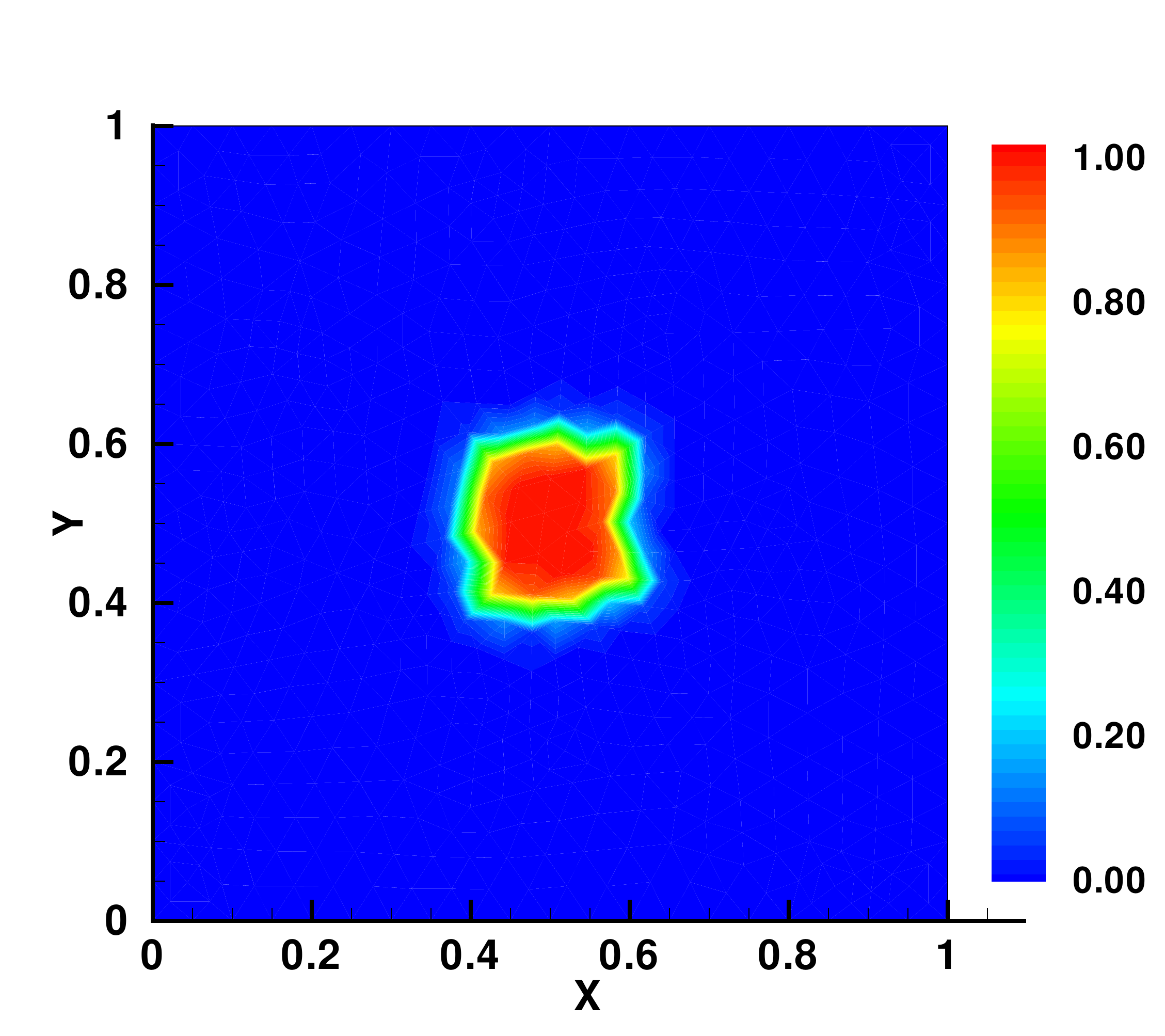}} 
  \caption{Two-dimensional problem with non-uniform initial 
    condition: This figure compares the numerical solutions 
    from MATLAB's PDE Toolbox and the proposed methodology 
    at time level $t = \Delta t = 10^{-4}$. 
    Subfigure (a) shows the computational mesh used in 
    the numerical simulation. 
    Subfigure (b) shows that numerical solution from 
    the MATLAB's PDE Toolbox violates the non-negative 
    constraint, as 40\% of the nodes have negative 
    concentrations. The regions with negative 
    concentrations are indicated in white color. 
    The obtained minimum concentration is $-0.0339$.
    Subfigure (c) shows that the numerical solution from 
    MATLAB's PDE Toolbox violates the maximum principle, 
    as 1.2\% of nodes have concentrations greater than 
    unity. The regions with concentrations greater than 
    unity are indicated in white color. The obtained 
    maximum concentration is $1.0397$. 
    Subfigure (d) shows that the proposed methodology 
    satisfies the maximum principle and the non-negative 
    constraint on the computational mesh generated by 
    MATLAB.} 
  \label{Fig:TransientDMP_2D_pdetool}
\end{figure}

\clearpage 
\newpage 

\begin{figure}
\subfigure{
\includegraphics[clip,scale=0.8]{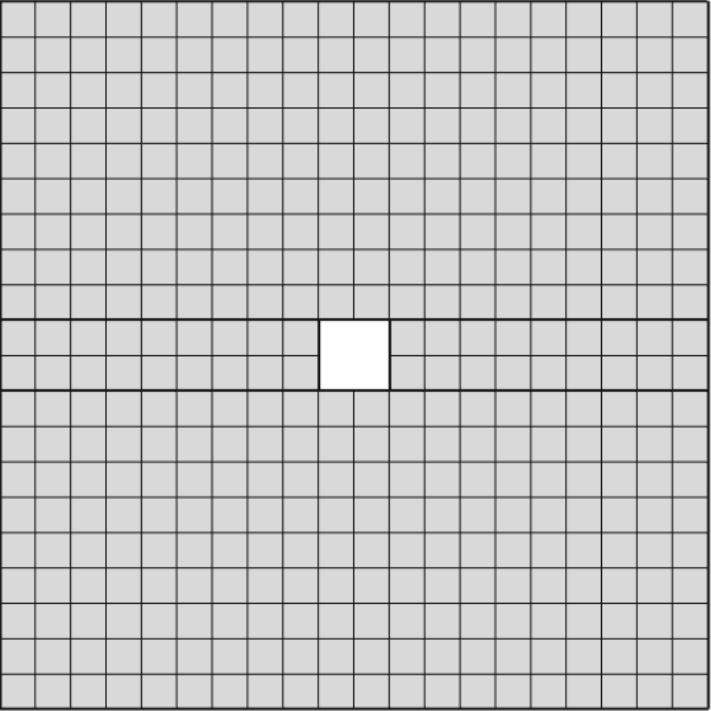}}
\subfigure{
\includegraphics[clip,scale=0.8]{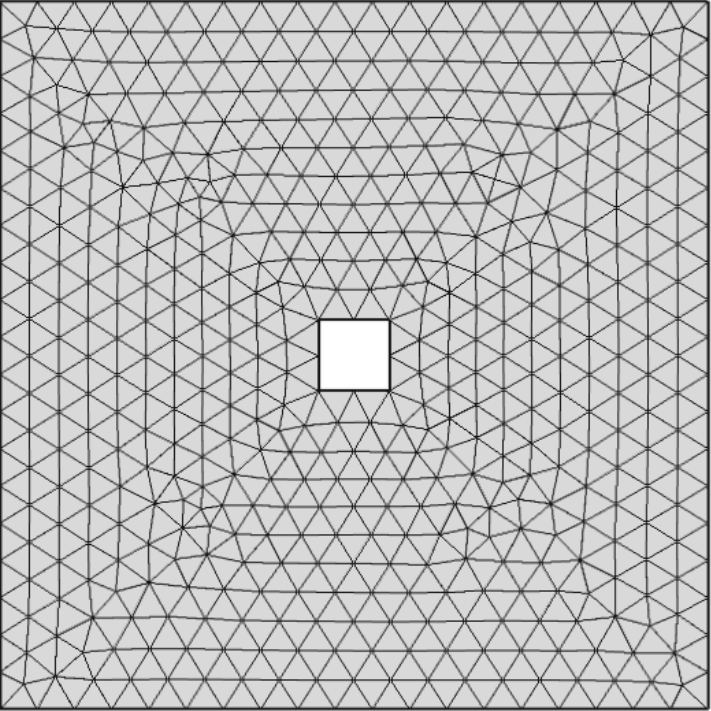}}
\caption{Anisotropic diffusion in a square plate with a hole: This 
figure shows the meshes employed in the numerical simulations 
using COMSOL \cite{multiphysics2012version}. The left figure 
shows a structured mesh based on four-node quadrilateral 
elements, and the right figure shows an unstructured mesh 
based on three-node triangular elements. 
\label{Fig:TransientDMP_COMSOL_meshes}}
\end{figure}

\begin{figure}
\subfigure{
\includegraphics[clip,scale=0.4]{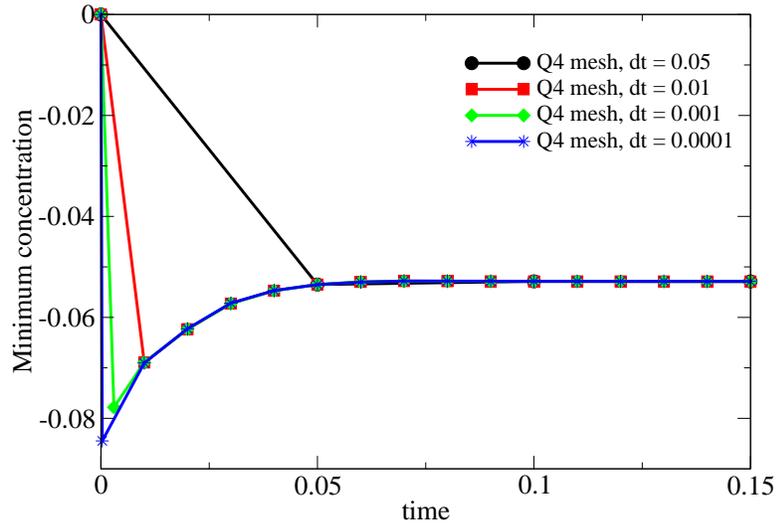}}
\subfigure{
\includegraphics[clip,scale=0.4]{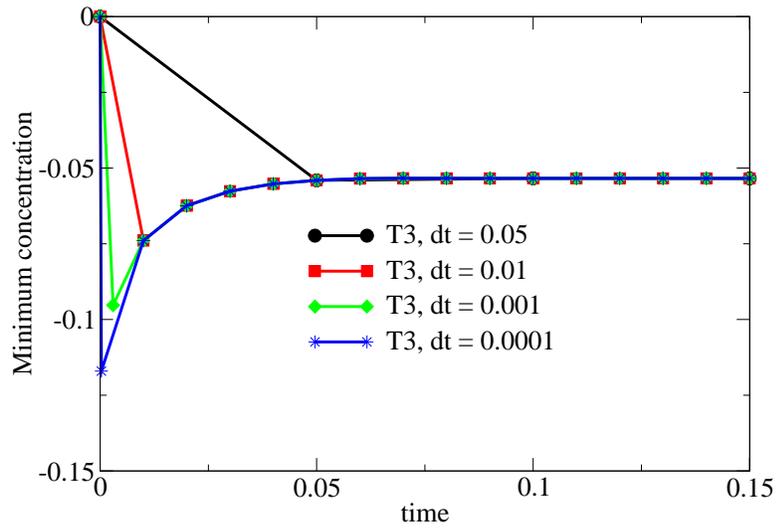}}
\caption{Anisotropic diffusion in square plate with a hole: This figure 
shows the variation of minimum concentration with time under the 
meshes shown in Figure \ref{Fig:TransientDMP_COMSOL_meshes}. 
COMSOL \cite{multiphysics2012version} is employed in the numerical 
simulation. The solution is very close to the steady-state response for 
time greater than 0.05. \label{Fig:TransientDMP_COMSOL_min_conc}}
\end{figure}

\begin{figure}
\subfigure[$\Delta t = 0.0001$, $ t= 3 \Delta t$ (minimum = -0.08447)]{
\includegraphics[clip,scale=1]{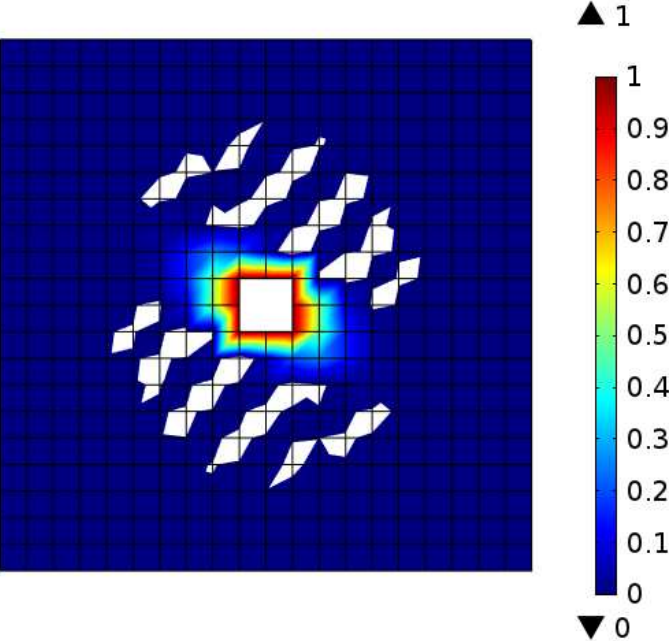}}
\subfigure[$\Delta t = 0.0001$, t = 0.05 (minimum = -0.05351)]{
\includegraphics[clip,scale=1]{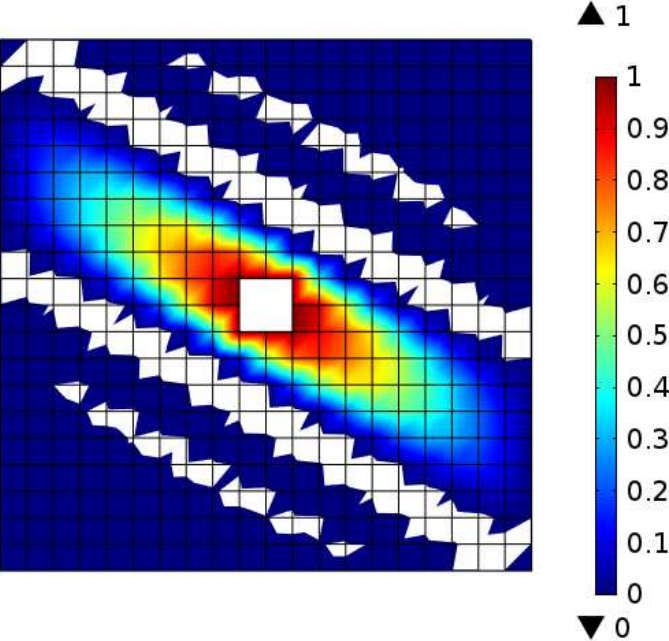}}
\subfigure[$\Delta t = 0.001$, $t = 3 \Delta t$ (minimum = -0.07781)]{
\includegraphics[clip,scale=1]{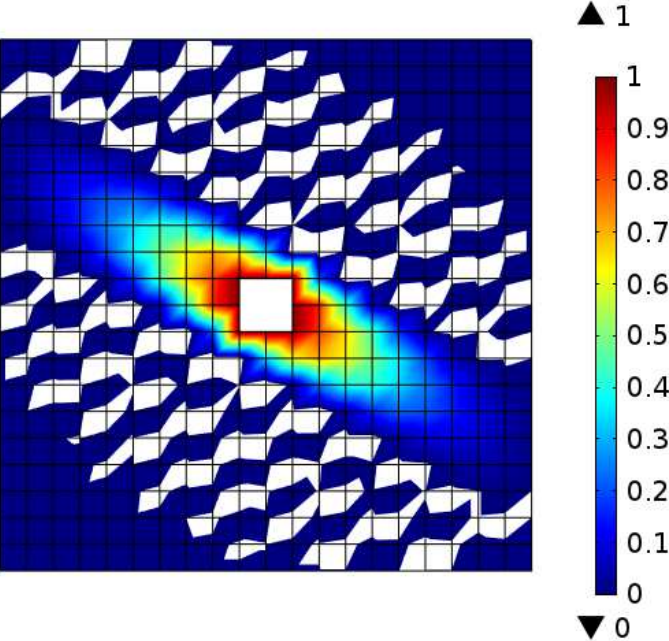}}
\subfigure[$\Delta t = 0.001$, t = 0.05 (minimum = -0.05350)]{
\includegraphics[clip,scale=1]{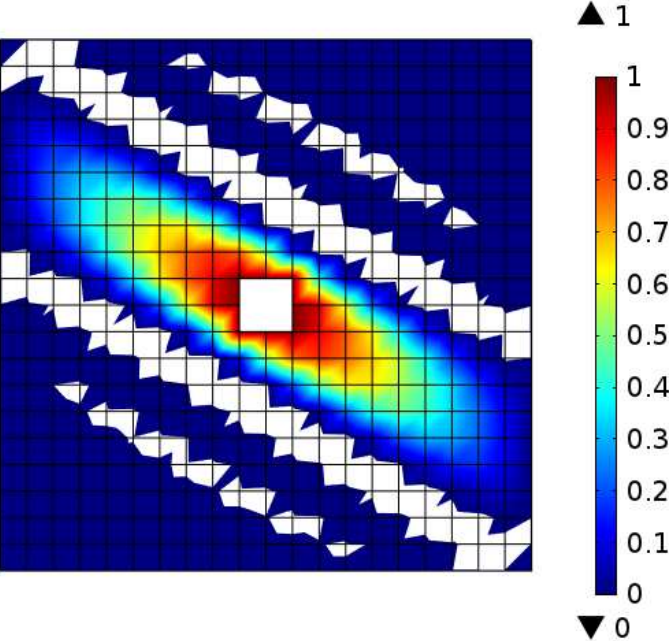}}
\caption{Anisotropic diffusion in square plate with a hole: Concentration 
profiles using COMSOL \cite{multiphysics2012version} by employing 
\emph{structured four-node quadrilateral mesh}. The finite element 
mesh is also shown. The numerical results clearly violated the 
non-negative constraint for the concentration. The regions that 
violated the non-negative constraint are indicated in white color. 
\label{Fig:TransientDMP_COMSOL_Q4}}
\end{figure}

\begin{figure}
\subfigure[$\Delta t = 0.0001$, $ t= 3 \Delta t$ (minimum = -0.1170)]{
\includegraphics[clip,scale=0.7]{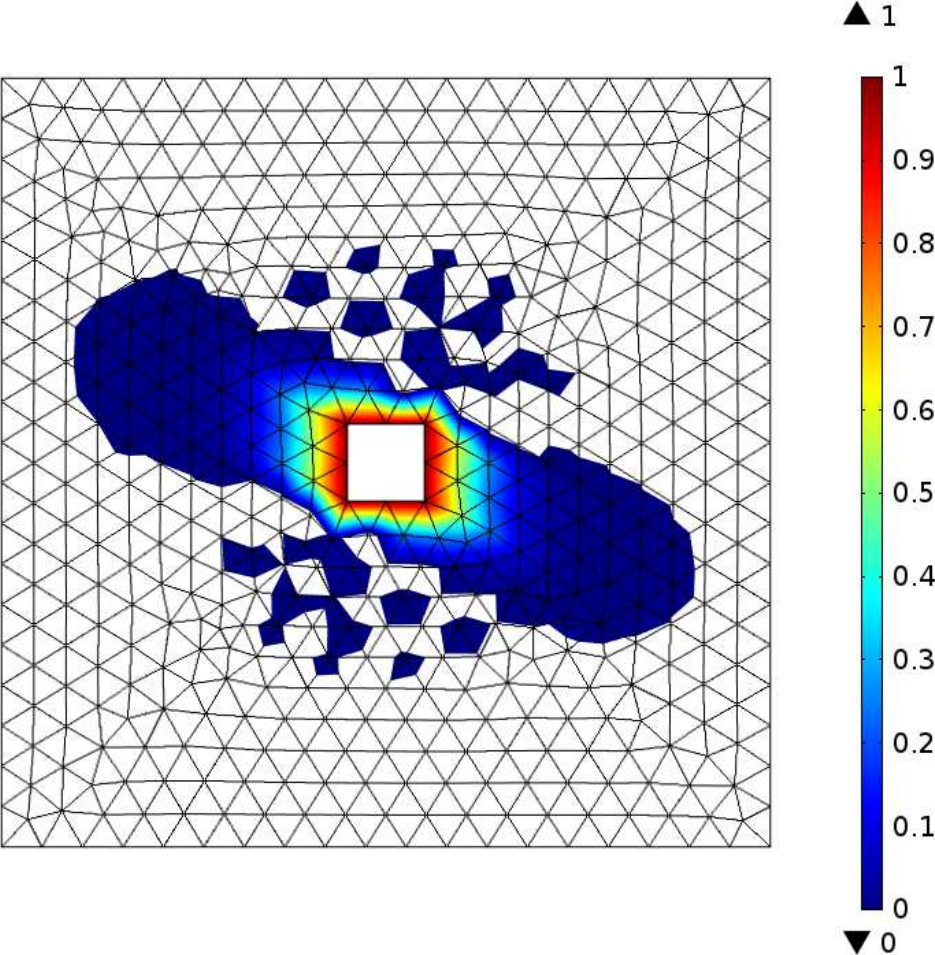}}
\subfigure[$\Delta t = 0.0001$, t = 0.05 (minimum = -0.05406)]{
\includegraphics[clip,scale=0.7]{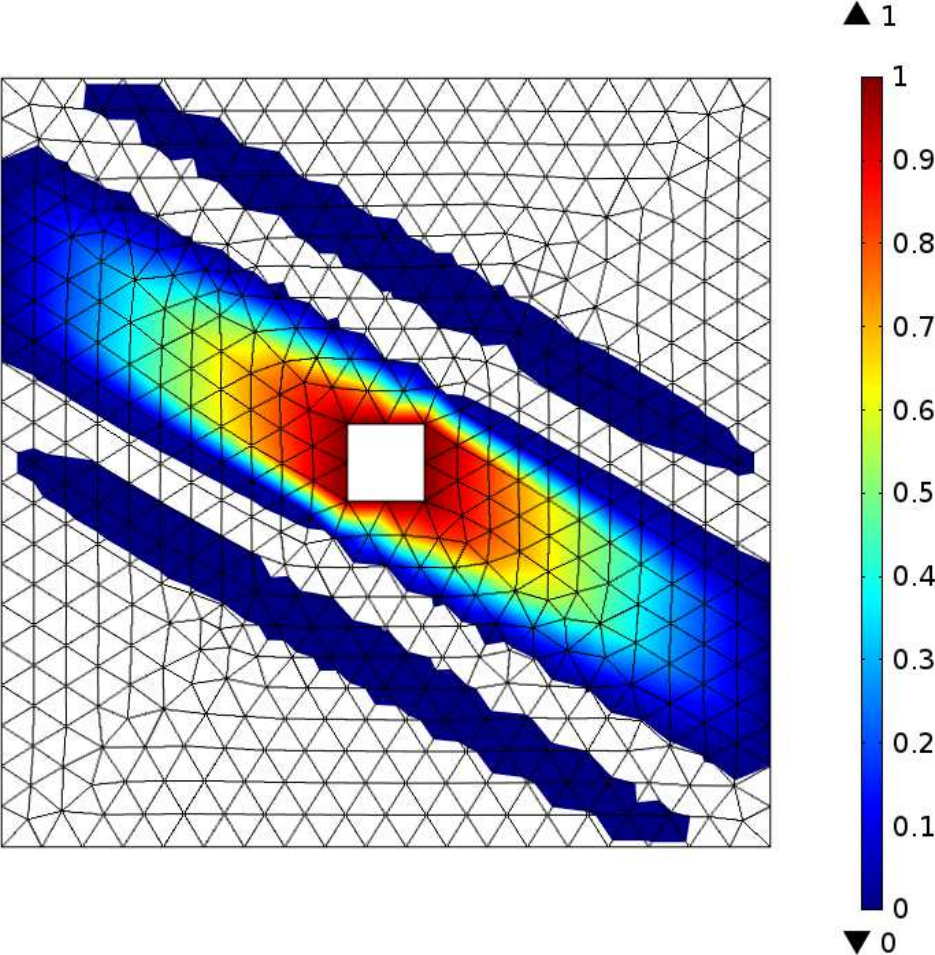}}
\subfigure[$\Delta t = 0.001$, $t = 3 \Delta t$ (minimum = -0.09527)]{
\includegraphics[clip,scale=0.7]{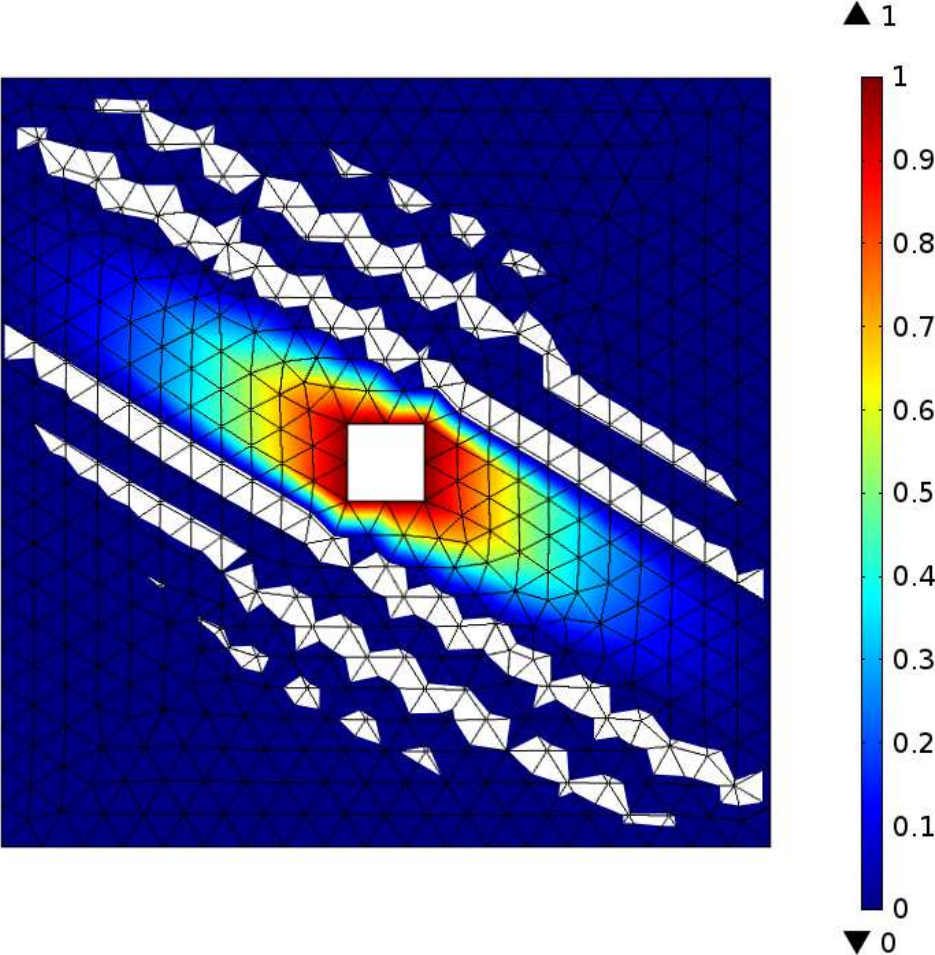}}
\subfigure[$\Delta t = 0.001$, t = 0.05 (minimum = -0.05406)]{
\includegraphics[clip,scale=0.7]{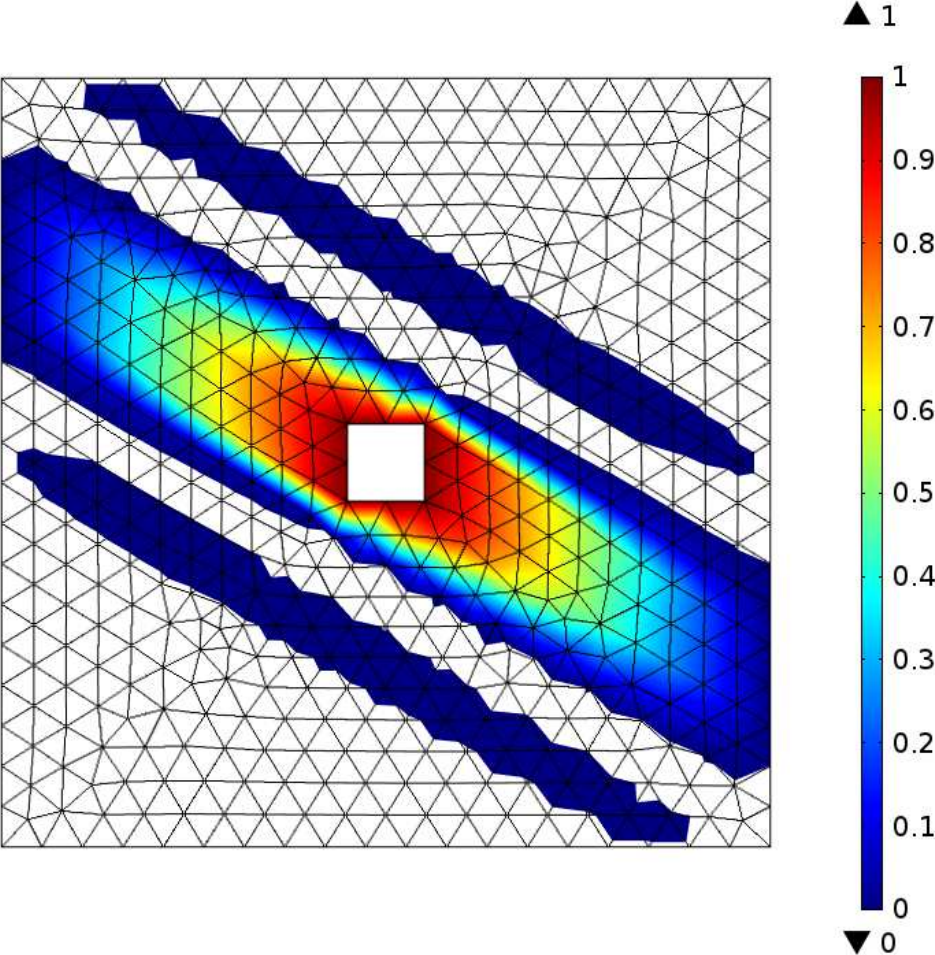}}
\caption{Anisotropic diffusion in square plate with a hole: Concentration 
profiles using COMSOL \cite{multiphysics2012version} by employing 
\emph{unstructured three-node triangular mesh}. The finite element 
mesh is also shown. The numerical results clearly violated the 
non-negative constraint for the concentration. The regions that 
violated the non-negative constraint are indicated in white color. 
\label{Fig:TransientDMP_COMSOL_T3}}
\end{figure}

\begin{figure}
\subfigure{
\label{Fig:TransientDMP_QP_code_Q4_mesh}
\includegraphics[clip,scale=0.35]{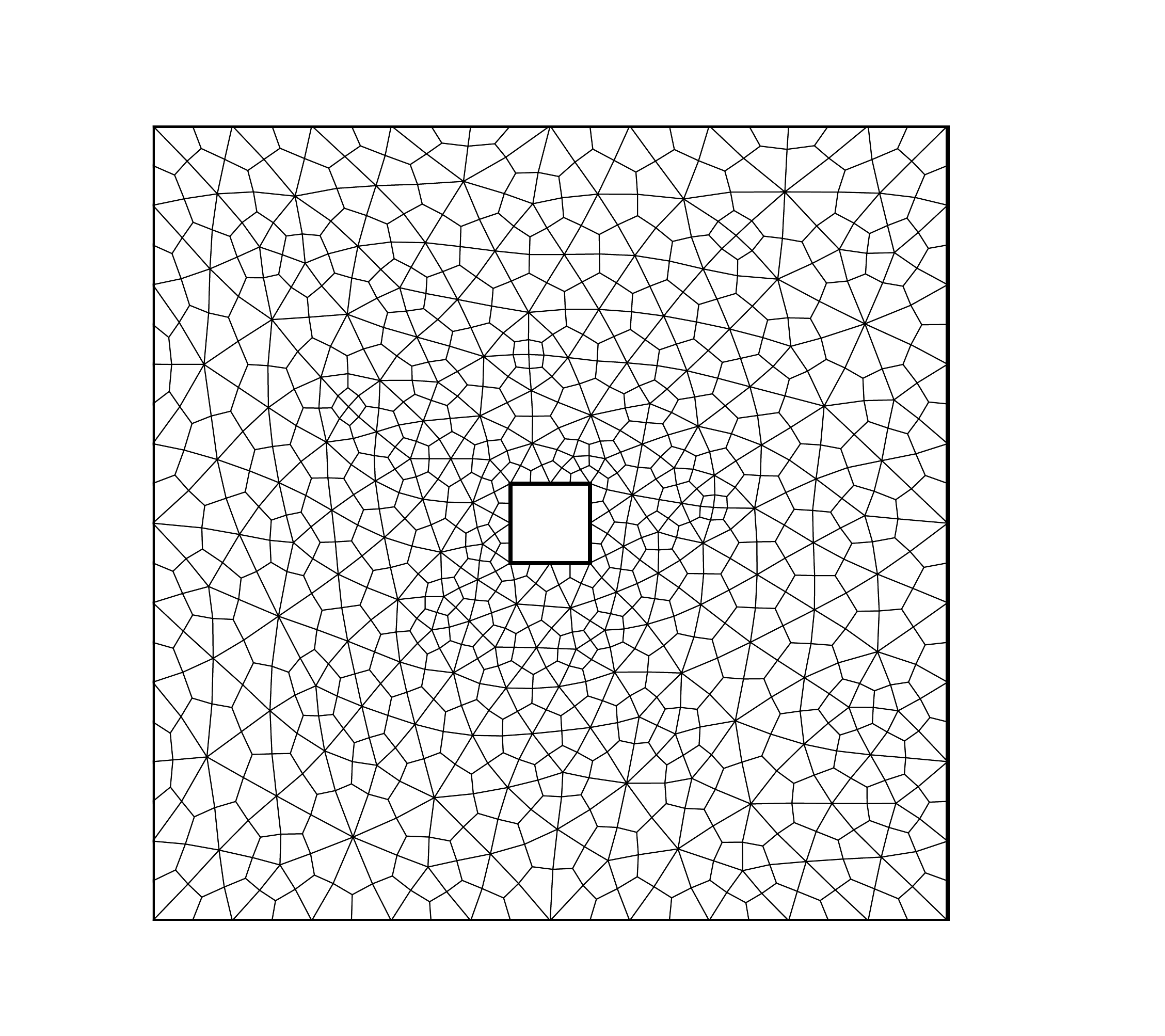}}
\subfigure{
\label{Fig:TransientDMP_QP_code_T3_mesh}
\includegraphics[clip,scale=0.35]{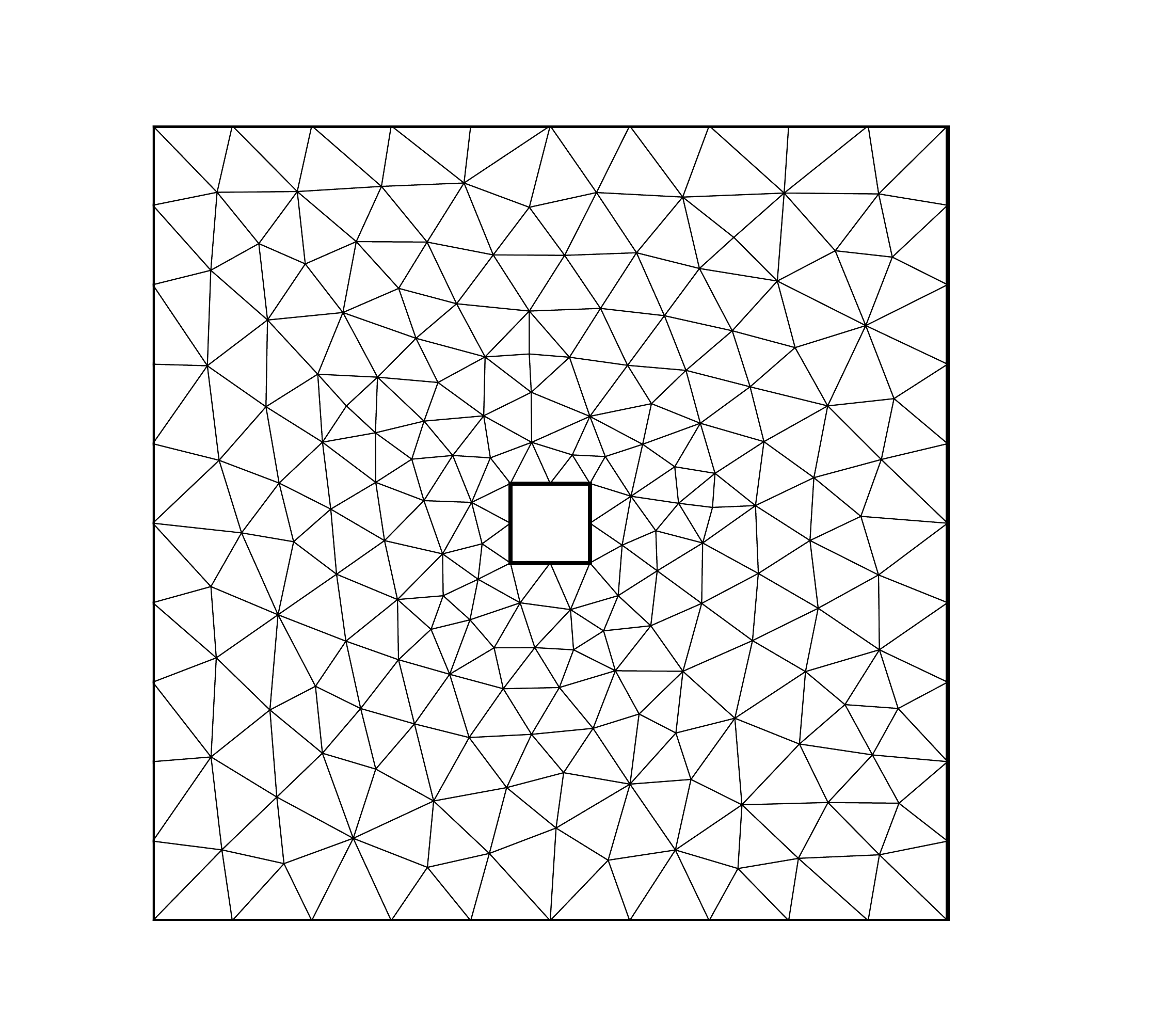}}
\caption{Anisotropic diffusion in a square plate with a hole: This figure 
shows the meshes employed in the numerical simulations using the 
proposed numerical methodology. The left figure shows an unstructured 
mesh based on four-node quadrilateral elements, and the right 
figure shows an unstructured mesh based on three-node triangular 
elements. The meshes are generated using GMSH \cite{gmsh.www}. 
\label{Fig:TransientDMP_QP_code_meshes}}
\end{figure}

\begin{figure}
\subfigure[$\Delta t = 0.0001$, $ t= 3 \Delta t$]{
\includegraphics[clip,scale=0.35]{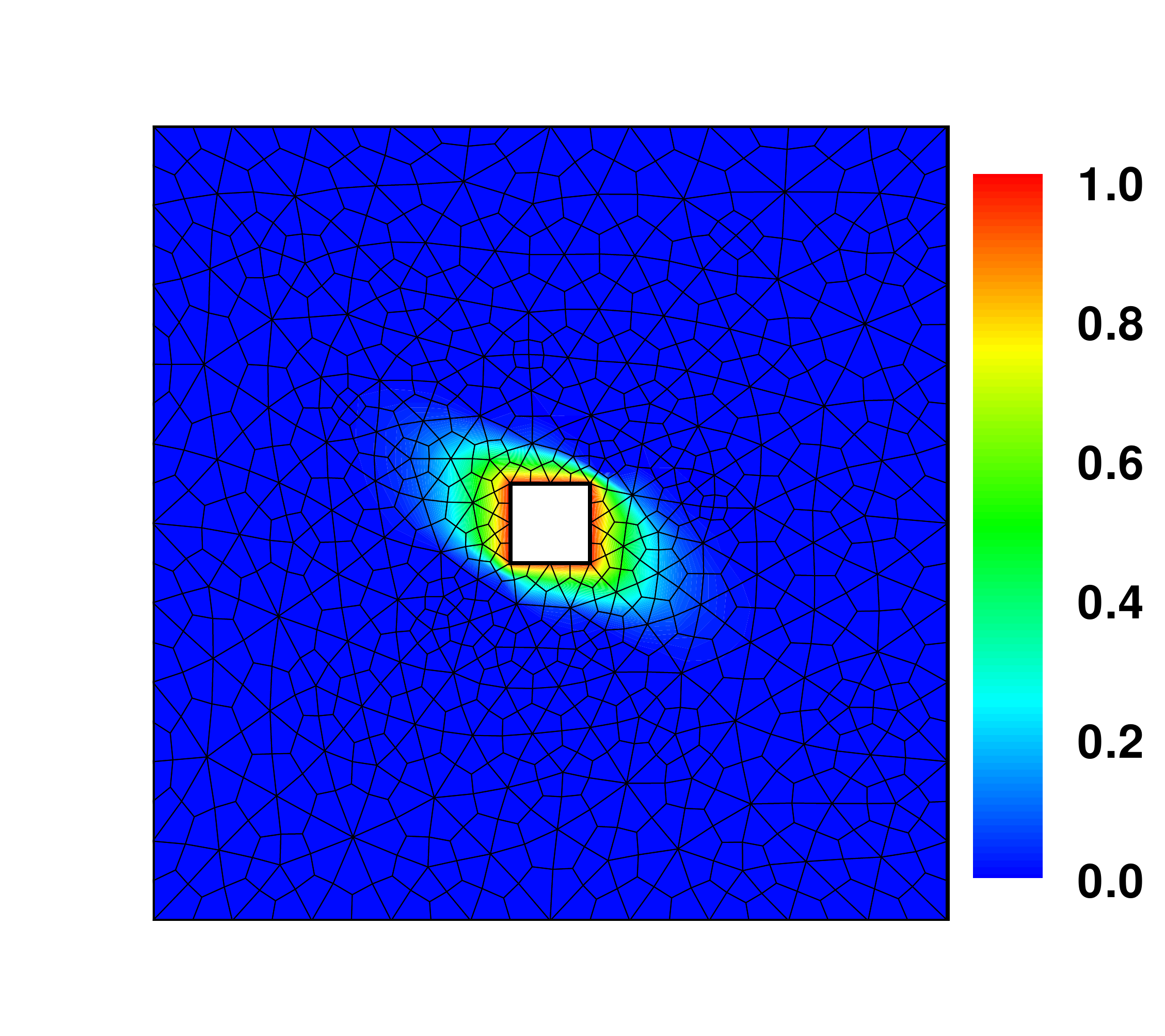}}
\subfigure[$\Delta t = 0.0001$, t = 0.05]{
\includegraphics[clip,scale=0.35]{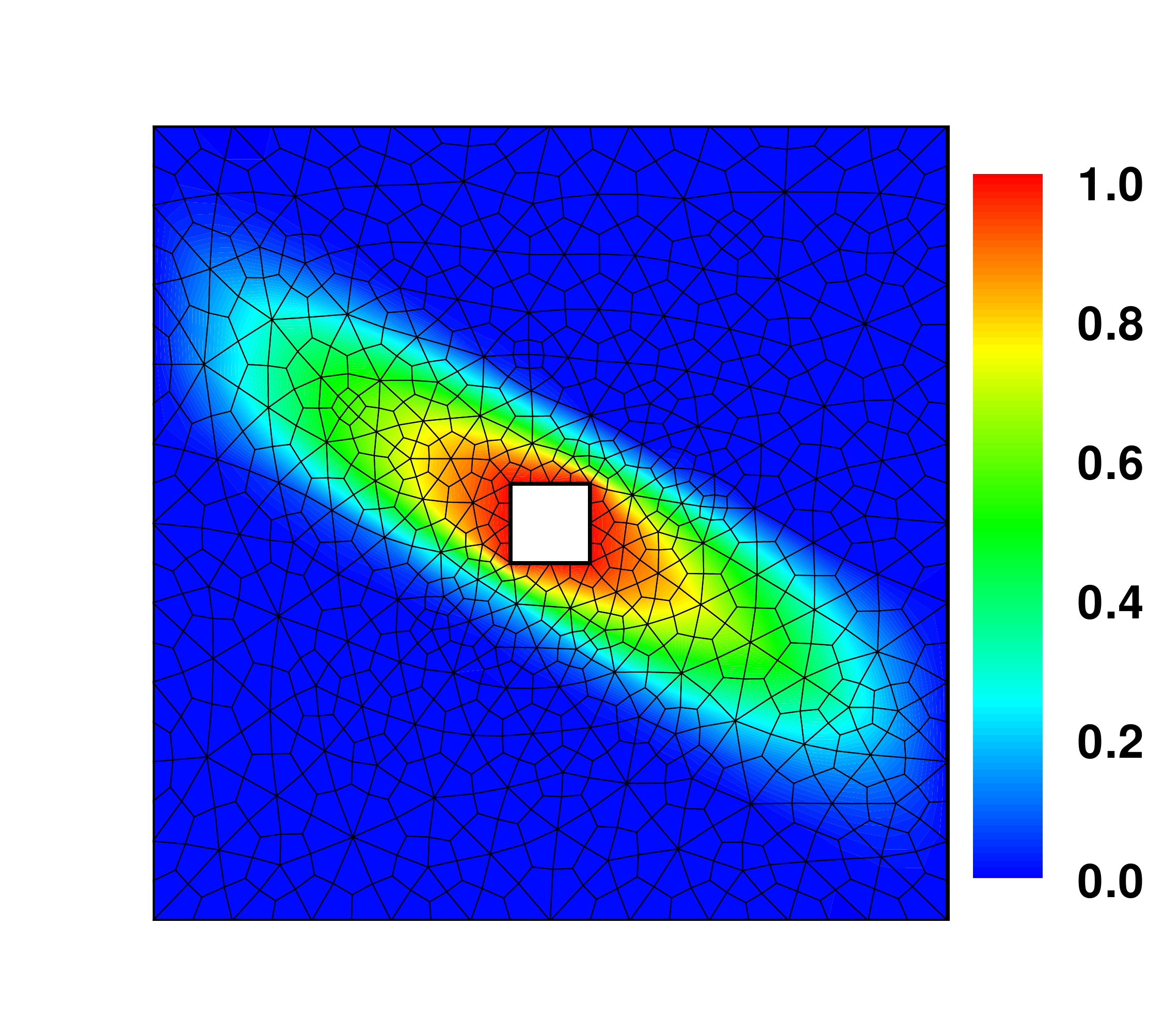}}
\subfigure[$\Delta t = 0.001$, $t = 3 \Delta t$]{
\includegraphics[clip,scale=0.35]{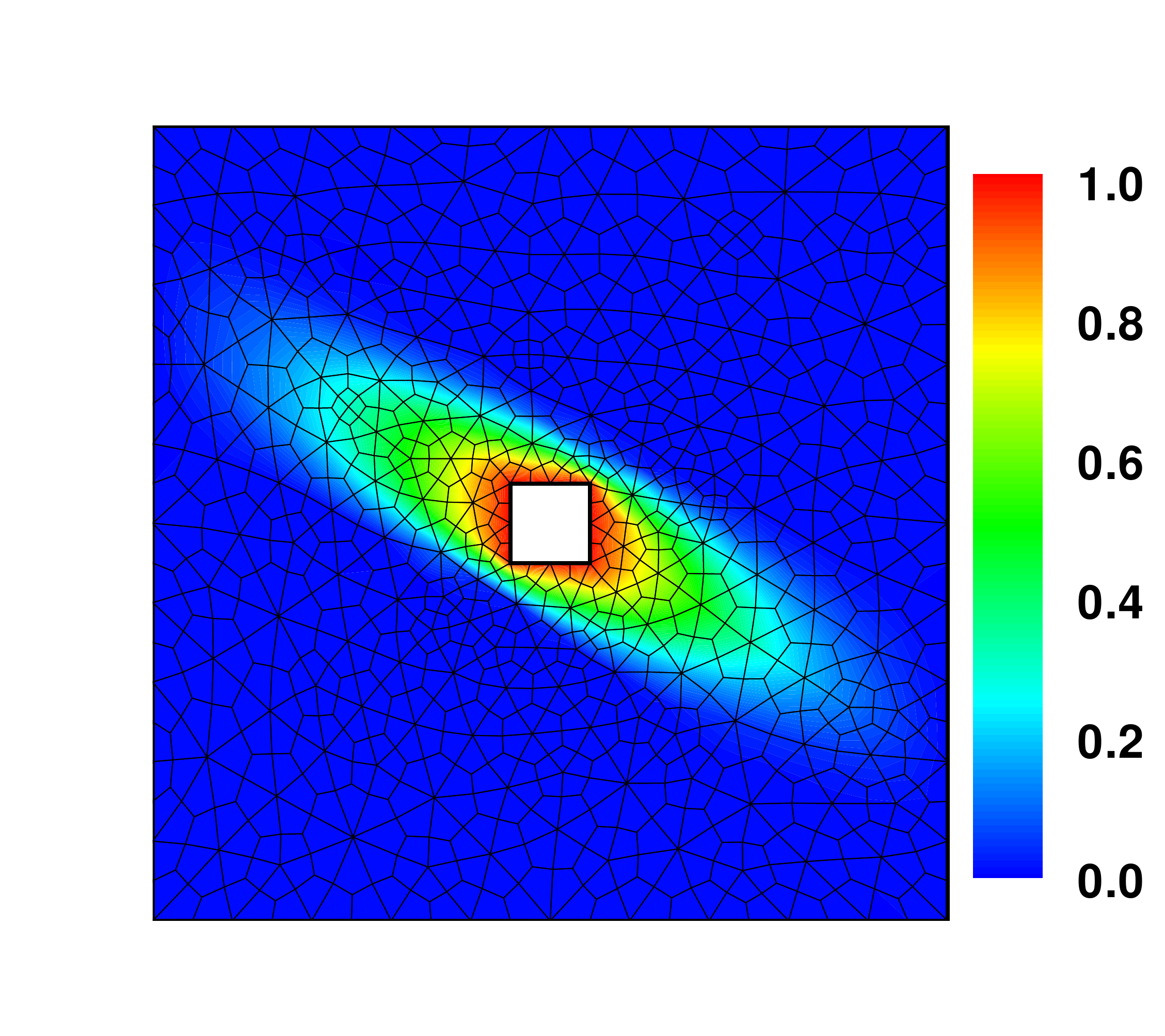}}
\subfigure[$\Delta t = 0.001$, t = 0.05]{
\includegraphics[clip,scale=0.35]{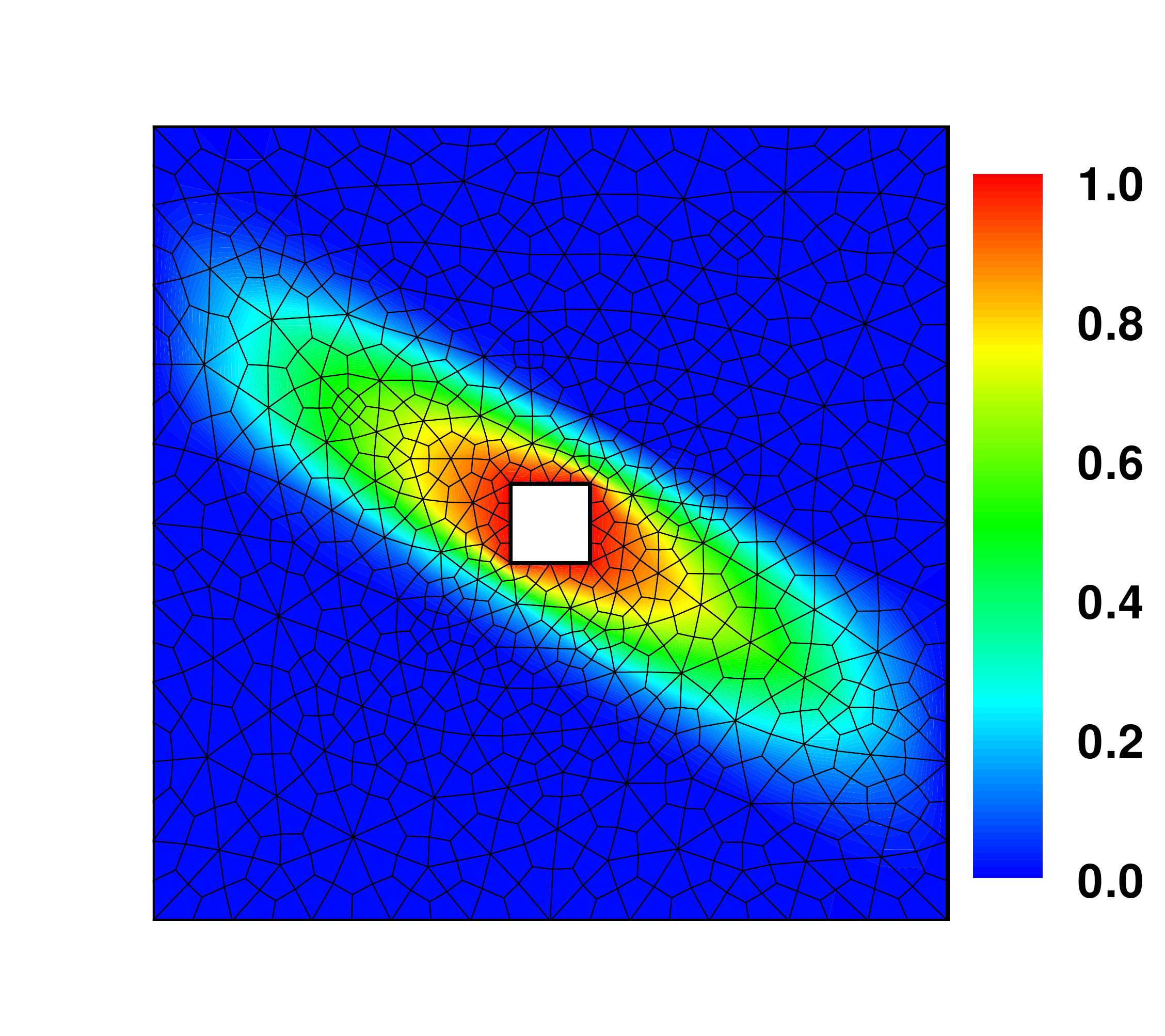}}
\caption{Anisotropic diffusion in square plate with a hole: 
Concentration profiles using the proposed methodology 
by employing \emph{unstructured four-node triangular mesh}, 
which is shown in figure \ref{Fig:TransientDMP_QP_code_Q4_mesh}. 
The numerical results satisfy the maximum principle and the non-negative 
constraint. The numerical results are visualized using Tecplot \cite{Tecplot360}. 
\label{Fig:TransientDMP_QP_code_conc_Q4}}
\end{figure}

\begin{figure}
\subfigure[$\Delta t = 0.0001$, $ t= 3 \Delta t$]{
\includegraphics[clip,scale=0.35]{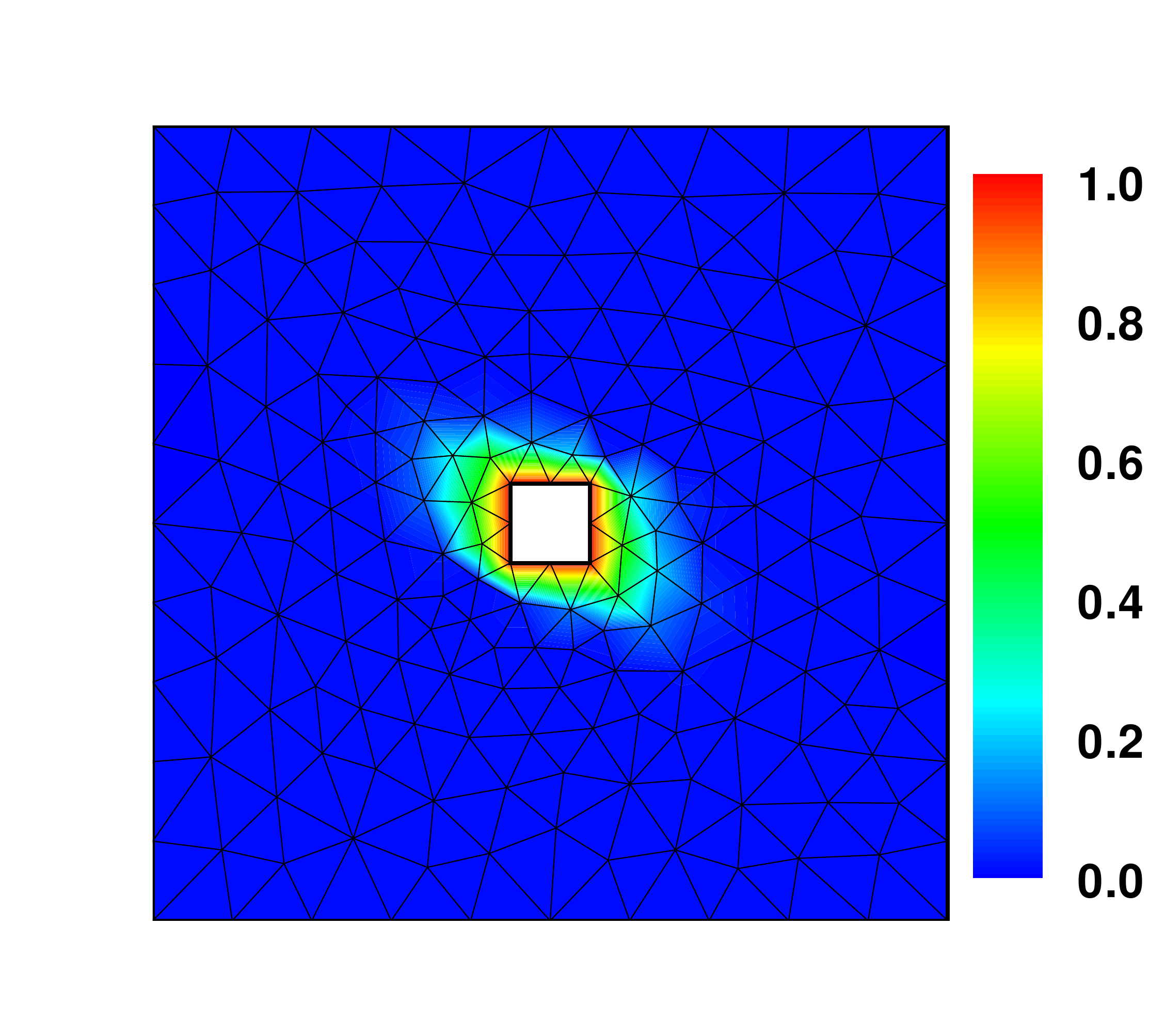}}
\subfigure[$\Delta t = 0.0001$, t = 0.05]{
\includegraphics[clip,scale=0.35]{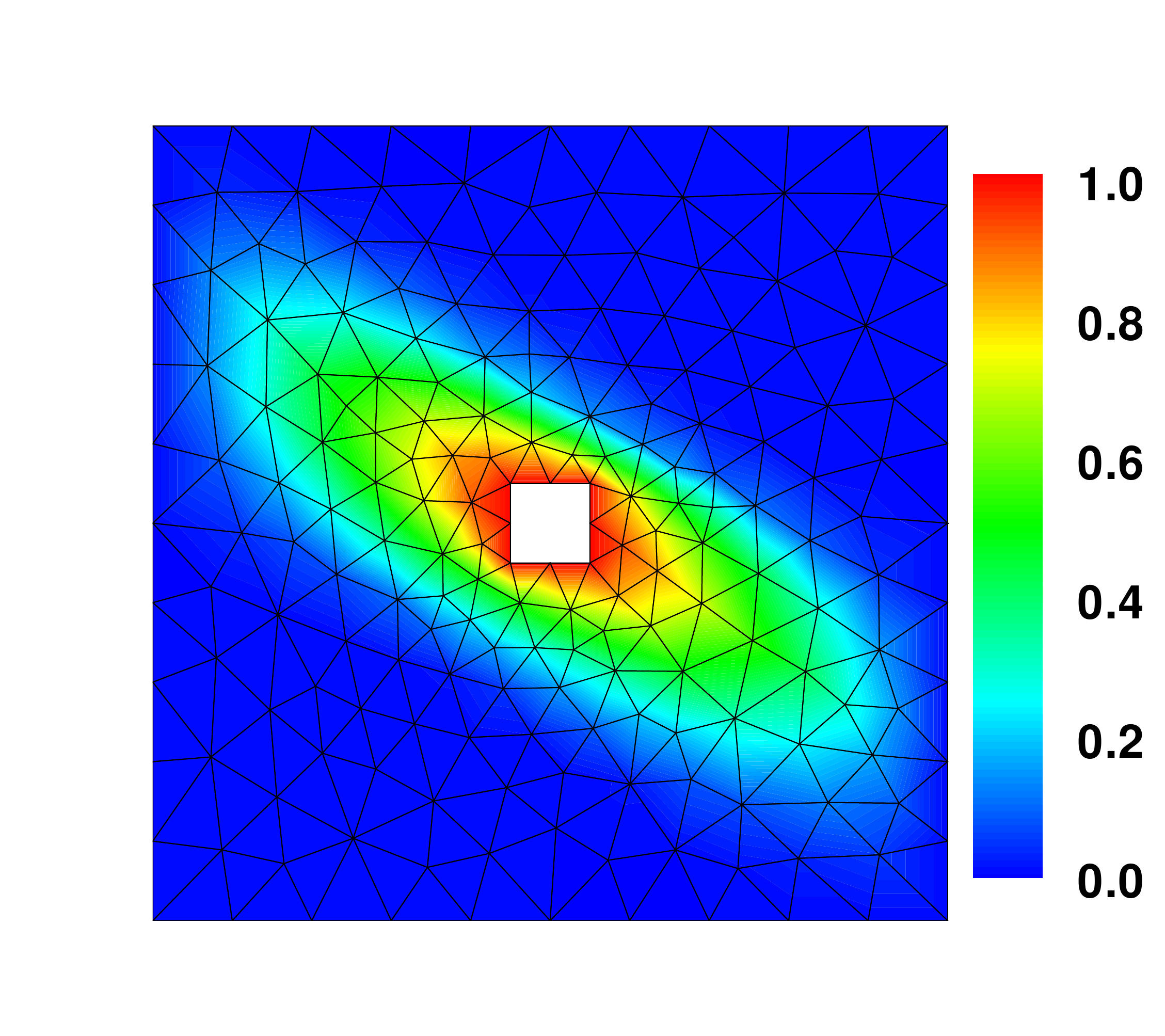}}
\subfigure[$\Delta t = 0.001$, $t = 3 \Delta t$]{
\includegraphics[clip,scale=0.35]{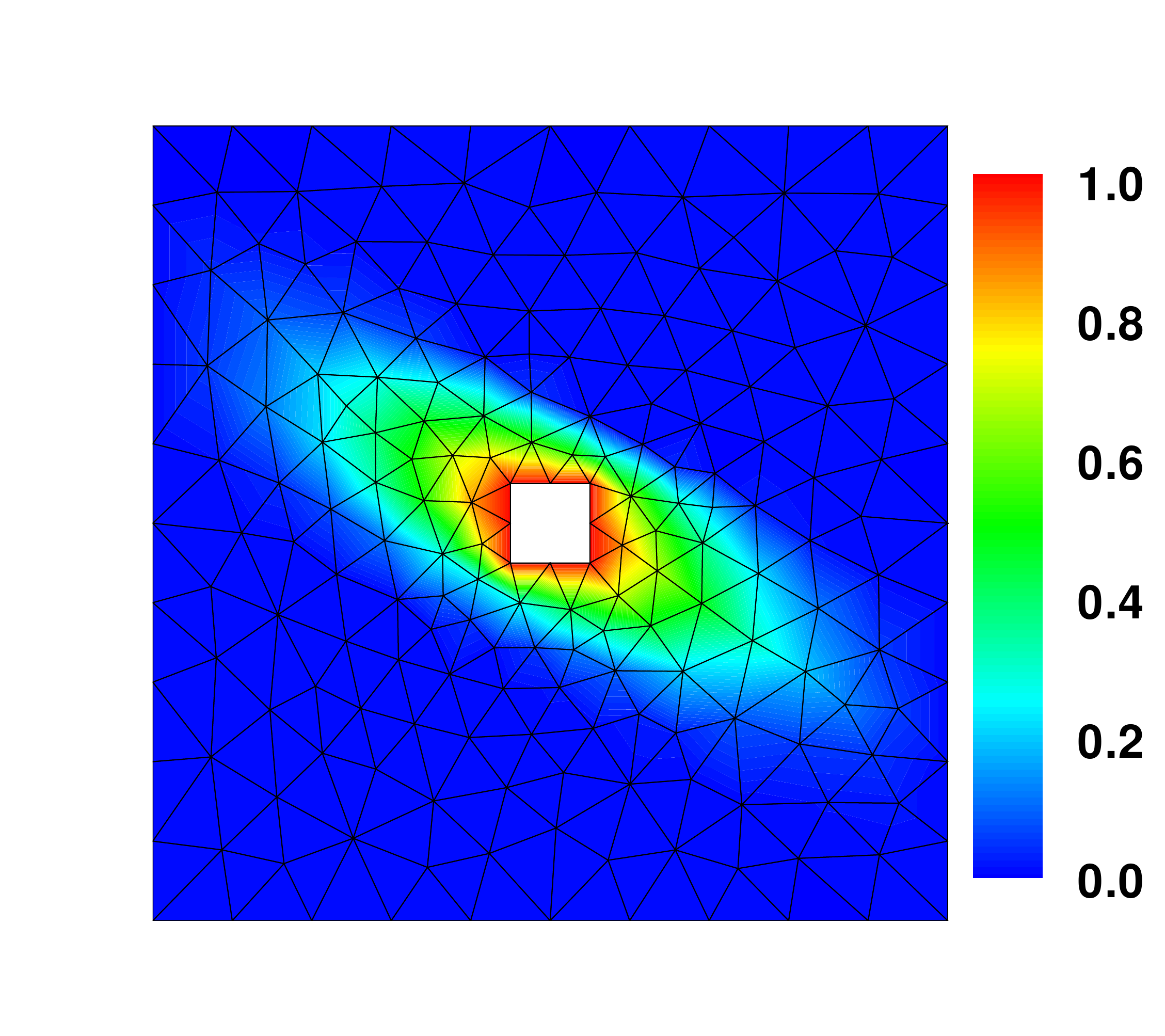}}
\subfigure[$\Delta t = 0.001$, t = 0.05]{
\includegraphics[clip,scale=0.35]{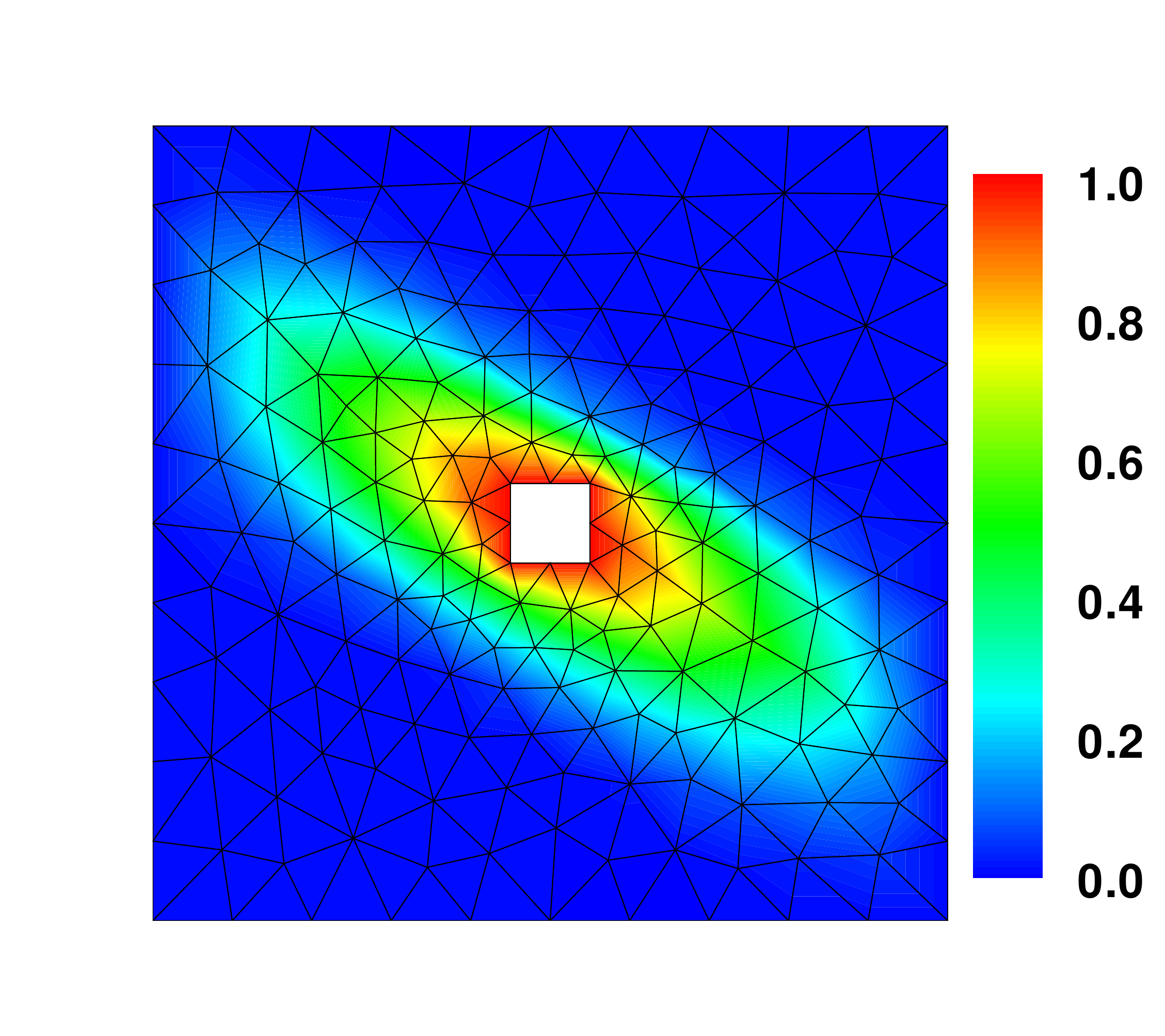}}
\caption{Anisotropic diffusion in square plate with a hole: 
Concentration profiles using the proposed methodology 
by employing \emph{unstructured three-node triangular mesh}, 
which is shown in figure \ref{Fig:TransientDMP_QP_code_T3_mesh}. 
The numerical results satisfy the maximum principle and the 
non-negative constraint. The numerical results are visualized 
using Tecplot \cite{Tecplot360}.\label{Fig:TransientDMP_QP_code_conc_T3}}
\end{figure}

\begin{figure}
\subfigure[$\Delta t = 0.0001$, $ t= 3 \Delta t$ (minimum = -0.01024)]{
\includegraphics[clip,scale=0.35]{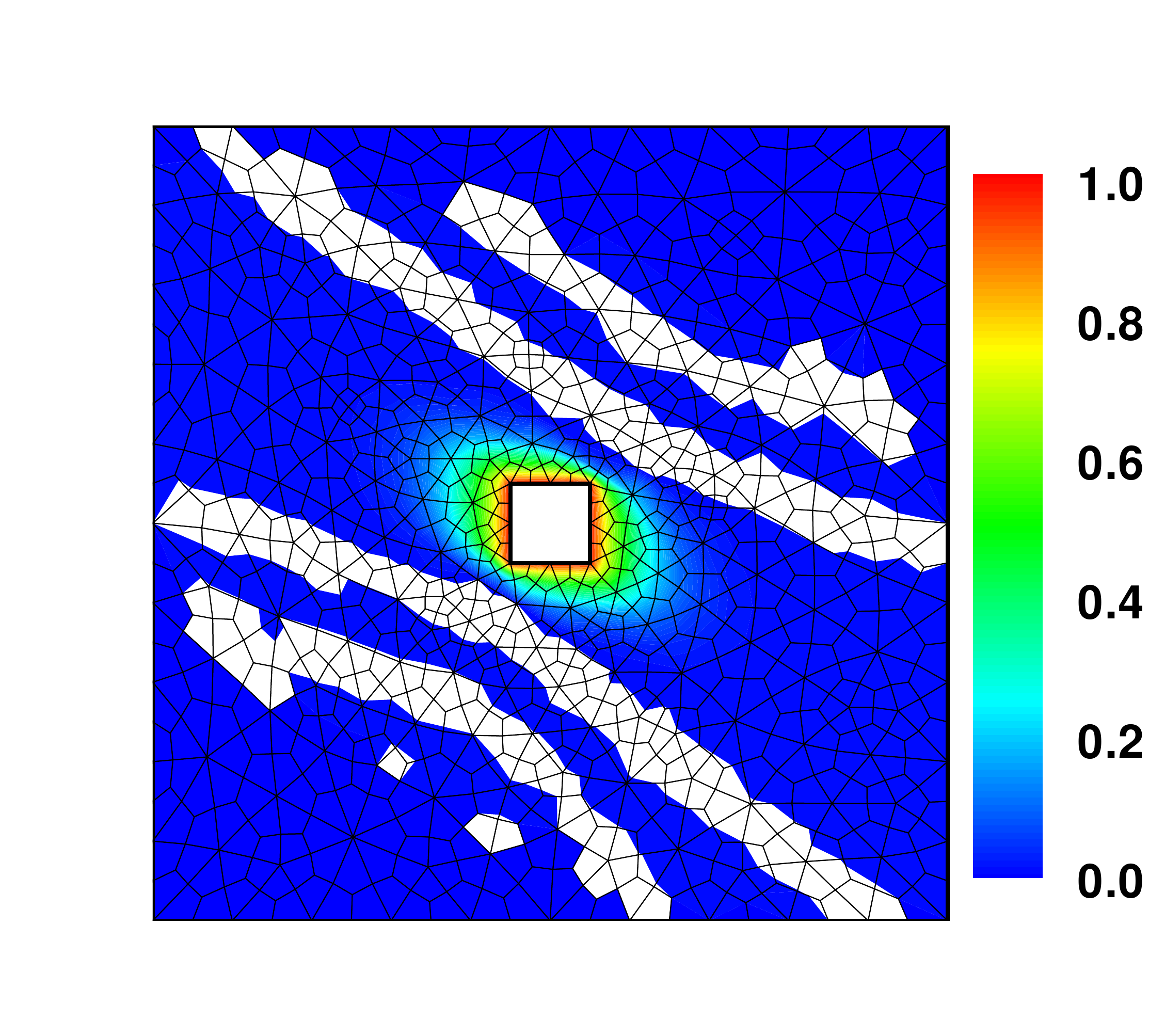}}
\subfigure[$\Delta t = 0.001$, $t = 3 \Delta t$ (minimum = -0.03603)]{
\includegraphics[clip,scale=0.35]{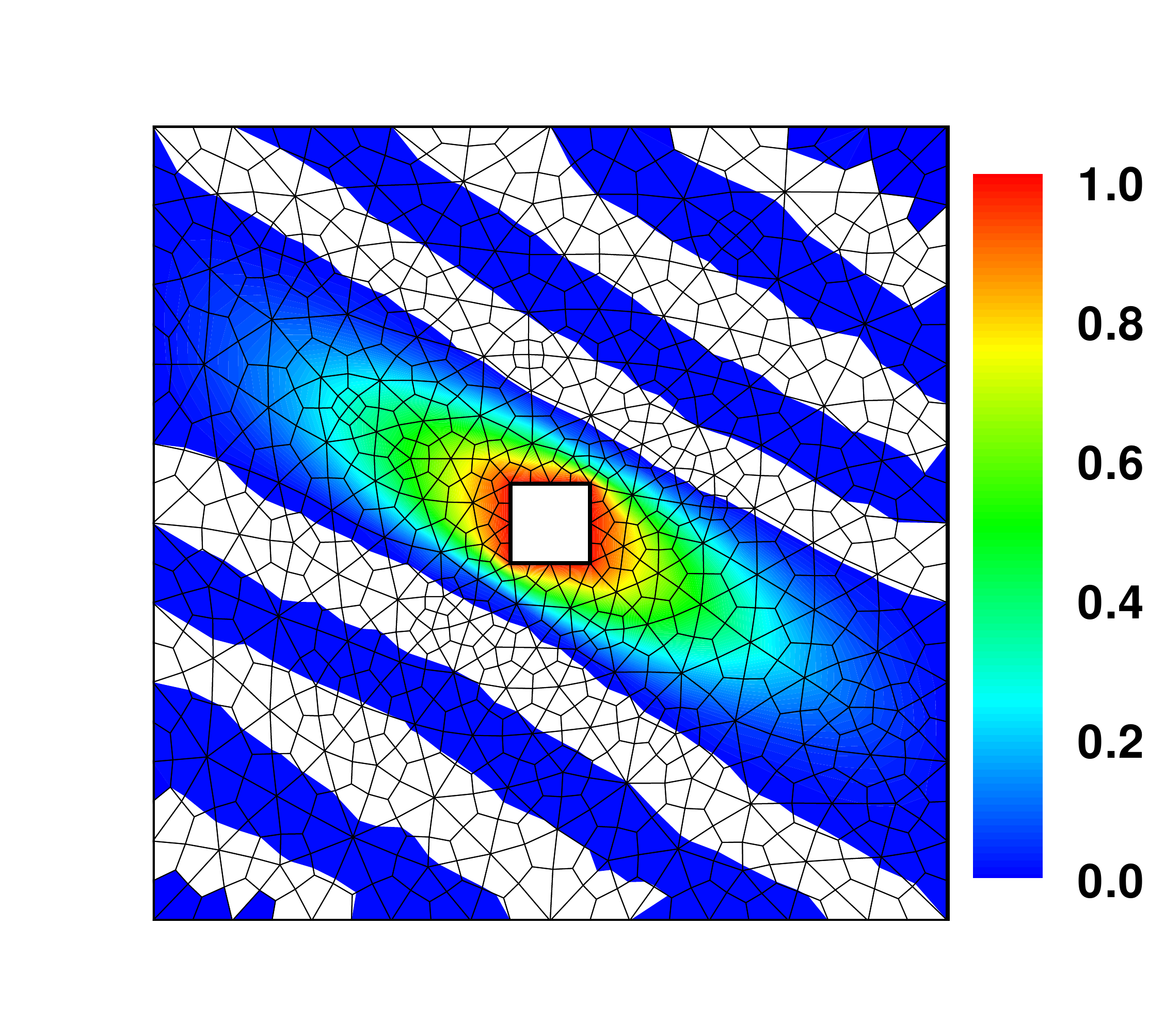}}
\caption{Anisotropic diffusion in square plate with a hole: 
This figure shows the concentration profiles obtained 
using the backward Euler time stepping scheme 
($\alpha_f = \alpha_m = \gamma = 1)$ and lumped 
capacity matrix approach.  
The unstructured four-node quadrilateral mesh shown in 
figure \ref{Fig:TransientDMP_QP_code_Q4_mesh} is 
used in the numerical simulation. 
Clearly, the numerical results do not satisfy the maximum principle 
and the non-negative constraint. In the case of isotropic diffusion, 
employing the backward Euler time-stepping scheme with lumped 
capacity matrix approach can be employed to satisfy maximum 
principles and the non-negative constraint (with some restrictions 
on the mesh). As it is evident from this figure, meeting these conditions 
is not sufficient in the case of transient anisotropic diffusion. The regions 
of the violation of the non-negative constraint are shown in white color. 
The numerical results are visualized using Tecplot \cite{Tecplot360}.
\label{Fig:TransientDMP_BE_lumped}}
\end{figure}


\begin{figure}[htp]
  \centering
    \includegraphics[clip,scale=0.4]{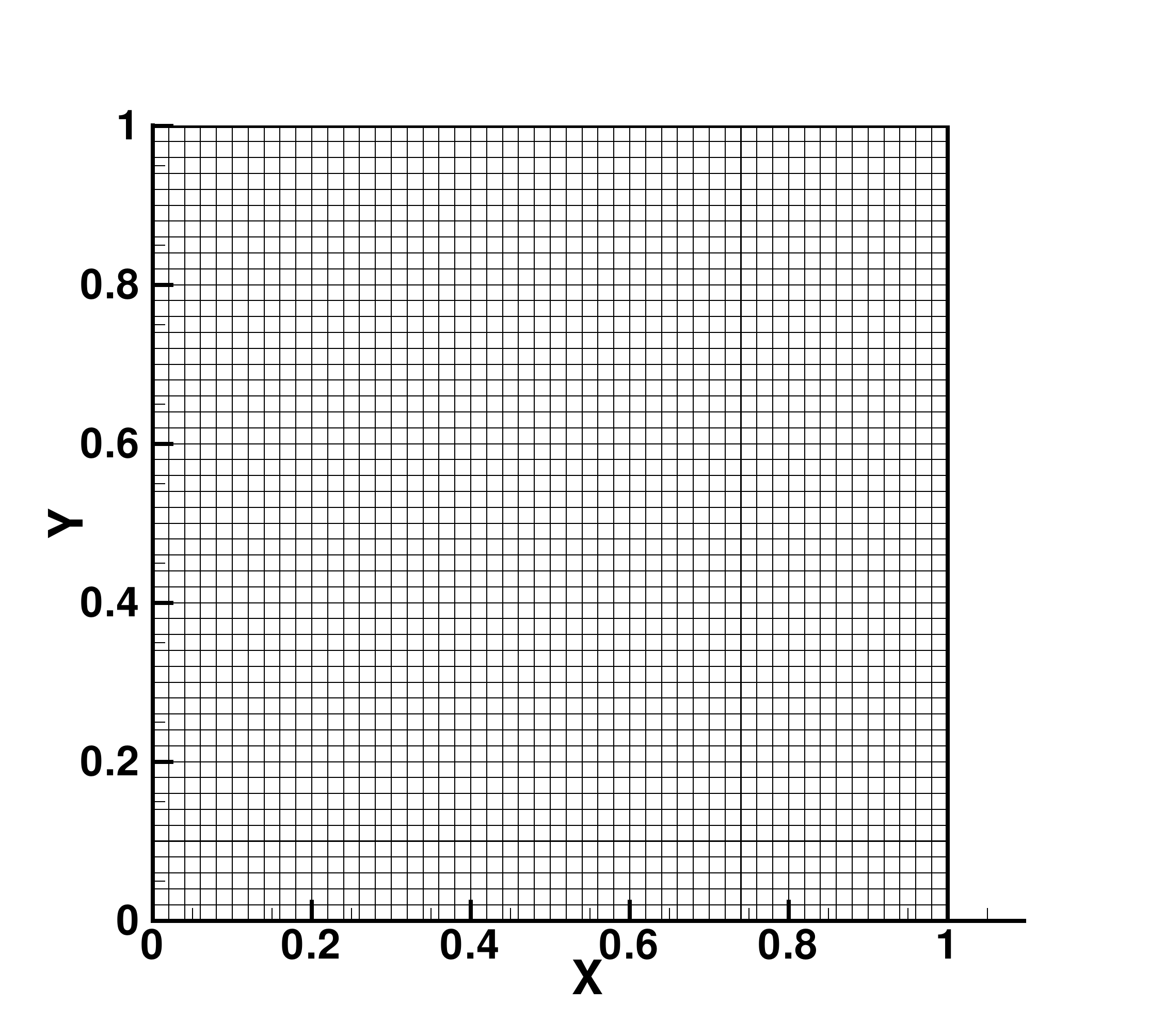}
  \caption{Diffusion in heterogeneous anisotropic medium: This 
    figure shows a typical computational mesh used in the numerical 
    simulations. The mesh in the figure is made of four node quadrilateral 
    finite elements with XSeed = YSeed = 51.  Note that XSeed and 
    YSeed, respectively, denote the number of nodes along the 
    x-direction and y-direction. A similar mesh with XSeed = YSeed 
    = 101 is also used in the numerical simulations. 
    \label{Fig:TransientDMP_2D_heterogeneous_mesh}}
\end{figure}

\begin{figure}[htp]
  \centering
  \includegraphics[clip,scale=0.45]{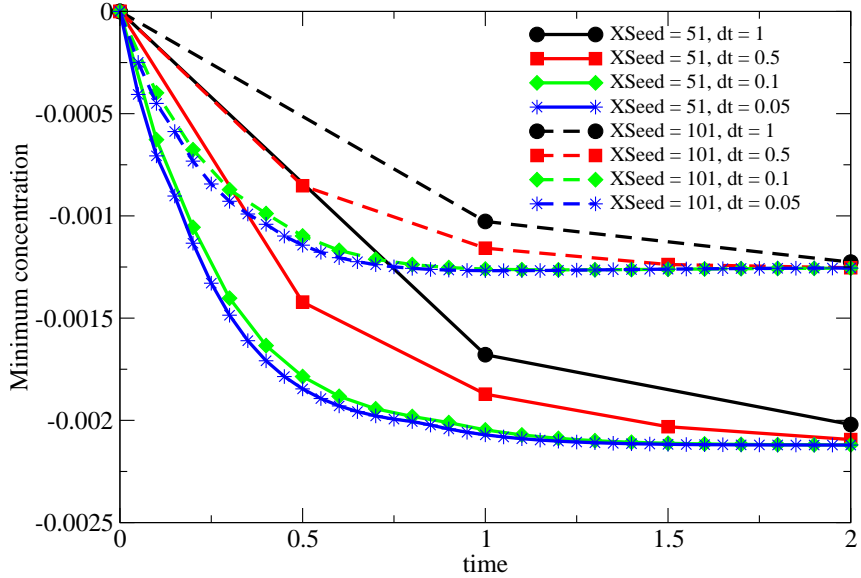}
  \caption{Diffusion in heterogeneous anisotropic medium: 
    This figure shows the variation of the minimum concentration 
    under the single-field formulation. The results are shown for 
    two different meshes (XSeed = YSeed = 51 and 101). Note that 
    XSeed and YSeed denote the number of nodes along x-direction 
    and y-direction, respectively. Various time steps ($\Delta t 
    = 0.05, \; 0.1, \; 0.5$ and $1$) are employed.  The rationale 
    behind the choice of the time steps is that the (transient) 
    solution is very close to the steady-state response for 
    times greater than 2. Hence, any time step bigger than the 
    ones used in the numerical simulation does not capture the 
    transient features of the problem, and will not be appropriate 
    for a transient analysis. Any smaller time step will result 
    in bigger violation of the non-negative constraint, which 
    will be evident from the numerical results. The single-field 
    formulation produced negative concentrations for both the 
    meshes and for both the time steps. The proposed methodology 
    produced non-negative solutions under all the cases considered. 
    \label{Fig:TransientDMP_2D_heterogeneous_min_conc}}
\end{figure}

\begin{figure}[htp]
  \centering
    \includegraphics[clip,scale=0.45]{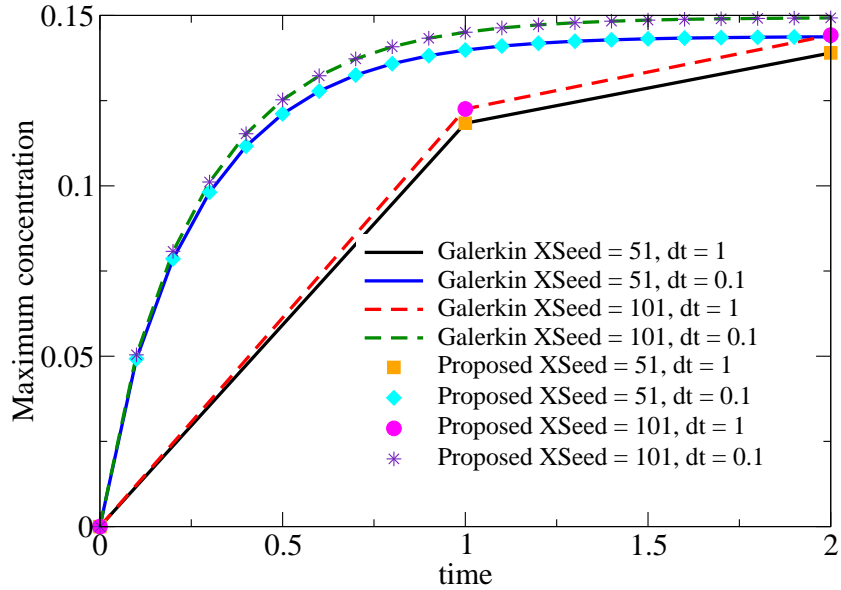}
  \caption{Diffusion in heterogeneous anisotropic medium: This 
    figure shows the variation of the maximum concentration under 
    the Galerkin single-field formulation and the proposed methodology. 
    The results are shown for two different meshes (XSeed = YSeed = 
    51 and 101), and for two time steps ($\Delta t = 1 \; \mathrm{and} 
    \; \Delta t = 0.1$). Note that XSeed and YSeed denote the number 
    of nodes along the x-direction and y-direction, respectively. As evident 
    from the figure, the Galerkin single-field formulation and the proposed 
    methodology produced similar results for the maximum concentration 
    with respect to time. \label{Fig:TransientDMP_2D_heterogeneous_max_conc}}
\end{figure}

\begin{figure}[htp]
  \centering
  \subfigure{
    \includegraphics[clip,scale=0.31]{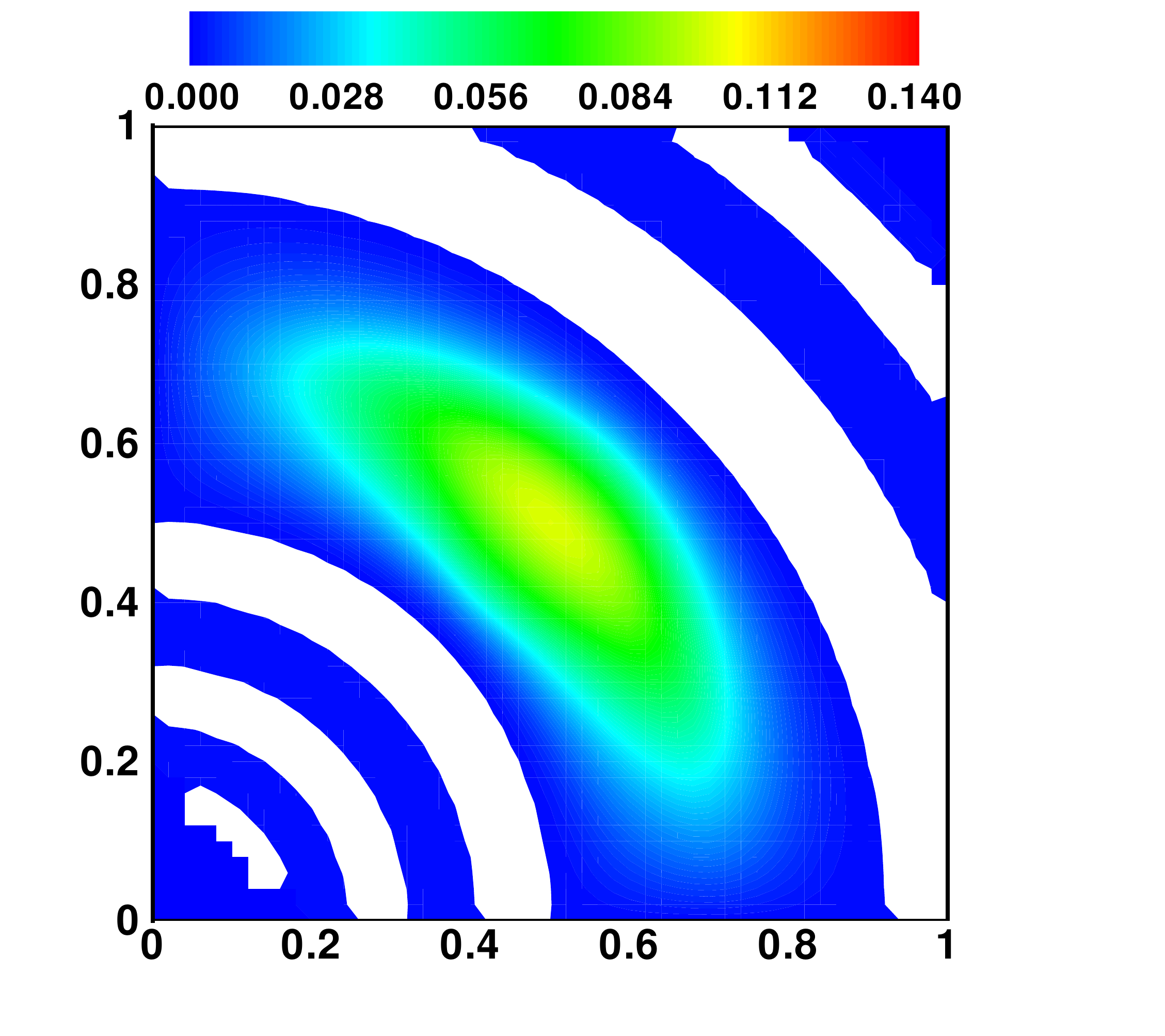}}
  \subfigure{
    \includegraphics[clip,scale=0.31]{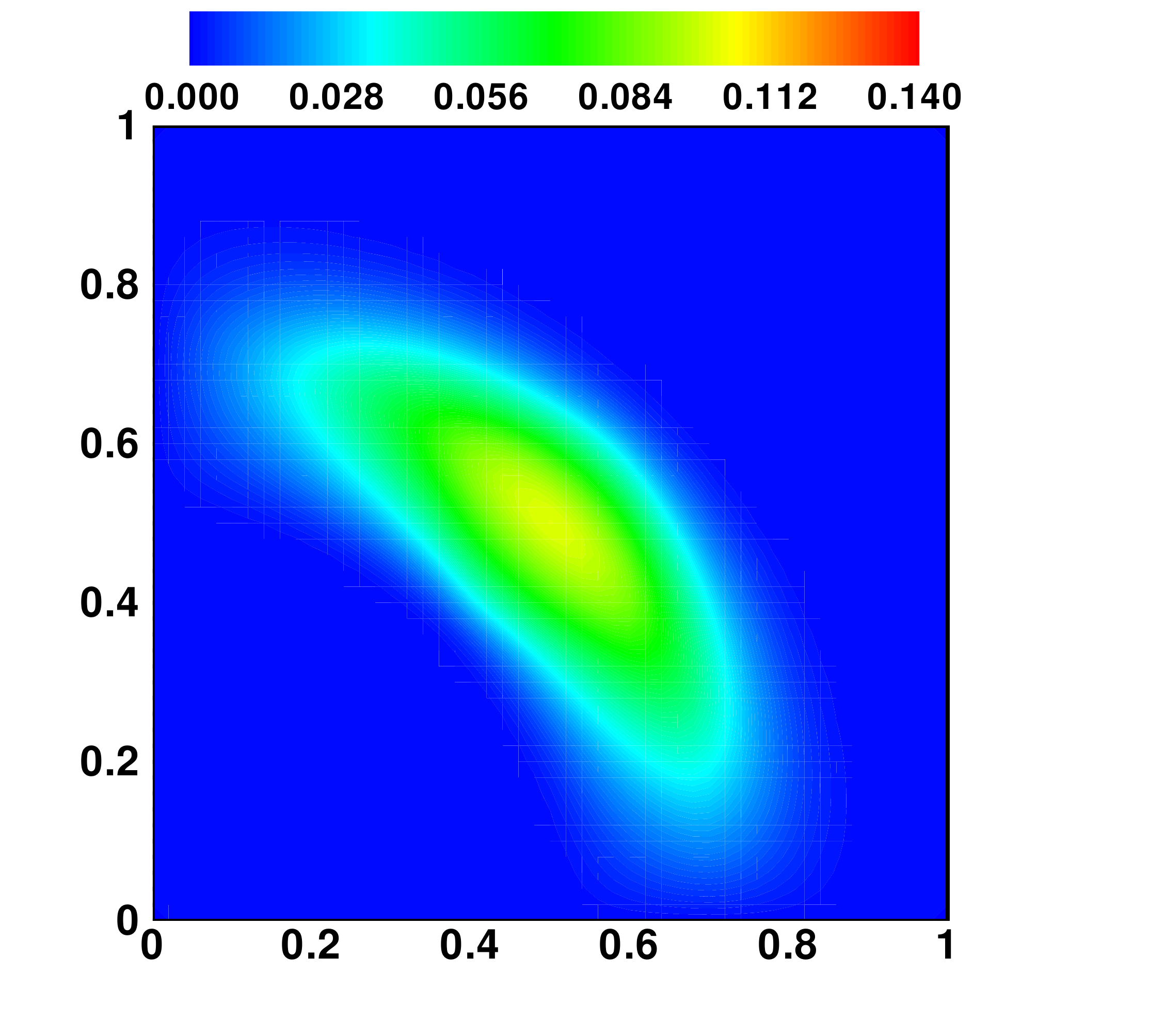}}
 \subfigure{
   \includegraphics[clip,scale=0.31]{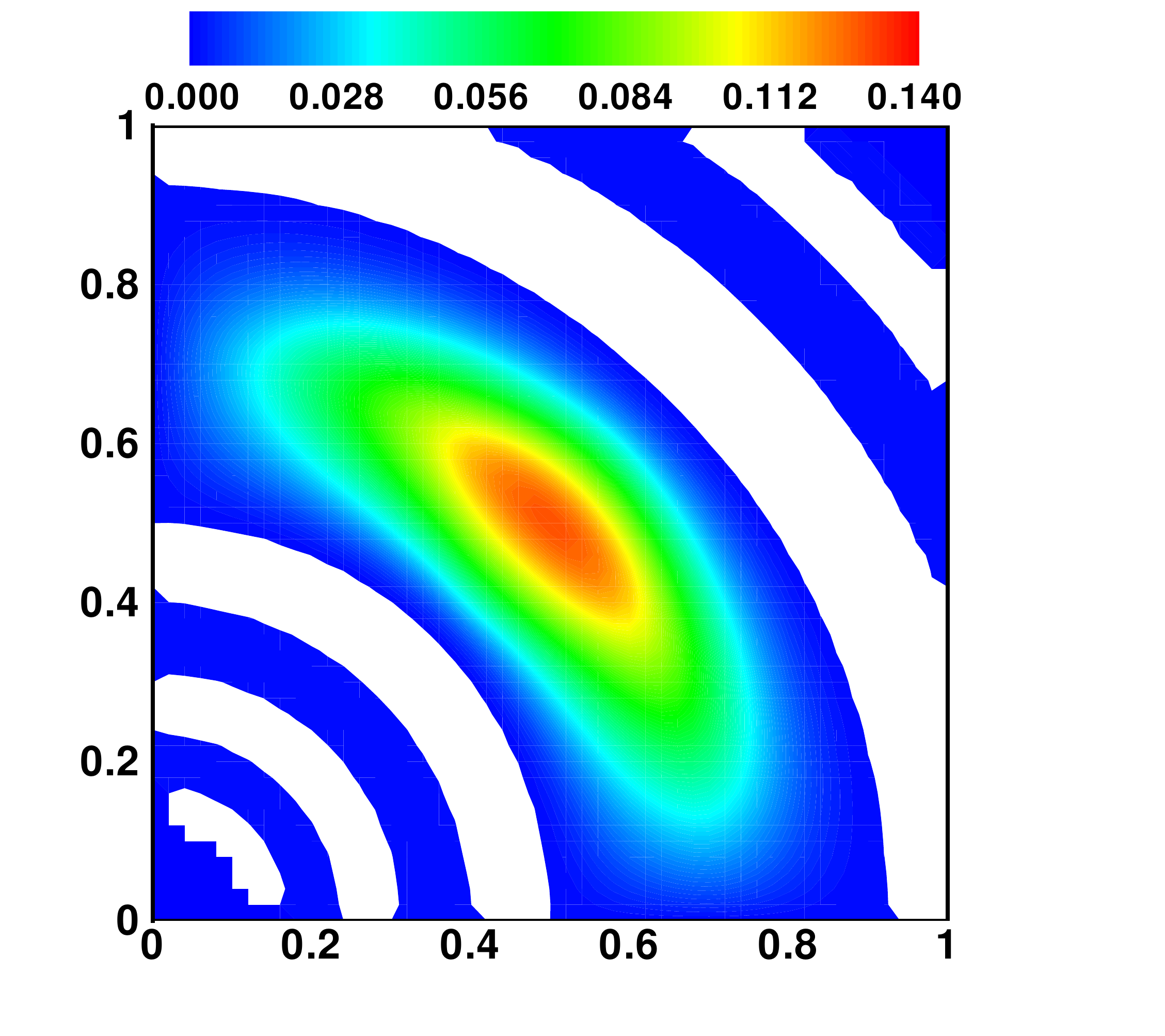}}
  \subfigure{
    \includegraphics[clip,scale=0.31]{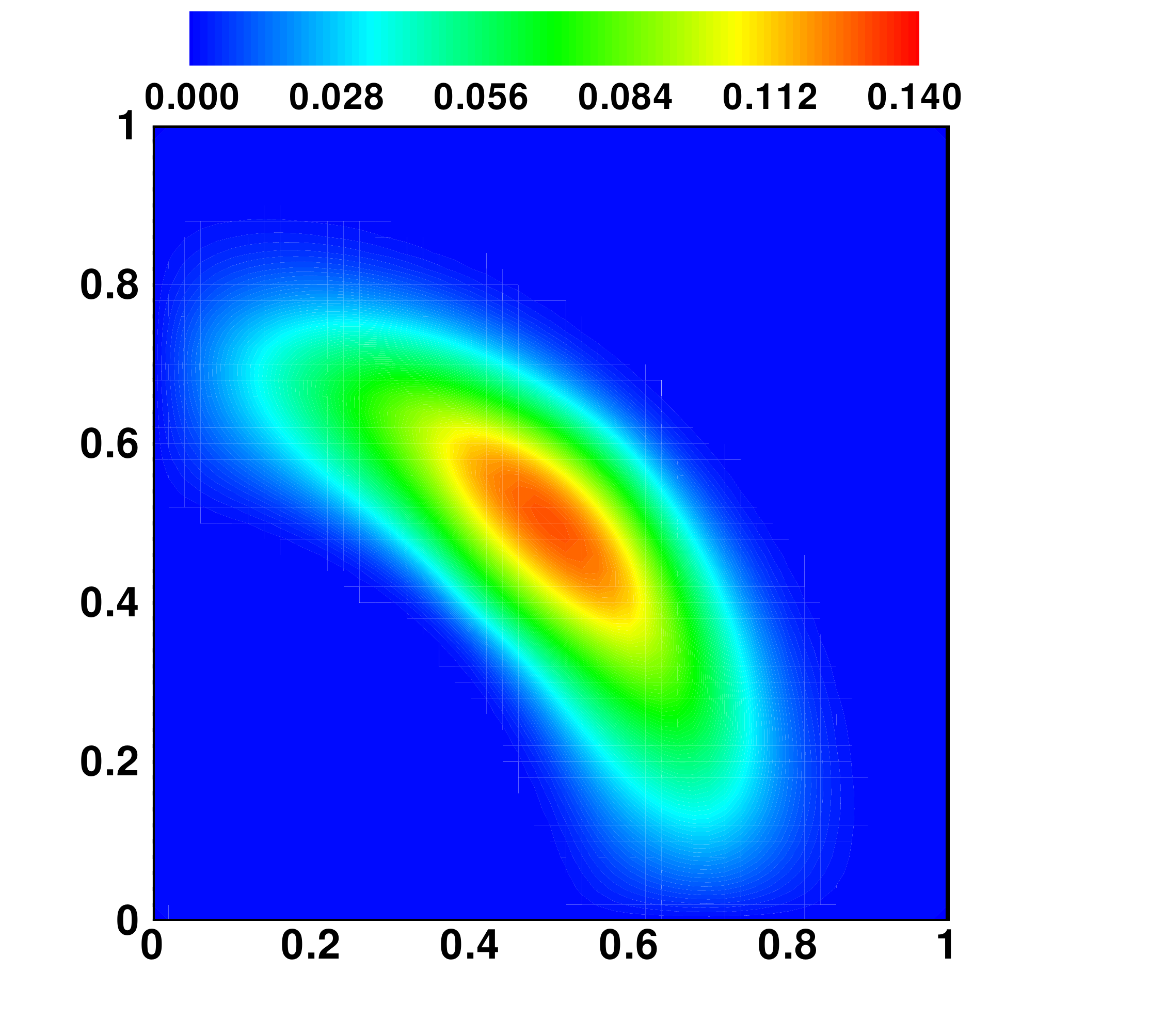}}
  \subfigure{
   \includegraphics[clip,scale=0.31]{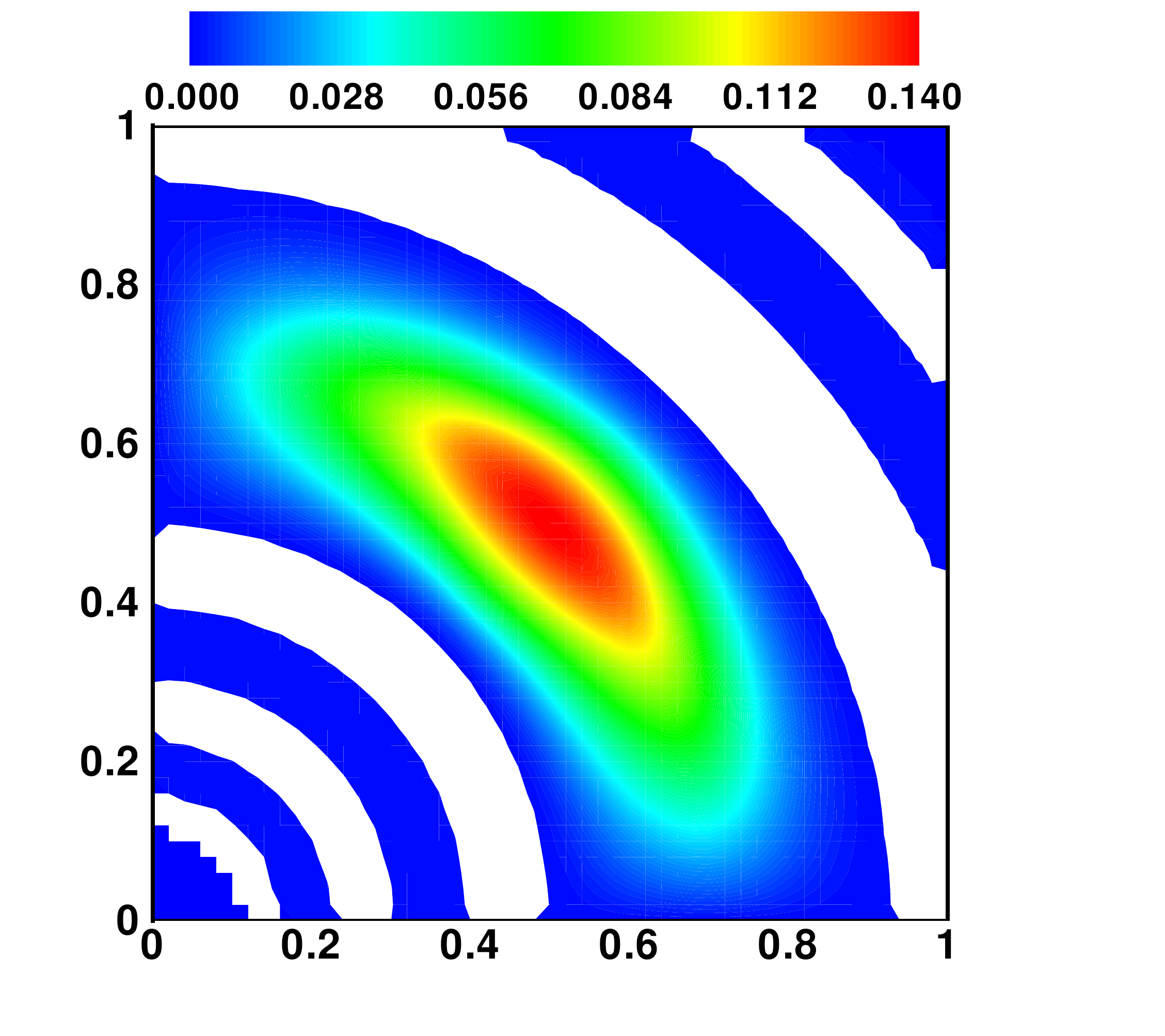}}
  \subfigure{
    \includegraphics[clip,scale=0.31]{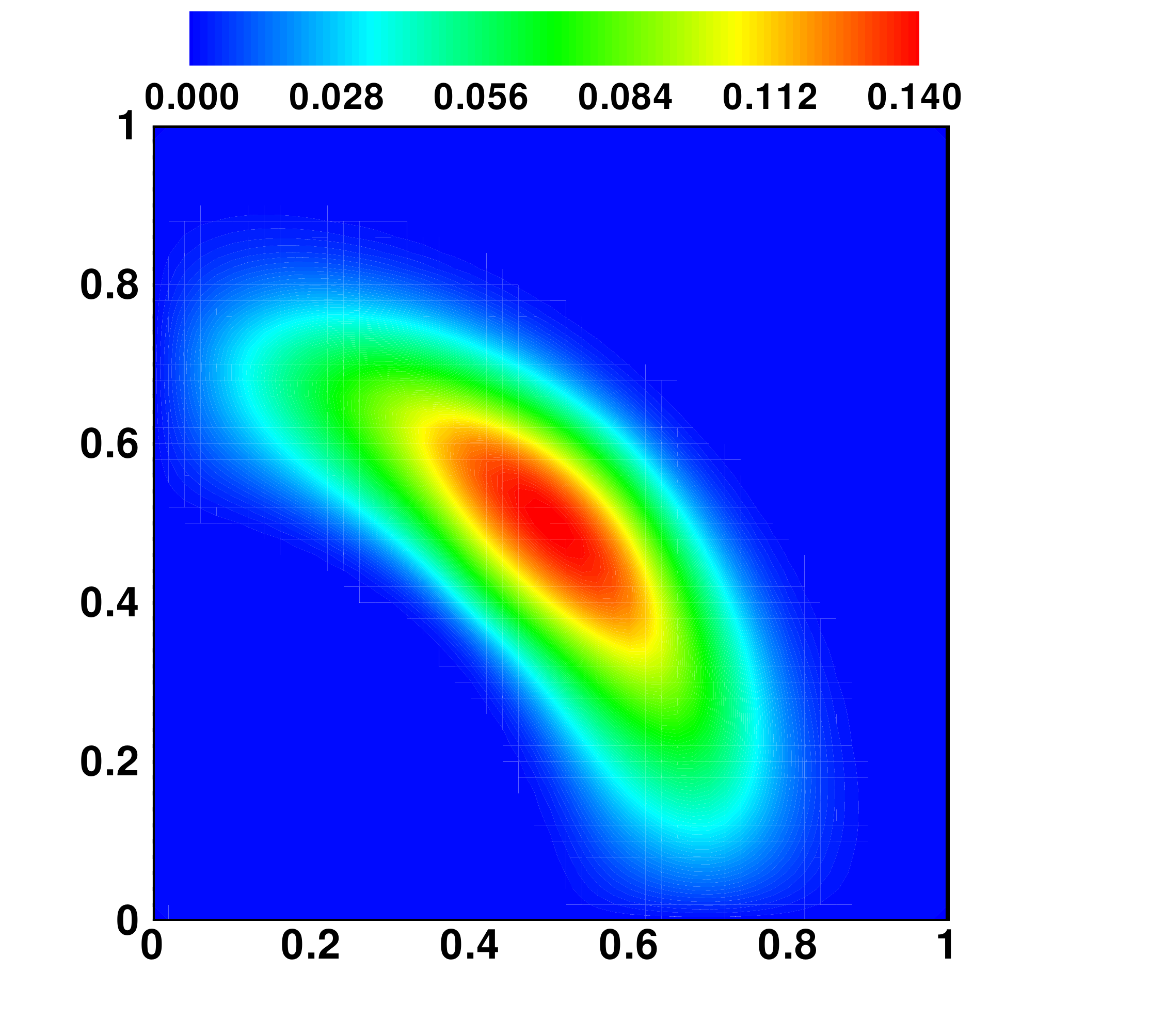}}
  \caption{Diffusion in heterogeneous anisotropic medium: 
    This figure shows the contours of the concentration under 
    the Galerkin single-field formulation (left) and the proposed 
    methodology (right) at time = 0.5 (top), time = 1.0 (middle) 
    and time = 2 (bottom). The time step is taken as $\Delta t 
    = 0.5$, and $\mathrm{XSeed} = \mathrm{YSeed} 
    = 51$. The number of nodes along the x-direction and 
    y-direction are, respectively, denoted by XSeed and 
    YSeed. The regions that have violated the non-negative 
    constraint are indicated in white color. 
    \label{Fig:TransientDMP_2D_heterogeneous_contours}}
\end{figure}

\begin{figure}[htp]
  \centering
    \includegraphics[clip,scale=0.5]{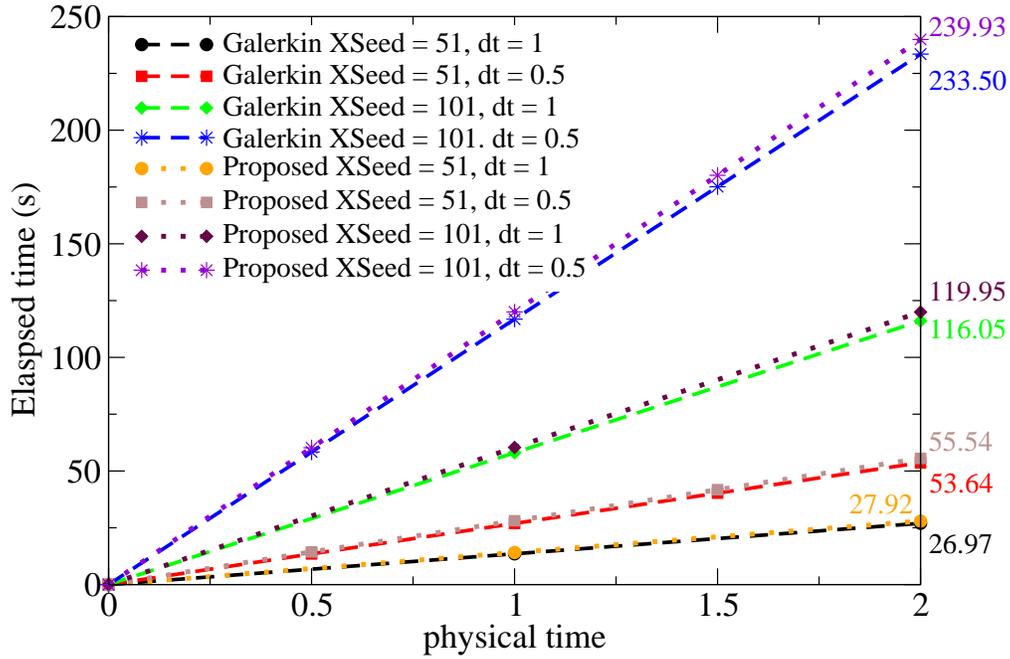}
  \caption{Diffusion in heterogeneous anisotropic medium: This 
    figure shows the elapsed time (i.e., wall clock time) taken by 
    the Galerkin single-field formulation and the proposed methodology 
    for various time steps and for various meshes. The x-axis indicates the 
    physical time, and the y-axis shows the wall clock time required 
    to reach various time levels. The numerical simulations 
    are carried using MATLAB R2012a \cite{MATLAB_2012a} on 
    Ubuntu Linux 12.04 LTS Operating System. The elapsed times 
    are obtained using tic-toc, which is a MATLAB's built-in 
    feature. 
    \label{Fig:TransientDMP_2D_heterogeneous_Elapsed_time}}
\end{figure}

\end{document}